\documentclass{amsart}
\usepackage{amsmath}
\usepackage{amssymb}
\usepackage{amscd}
\usepackage{latexsym}
\usepackage{euscript}
\usepackage{xy}\xyoption{all}

\theoremstyle{plain}
\newtheorem{theorem}{Theorem}[section]
\newtheorem{lemma}[theorem]{Lemma}
\newtheorem{proposition}[theorem]{Proposition}
\newtheorem{corollary}[theorem]{Corollary}

\theoremstyle{definition}
\newtheorem{definition}[theorem]{Definition}
\newtheorem{example}[theorem]{Example}
\newtheorem{examples}[theorem]{Examples}
\newtheorem{exercise}[theorem]{Exercise}
\newtheorem{remark}[theorem]{Remark}
\newtheorem{remarks}[theorem]{Remarks}
\newtheorem{rev}[theorem]{Review}
\numberwithin{equation}{section}

\newcommand\sm{\smallskip}
\newcommand\ms{\medskip}

\newcommand\ot{\otimes}
\newcommand\tplus{\textstyle \bigoplus}
\newcommand\pa{\partial}
\newcommand\inpr{(\cdot|\cdot)} 
\newcommand\an{^{\rm an}}

\newcommand\Iff{\Leftrightarrow}
\newcommand\lan{\langle} \newcommand\ran{\rangle}

\newcommand\bull{\bullet}
\newcommand\tsum{\textstyle \sum}
\newcommand\ti{^\times}

\newcommand\ch{\sp{\scriptscriptstyle\vee}}

\newcommand\re{^{\rm re}}

\newcommand\ind{_{\rm ind}}
\renewcommand\div{_{\rm div}}
\newcommand\sh{_{\rm sh}}
\renewcommand\lg{_{\rm lg}}
\newcommand\ts{\textstyle}

\newcommand\gr{{\rm gr}}

\DeclareMathOperator{\ad}{ad} 

\DeclareMathOperator{\Cent}{Cent} 
\DeclareMathOperator{\cent}{Cent} 
\DeclareMathOperator{\grCent}{grCent} 
\DeclareMathOperator{\CDer}{CDer} 
\DeclareMathOperator{\Der}{Der} 
\DeclareMathOperator{\grDer}{grDer} 
 \DeclareMathOperator{\diff}{diff}
\DeclareMathOperator{\rmD}{D} 
\DeclareMathOperator{\rmE}{E} \DeclareMathOperator{\End}{End}
\DeclareMathOperator{\ev}{ev} \DeclareMathOperator{\grEnd}{grEnd}
 
\DeclareMathOperator{\Hom}{Hom} \DeclareMathOperator{\grHom}{grHom}
\DeclareMathOperator{\Id}{Id} \DeclareMathOperator{\inc}{inc}
\DeclareMathOperator{\IDer}{IDer} 
\DeclareMathOperator{\IF}{IF} 
\DeclareMathOperator{\Ker}{Ker} \DeclareMathOperator{\pr}{pr}
\DeclareMathOperator{\rank}{rank} 
\DeclareMathOperator{\Rea}{Re} \DeclareMathOperator{\Span}{span}
\DeclareMathOperator{\SDer}{SDer}
\DeclareMathOperator{\grSDer}{grSDer}
\DeclareMathOperator{\supp}{supp}
\DeclareMathOperator{\SCDer}{SCDer}
\DeclareMathOperator{\tr}{tr} 

 \DeclareMathOperator{\rma}{A}
 \DeclareMathOperator{\rmb}{B}
 \DeclareMathOperator{\rmbc}{BC}
 \DeclareMathOperator{\rmc}{C}
 \DeclareMathOperator{\rmd}{D}
 \DeclareMathOperator{\rme}{E}
 \DeclareMathOperator{\rmf}{F}
 \DeclareMathOperator{\rmg}{G}

\newcommand\bq{{\mathbf q}}

\newcommand\CC{\mathbb{C}}
\newcommand\FF{\mathbb{F}}
\newcommand\NN{\mathbb{N}}
\newcommand\QQ{\mathbb{Q}}
\newcommand\RR{\mathbb{R}}
\newcommand\ZZ{\mathbb{Z}}

\newcommand\scA{\mathcal{A}}

\newcommand\scL{\mathcal{L}}

\newcommand\scQ{\EuScript{Q}} 

\newcommand\scV{\mathcal{V}}

\newcommand{\euC}{\EuScript{C}}
\newcommand{\euD}{\EuScript{D}}

\newcommand\frL{\mathfrak{L}}

\newcommand\fru{\mathfrak{u}}

\newcommand\uce{\mathfrak{uce}}

\newcommand\frg{\ensuremath{\mathfrak{g}}} \newcommand\g{\frg}
\newcommand\frh{\ensuremath{\mathfrak{h}}}

\newcommand\lsl{\ensuremath{\mathfrak{sl}}}
\newcommand\gl{\ensuremath{\mathfrak{gl}}}

\newcommand\al{\alpha}
\newcommand\be{\beta}
\newcommand\ga{\gamma} \newcommand\Ga{\Gamma}
\newcommand\de{\delta} \newcommand\De{\Delta}
\newcommand\eps{\epsilon} \newcommand\veps{\varepsilon}

\newcommand\ka{\kappa}
\newcommand\la{\lambda} \newcommand\La{\Lambda}
 \newcommand\vphi{\varphi}
\newcommand\rh{\rho}
\newcommand\si{\sigma} \newcommand\boldsi{{\boldsymbol{\sigma}}}

\newcommand\ta{\tau}
\newcommand\thet{\theta} 
  
\newcommand\ze{\zeta}
 \newcommand\Om{\Omega}

\begin{document}

\title{Lectures on extended affine Lie algebras}
\author{Erhard Neher}
\address{Department of Mathematics and Statistics \\
University of Ottawa\\ Ottawa, Ontario, K1N 6N5, Canada}
\email{neher@uottawa.ca\\}
\thanks{This work is partially supported by the Natural
Sciences and Engineering Research Council (NSERC) of Canada through
the author's Discovery Grant.}
\date{\today}


\begin{abstract} We give an introduction to the structure theory of extended
affine Lie algebras, which provide a common framework for
finite-dimensional semisimple, affine and toroidal Lie algebras. The
notes are based on a lecture series given during the Fields
Institute summer school at the University of Ottawa in June 2009.
\end{abstract}

\maketitle

\tableofcontents

\section*{Introduction}

Extended affine Lie algebras form a category of Lie algebras
containing finite-dimensional semisimple, affine, toroidal and some
other interesting classes of Lie algebras. \sm

Like finite-dimensional simple Lie algebras, extended affine Lie
algebras are defined by a set of axioms prescribing their internal
structure, rather than a potentially elusive presentation. The
structure of extended affine Lie algebras is now well understood,
and is quite similar to the construction of affine Lie algebras:
They are obtained from a generalized loop algebra, a so-called
invariant Lie torus, by taking a central extension and adding some
derivations:
\[ \xymatrix@C=80pt{
   { \begin{matrix}\hbox{central extension of $L$}\\
              \hbox{(another Lie torus)} \end{matrix} }
    \ar@{~>}[r]^{\txt{add \\ derivations}}
        & \hbox{extended affine Lie algebra}  \\
     \hbox{invariant Lie torus $L$} \ar@{~>}[u]
}\] Invariant Lie tori have been classified. Although there are some
rather sophisticated examples, many of them have a concrete matrix
realization or can be described in terms of familiar objects like
finite-dimensional simple Lie algebras and Laurent polynomial rings.
This makes extended affine Lie algebras easily accessible. Since
they are an emerging new area, there are many open questions,
opportunities for research and applications, for example in physics.
A short history of extended affine Lie algebras is given in section
\ref{n:sec:eala-def}, in particular it describes the role physicists
have played. \sm

The goal of these notes is to provide a survey of the structure
theory of extended affine Lie algebras, accessible to graduate
students. The emphasis is on examples, and not on an exposition
containing all proofs. Such an exposition will appear elsewhere.
Thus, while we have endeavored to present a complete picture of the
theory by giving precise definitions and theorems, most of the
proofs have been left out. But references to proofs are provided, as
far as possible. \sm

\emph{Outline.} Section \ref{n:ch:affgen} reviews the construction
of affine Kac-Moody algebras and discusses some natural
generalizations, like toroidal algebras. It also contains an
exposition of central extensions of Lie algebras, which are crucial
for the theory. The following section \ref{n:ch:eala-def-exam}
starts with the definition of an extended affine Lie algebra and
then presents some easily proven properties. We also give examples
of extended affine Lie algebras: finite-dimensional split simple,
affine Kac-Moody and untwisted multi-loop algebras. Part of the
axioms for an extended affine Lie algebra is the existence of a root
space decomposition. Section \ref{n:sec:roots} describes the
structure of the roots occurring in an extended affine Lie algebra,
naturally called extended affine root systems. They turn out to be
special types of so-called affine reflection systems. In section
\ref{n:sec:core} we reverse the picture above: We start with an
extended affine Lie algebra and, using the structure of affine
reflection systems, we associate to it a graded ideal, the so-called
core, and its central quotient, the centreless core. Both are Lie
tori. This section also presents properties of Lie tori and
examples. Finally, in section \ref{n:sec:const} we survey the
general construction of extended affine Lie algebras, as summarized
in the picture above. \sm

\emph{Prerequisites.} We assume that the reader is familiar with the
basic structure theory of complex finite-dimensional semisimple Lie
algebras, as for example developed in \cite{hum}. Some familiarity
with affine Kac-Moody algebras, {e.g.} chapters 7 and 8 of
\cite{kac}, is helpful but not essential, since section
\ref{n:sec:aff} will give a short review of the necessary
background. Similarly, knowing split simple Lie algebras will
facilitate reading the notes, but is not required. A short summary
of the facts used here is presented in section \ref{n:sec:eala0}.
\sm

\emph{Notation and setting.} With some rare exceptions (in
\ref{n:sec:aff}, \ref{n:sec:degree} and \ref{n:sec:centr}), all
vector spaces and algebras are defined over a field $F$ of
characteristic $0$. \textit{We will not assume that $F$ is
algebraically closed\/}, since this is not needed and would not do
proper justice to the theory to be explained here. Thus, $F$ could
be, but need not be the field $\CC$ of complex numbers or the field
$\RR$ of real numbers or the field of rational numbers $\QQ$ or ...
Unless specified otherwise, linear maps will always be $F$-linear.
All unadorned tensor products will be over $F$.

The symbol $\g$ will always denote a split simple finite-dimensional
Lie algebra. We let $Z(L) = \{ z\in L : [z,L]=0\}$ denote the centre
of a Lie algebra $L$. We will say that $L$ is \textit{centreless} if
$Z(L)=0$. If $K$ is a subspace of a Lie algebra $E$, the
\emph{centralizer\/} of $K$ in $E$ is $C_E(K)= \{ c\in E :
[c,K]=0\}$.

With the exception of some remarks, all algebras will be associative
or Lie algebras. For an $F$-algebra $A$ we denote by $\Der_F(L)$ the
Lie algebra of all derivations of $L$ (recall that an $F$-linear map
$d: L \to L$ is a \emph{derivation\/} if $d([l_1,l_2])= [d(l_1),l_2]
+ [l_1, d(l_2)]$ holds for all $l_1, l_2 \in L$).

The algebras considered here will often be graded by some abelian
group, usually denoted $\La$ and always written additively. A
\textit{$\La$-grading} of a vector space $V$ by the abelian group
$\La$ is a decomposition $V= \bigoplus_{\la \in \La} V^\la$ into
subspaces $V^\la$. Suppose $V$ is such a $\La$-graded vector space.
Then the \emph{$\La$-support\/} of $V$ is defined as $\supp_\La V =
\{ \la \in \La : V^\la \ne 0\}$. A \textit{graded subspace of\/
$V$\/} is a subspace $U$ of $V$ satisfying $U= \bigoplus_\la (U \cap
V^\la)$. We will say that $V$ has \textit{finite bounded dimension}
if there exists a constant $M$ such that for all $\la \in \La$ we
have $\dim V^\la \le M$. Note that this is a stronger condition than
requiring that $V$ has \textit{finite homogeneous dimension}, which
by definition just means that every $V^\la$, $\la\in \La$, is
finite-dimensional.

Given two $\La$-graded vector spaces $V=\bigoplus_\la V^\la $ and
$W= \bigoplus_\la W^\la$, we say an $F$-linear map $f: V \to W$ has
\emph{degree $\la$\/} if $f(V^\mu) \subset W^{\la + \mu}$ holds for
all $\mu \in \La$. We denote by $\Hom_F(V,W)^\la$ the linear maps of
degree $\la$ and put \[  \grHom_F(V,W) = \ts \bigoplus_{\la \in \La}
\Hom_F(V,W)^\la \quad \hbox{and} \quad \grEnd_F(V) =
\grHom_F(V,V).\] We note that $\grEnd_F(V)$ is a $\La$-graded
associative algebra with respect to composition of maps. We give $F$
the trivial grading $F=F^0$ and define the \textit{graded dual space
of $V$\/} as \[ V^{{\rm gr}*} = \grHom_F(V,F) = \ts \bigoplus_{\la
\in \La}(V^{{\rm gr}*})^\la.
\]
Observe that $(V^{\gr *})^\la$ consists of those linear forms $\vphi
: V \to F$ which satisfy $\vphi(V^\mu) = 0$ whenever $\la + \mu \ne
0$ and can therefore be identified with the usual dual space
$(V^{-\la})^*$.

Given a symmetric bilinear form on a vector space $V$, an
endomorphism $d$ of $V$ is called \emph{skew-symmetric} if $(d(v)
\mid v)=0$ for all $v\in V$. Since we assume that our base field has
characteristic $0$, this is equivalent to the condition $(d(v_1)
\mid v_2)+ (v_1 \mid d(v_2))=0$ for all $v_1, v_2 \in V$. A bilinear
form is \emph{nondegenerate\/} if $(v\mid u)=0$ for all $u\in V$
implies $v=0$.

If $A$ is an algebra, a \textit{$\La$-grading of the algebra $A$\/}
is a $\La$-grading of the underlying vector space $A$, say $A =
\bigoplus_{\la \in \La} A^\la$, for which in addition $A^\la A^\mu
\subset A^{\la + \mu}$ holds for all $\la,\mu \in \La$. Since we
will often deal with algebras with two gradings, it is convenient to
use superscripts and subscripts to distinguish them. \ms

These notes grew out of my notes for a lecture series during the
Fields Institute summer school on Geometric Representation Theory
and Extended Affine Lie Algebras, held at the University of Ottawa
in June 2009. I would like to thank all the participants of the
summer school for their interest and questions. I also thank Bruce
Allison and Juana S\'anchez Ortega for their careful reading of an
earlier version of these notes.


%
%
\section{Affine Lie algebras and some generalizations}
\label{n:ch:affgen}

We will always assume that $F$ is a field of characteristic $0$.
Occasionally we will need some roots of unity in $F$, so certainly
an algebraically closed field like $\CC$ will do.

We denote by $\g$ a split simple finite-dimensional Lie algebra over
$F$. For example, if $F$ is algebraically closed then this just
means that $\g$ is a simple and finite-dimensional. Their structure
theory is explained in most standard textbooks, for example in
\cite{hum}. For more general fields, an example of a split simple
$\g$ is the Lie algebra $\lsl_n(F)$ of $n\times n$-matrices over $F$
which have trace $0$. These types of Lie algebra are investigated in
\cite[Ch.~VII]{bou:lie78}, \cite[Ch.~1]{Dix} or \cite[Ch.~IV]{jake}.

\subsection{Realization (construction of affine Kac-Moody Lie algebras)}
\label{n:sec:aff}

Let $\ze \in F$ be a primitive $m$th root of $1$. In other words,
the multiplicative subgroup of $F$ generated by $\ze$ is isomorphic
to $\ZZ/ m \ZZ$. For example, in $F=\CC$ we can take $\ze =
\exp(2\pi i/m)$.

Let $\si$ be an automorphism of $\g$ of finite order $m\in \NN$.
Thus, the subgroup $\langle \si\rangle$ of the automorphism group of
$\g$ is isomorphic to $\ZZ/m\ZZ$. For example, if $\g=\lsl_n(F)$ an
example of such an automorphism is $\si(x) = a x a^{-1}$, where $a$
is an $n\times n$-matrix of order $m$, and an example of such a
matrix is $a=\ze E_n$ where $E_n$ is the $n\times n$ identity
matrix.

Observe that $\si$ is diagonalizable. Indeed, its minimal polynomial
divides the polynomial $t^m=1$ and therefore has no multiple roots
in $F$. For a general field $F$ this would of course only say that
$\si$ is a semisimple endomorphism. But since as we assumed that $F$
contains all roots of unity which we need, $\si$ is diagonalizable
over $F$. To describe its eigenspaces we need some notation. In
anticipation of the later developments we put
$$
   \La = \ZZ \quad \hbox{and} \quad \bar \La = \ZZ / m\ZZ,
$$
and denote the canonical map $\La \to \bar \La$ by $\la \mapsto \bar
\la$. That $\si$ is diagonalizable, means \begin{equation}
\label{n:eq:aff1}
  \g = \textstyle \bigoplus_{\bar \la \in \bar \La} \g_{\bar \la}
 \quad\hbox{for } \g_{\bar \la} = \{ x\in \g : \si(x) = \ze^\la x\}
\end{equation} Of course, some of the $\g_{\bar \la}$ could be zero. The
eigenspaces of $\si$ are precisely the non-zero among the subspaces
$\g_{\bar \la}$. It is also appropriate to note that $\g_{\bar \la}$
is well-defined: if $\bar \la = \bar \mu$ then $\ze^\la = \ze^\mu$.
Finally we point out that the decomposition (\ref{n:eq:aff1}) is a
$\bar \La$-grading, which means that it satisfies
\begin{equation} \label{n:eq:aff1.5}
[\g_{\bar \la}, \g_{\bar \mu}] \subset \g_{\bar \la + \bar \mu}
\quad \hbox{for all $\bar \la, \bar \mu \in \bar \La$.}
\end{equation}

Let $F[t^{\pm 1}]$ be the ring of Laurent polynomials. This is a
unital associative commutative $F$-algebra with $F$-basis $\{t^\la :
\la \in \ZZ\}$ and multiplication rule $t^\la t^\mu = t^{\la +\mu}$.

The \textit{loop algebra} associated to the data $(\g, \si)$ is the
Lie algebra \begin{equation} \label{n:oneloop} \scL= L(\g, \si) =
\tplus_{\la \in \La} \g_{\bar \la} \ot F t^\la \end{equation} with
product
\begin{equation} \label{n:eq:aff2}
   [u_{\bar \la} \ot t^\la, \, v_{\bar \mu} \ot t^\mu] =
   [u_{\bar \la}, v_{\bar \mu}] \ot t^{\la + \mu}.
\end{equation}
We will sometimes use more precise terminology: If $\si = \Id$,
i.e., $m=1$, we will call $L(\g, \Id) = \g \ot F[t^{\pm 1}]$ the
\textit{untwisted loop algebra}, and we will call $L(\g,\si)$ a
\textit{twisted loop algebra} if it is clear that $\si\ne \Id$ and
we want to emphasize this.

We point out that we consider $L(\g,\si)$ as a Lie algebra over $F$.
It is therefore infinite-dimensional. It is also important to note
that $\scL$ is a $\La$-graded algebra, whose homogenous spaces are
$\scL^\la = \g_{\bar \la} \ot Ft^\la$ for $\la \in \La$. For the
reader with some background in algebraic geometry, a more geometric
definition of $L(\g, \si)$ is the following: It is (isomorphic to)
the Lie algebra of equivariant maps $F^\times \to \g$, where $\si$
acts on $F^\times$ by $\si(x) = \ze x$.) \sm

Let $\ka$ be the Killing form of $\g$, i.e., $\ka(u,v) = \tr(\ad u
\circ \ad v)$, and define \begin{equation}\label{n:eq:aff2.5} \psi :
\scL \times \scL \to F, \quad
     \psi (u\ot t^\la, v \ot t^\mu) = \la \,\de_{\la, -\mu} \, \ka(u,v)
\end{equation} where $\de_{\la, -\mu}$ is the Kronecker delta: It has the
value $1$ if $\la = -\mu$ and is zero otherwise.

\begin{exercise}\label{n:ex:aff1}
Check that the map $\psi$ of (\ref{n:eq:aff2.5}) is a {\it
$2$-cocycle\/} of $\scL$, i.e.~an $F$-bilinear map satisfying
\begin{equation} \label{n:eq:aff3}
 \psi(l,l) = 0 = \psi([l_1,l_2], l_3) + \psi([l_2, l_3], l_1) +
\psi([l_3,l_1], l_2)
\end{equation}
for $l, l_i \in \scL$.
\end{exercise}
A consequence of Exercise~\ref{n:ex:aff1} is that we can enlarge our
Lie algebra $\scL$ by adjoining a $1$-dimensional space, denoted
$Fc$ here: \begin{equation} \label{n:aff3.2}
  \tilde \scL = \tilde \scL(\g, \si) = L(\g, \si) \oplus Fc
\end{equation} is a Lie algebra over $F$ with respect to the product
$$
   [l_1 \oplus s_1 c, \, l_2 \oplus s_2 c]_{\tilde \scL} = [l_1, l_2]_L
         \oplus \psi(l_1, l_2) c
$$
for $l_i \in \scL$ and $s_i \in F$. We have added subscripts on the
products to emphasize where the product is calculated, in $\tilde
\scL$ or in $\scL$. It is obvious from the product formula, that it
is important to know in which Lie algebra the product is being
calculated. But in the future we will leave out the subscripts, if
it is clear in which algebra the product is calculated.

The equations (\ref{n:eq:aff3}) are exactly what is needed to make
$\tilde \scL$ a Lie algebra. The map $$  \fru : \tilde \scL \to
\scL, \quad \fru(l \oplus sc) = l
$$
is a surjective Lie algebra homomorphism with kernel $\Ker (\fru) =
Fc = Z(\tilde \scL)$, the centre of $\tilde \scL$. In other words,
$\fru$ is a {\it central extension\/} (see \ref{n:appcen} for a
short review of central extensions). In fact, $\fru $ is the
``biggest'' central extension, the so-called {\it universal central
extension\/}, see \cite{gar} and \cite{wilson} for a proof.\sm

The Lie algebra $\tilde \scL$ has a canonical derivation $d$, the
so-called \textit{degree derivation} \begin{equation}
\label{n:eq:aff3.5}
 d\big( (u \ot t^\la) \oplus sc\big) = \la u \ot t^\la, \quad (\la \in \ZZ,
            u \in \g_{\bar \la}, s \in F).
\end{equation}  Hence we can form the semidirect product $
 \hat \scL = L( \g, \si)\hat{\;} = \tilde \scL \rtimes F d
$ with product
$$
 [\tilde l_1 \oplus s_1 d, \, \tilde l_2 \oplus s_2 d]_{\hat \scL}
    =  [\tilde l_1, \tilde l_2]_{\tilde \scL} + s_1 d(\tilde l_2)
       - s_2 d(\tilde l_1)
$$
for $\tilde l_i \in \tilde \scL$ and $s_i \in F$. In untangled form,
\begin{equation} \label{n:eq:aff8}
 \hat \scL = \big( \tplus_{\la \in \ZZ} (\g_{\bar \la} \ot F t^\la)
\big) \, \oplus \, F c \, \oplus \, Fd \end{equation} is the Lie
algebra with product \begin{multline} \label{n:eq:aff9}
   [u_{\bar \la} \ot t^\la \oplus s_1 c \oplus s'_1 d , \,
        v_{\bar \mu} \ot t^\mu \oplus s_2 c \oplus s'_2 d]
  \\ = \big( [u_{\bar \la}, v_{\bar \mu}] \ot t^{\la + \mu} +
      \mu s_1' v_{\bar \mu} \ot t^\mu
          - \la s_2' u_{\bar \la} \ot t^\la \big)
  \oplus \la \, \de_{\la, - \mu} \,\ka(u_{\bar \la}, v_{\bar \mu}) \,c.
\end{multline}

\begin{exercise} \label{n:ueb1} Show $[\hat \scL, \hat \scL] = \tilde \scL$ and
$Z(\tilde\scL) = Fc= Z(\hat \scL)$.
\end{exercise}

The importance of the Lie algebras $L(\g,\si)\hat{\;\;}$ stems from
the following. \ms

\begin{theorem} \label{n:th-kac}
{\rm (\textbf{Realization Theorem} \cite[Th.~7.4, Th.~8.3,
Th.~8.5]{kac})} Suppose $F$ is algebraically closed. \sm

{\rm (a)} The Lie algebra $L(\g, \si)\hat{\;\;}$ is an affine
Kac-Moody Lie algebra, and every affine Kac-Moody Lie algebra is
isomorphic (as $F$-algebra) to some $L (\g, \si)\hat{\;\,}$. \sm

{\rm (b)} $L (\g, \si)\hat{\;\;} \cong L(\g, \si')\hat{\;\,}$ where
$\si'$ is a diagram automorphism with respect to some Cartan
subalgebra of $\g$. \end{theorem} \sm

We note that diagram automorphisms have order $1, 2$ or $3$, with
the latter case only occurring for $\g$ of type $\rmd_4$.

\subsection{Multiloop and toroidal Lie algebras}\label{n:sec:toroidal}

We will discuss some (straightforward) generalizations of
$\scL=L(\g,\si)$, the central extension $\tilde \scL$ and the big
Lie algebra $\hat \scL$. \ms

The first idea is to replace the Laurent polynomial ring $F[t^{\pm
1}]$ by a ring with similar properties. Instead of one variable we
will use the Laurent polynomial ring $F[t_1^{\pm 1}, \ldots,
t_n^{\pm 1}]$ in $n$ variables. This ring has indeed very similar
properties to the ring $F[t^{\pm 1}]$. We put $\La = \ZZ^n$ and
define
$$
   t^\la = t_1^{\la_1} \cdots t_n^{\la_n} \quad \hbox{for }
   \la=(\la_1, \ldots, \la_n) \in \La
$$
Then $\{ t^\la : \la \in \La\}$ is an $F$-basis of $F[t_1^{\pm 1},
\ldots, t_n^{\pm 1}]$ and the multiplication rule in $F[t_1^{\pm 1},
\ldots, t_n^{\pm 1}]$ is $t^\la t^\mu = t^{\la + \mu}$, which is the
``same'' as in the $1$-variable case. Also, $F[t_1^{\pm 1}, \ldots,
t_n^{\pm 1}]$ is still a unital commutative associative $F$-algebra.
We can therefore define the \textit{untwisted multiloop algebra},
the ``several variable'' generalization of the untwisted loop
algebra of \ref{n:sec:aff} as \begin{equation}  \label{n:untloo}
L(\g) =  \g \ot F[t_1^{\pm 1}, \ldots, t_n^{\pm 1}], \end{equation}
which becomes a Lie algebra with respect to the product $$[u\ot
t^\la, v \ot t^\mu] = [u,v] \ot t^{\la + \mu}$$ for $u,v\in \g$ and
$\la, \mu \in \ZZ^n$. We will meet this Lie algebra again in
Example~\ref{n:lieuntwist}.

To continue the analogy we let $\boldsi=(\si_1, \ldots, \si_n)$ be a
family of $n$ commuting finite order automorphisms of $\g$, say
$\si_i$ has order $m_i \in \NN_+$. Let $\ze_i \in F$ be a primitive
$m_i$-th root of $1$ (recall that we assumed that $F$ has an ample
supply of them). We put
$$
 \bar \La = (\ZZ/ m_1 \ZZ) \oplus \cdots \oplus (\ZZ/ m_n \ZZ)$$ and
let $\la \mapsto \bar \la$ be the obvious map. The automorphisms
$\si_i$ are simultaneously diagonalizable:
\begin{equation}\label{n:eq:aff4}
  \g = \tplus_{\bar \la \in \bar \La} \g_{\bar \la}, \quad
    \g_{\bar \la} = \{ u \in \g: \si_i(u) = \ze_i^{\la_i} u,
    1 \le i \le n\}.
\end{equation}
As in the one-variable case, the decomposition (\ref{n:eq:aff4}) is
a $\bar \La$-grading: $[\g_{\bar \la}, \g_{\bar \mu}] \subset
\g_{\bar \la + \bar \mu}$ for $\bar \la, \bar \mu \in \bar \La$. It
follows from this that \begin{equation} \label{n:multdeff}
  L(\g, \boldsi ) = \tplus_{\la \in \La} \, \g_{\bar \la} \ot Ft^\la
\end{equation} is a subalgebra of $\g \ot F[t_1^{\pm 1}, \ldots, t_n^{\pm
1}]$, called the \textit{multiloop algebra associated to $\g$ and
$\boldsi$}. If all $\si_i= \Id_\g$ we will (of course) call it an
\textit{untwisted multiloop algebra}. Multiloop algebras are
investigated in the papers \cite{abfp}, \cite{abfp2}, \cite{abp1},
\cite{abp2} and \cite{abp2.5}. \sm

Following our procedure in section \ref{n:sec:aff} we should now
make a central extension to get a bigger Lie algebra $\tilde \scL$
and then add some derivations:
\begin{equation} \label{n:eq:aff4.5}
\xymatrix@C=80pt@R=40pt{ \tilde \scL= \scL \oplus C \;
 \ar@{^{(}->}[r]^{\txt{add \\ derivations}}
         \ar[d]_{\txt{central \\ extension}}
 &  \hat \scL = \tilde \scL \rtimes D \\
 \scL =L(\g,\boldsi) \ar@{~>}[ur]}
\end{equation} To define the Lie algebra product on $\tilde \scL $ we would
use a $2$-cocycle $\psi : \scL \times \scL \to C$ where $C$ is some
vector space and then put
\begin{equation} \label{n:eq:aff5}
    [l_1 \ot c_1, l_1\ot c_2]_{\tilde \scL}
         = [l_1, l_2]_\scL \oplus \psi(l_1, l_2)
\end{equation} for $l_i \in \scL$ and $c_i \in C$. The Lie algebra $\hat
\scL$ should be a semidirect product with $D$ acting on $\tilde
\scL$ by derivations. \sm

But here is where the problems start, or things become interesting
depending on one's taste. In the one-variable case the $2$-cocycle
$\psi$ of (\ref{n:eq:aff2.5}) was the only possible choice up to
scalars, i.e., the universal central extension $\tilde \scL$ of
$\scL$ had a $1$-dimensional centre $C=Fc$. This is no longer true
in the case of several variables. It is not so surprising that there
exists a $2$-cocycle with values in $F^n$: We can simply use the
same formula as in (\ref{n:eq:aff2.5}).

\begin{exercise} \label{n:ueb3} Let $\scL= L(\g, \boldsi)$ be a
multiloop algebra and embed $\La \subset F^n$ canonically. Then
$\psi: \scL \times \scL \to F^n$, given by \begin{equation}
\label{n:eq:aff6} \psi(u \ot t^\la, v\ot t^\mu) =  \de_{\la +\mu, 0}
\, \ka(u,v) \,\la \,, \end{equation} is a $2$-cocycle of $\scL$.
\end{exercise}

However, this is still not the ``biggest'' possible. Rather, the
centre of the universal central extension is infinite-dimensional
and the so-called \textit{universal $2$-cocycle}, i.e., the
$2$-cocycle used in (\ref{n:eq:aff5}) to describe the universal
central extension $\hat \scL$ of $\scL$, is described in the
following result. \ms

\begin{theorem}[{\cite{n:uce}}] \label{n:ucemm}  Let $\scL=L(\g,\boldsi)$
be a multiloop algebra. We embed $\La\subset F^n$ canonically, put
$\Ga = m_1 \ZZ \oplus \cdots \oplus m_n \ZZ$ and let
$C=\bigoplus_{\ga\in \Ga} C_\ga$ where $C_\ga = F^n / F\ga$. Then
the universal $2$-cocycle is $\psi_\fru : \scL \times \scL \to C$
for which the $\ga$-component of $\psi_\fru$ is \begin{equation}
\label{n:eq:aff7}   \psi_\fru( u \ot t^\la, y \ot t^\mu)_\ga =
\ka(u,v) \de_{\la + \mu, -\ga} \, \bar \la \in C_\ga.\end{equation}
\end{theorem}

Observe that (\ref{n:eq:aff6}) is just the $0$-component of
(\ref{n:eq:aff7}). The theorem is well-known in the untwisted case
(all $\si_i = \Id_\g$, so $\Ga=\La$), in which it can be deduced
from the description of the universal central extension of the Lie
algebra $\g \ot A$ where $A$ is any unital commutative associative
$F$-algebra, see \cite{Kas} and \cite{Moo-Rao-Yok}. (In these
references the centre $C$ of the universal central extension is
described as $\Om_A /dA$ where $\Om_A$ is the module of K\"ahler
differentials, which is also the same as the first cyclic homology
group ${\rm HC}_1(A)$.)

In the untwisted case, the universal central extension $\hat \scL$
was termed the \textit{$n$-toroidal Lie algebra} based on $\g$. The
reader should however be warned that this terminology is not
standard. It is sometimes used for the Lie algebra $\tilde \scL$
with the $2$-cocycle of exercise~\ref{n:ueb3}, and sometimes also
for the Lie algebras of the form $\hat \scL = \scL \oplus C \oplus
D$ for an appropriate subalgebra $D$ of derivations, e.g.~in
\cite{DiFuPe}.

Thus, there are many possibilities for $C$ in the diagram
(\ref{n:eq:aff4.5}), and it is not clear which one is the best
possible choice. (In fact, we will later allow any central
extension). \sm

Assuming that we have settled for some $C$, which $D$ should we
take? For simplicity we will discuss this only in the untwisted
case. If $n=1$ we added the degree derivation $d$ described in
(\ref{n:eq:aff3.5}). This is far from being an arbitrary derivation.
The full derivation algebra of the Lie algebra $\g \ot A$ for $A$ is
described in \cite[Th.~1]{bemo}: \begin{align}
 \Der_F(\g\ot A) &=
    \big( \Der_F(\g) \ot A \big) \oplus \big( F\Id \ot \Der_F(A)\big)
  \nonumber \\
  &= \quad \IDer(\g \ot A) \oplus F\Id \ot \Der_F(A). \label{n:eq:aff11}
 \end{align}
where $\Der_F(\g) \ot A$ and $F \ot \Der_F(A)=F\Id \ot \Der_F(A)$
act on $\g \ot A$ in the obvious way.

Since $\g \ot A$ is perfect, up to a canonical isomorphism, this is
then also the derivation algebra of the universal central extension
of $\g \ot A$ (see for example \cite[Th.~2.2]{bemo}). From
$$\Der F[t^{\pm 1}] = F[t^{\pm 1}]d$$ we see that we added a rather
special derivation, one which can be used to define the
$\La$-grading of $\scL$ (see also Ex.~\ref{n:ex:derloop}).

We can do something similar in multi-variable case. Define the
$i$-th \textit{degree derivation} $\pa_i$ of $L(\g)\oplus C$ by
\begin{equation} \label{n:eq:aff10}
\pa_i ( u \ot t^\la \oplus c) = \la_i \, u \ot t^\la \quad \hbox{for
} \la= (\la_1, \ldots, \la_n) \in \La=\ZZ^n \end{equation} and put
$$
 \euD = \Span_F \{ \pa_i : 1 \le i \le n\},$$
the space of \textit{degree derivations}. Possible (interesting)
choices for $D$ are: \begin{enumerate}

 \item $D= \euD $, \sm

 \item $F[t_1^{\pm 1}, \ldots, t_n^{\pm 1}] \euD$ (in physics parlance:
``all vector fields''), and \sm

 \item $\bigoplus_{\la \in \La} Ft^\la \{ \sum_{i=1}^n s_i \pa_i :
    \sum_i s_i = 0 \}$ (the ``divergence $0$ vector fields'').
\end{enumerate}

It will turn out that for the Lie algebras which we are going to
study in the next chapters, the choices (1) and (3) are the correct
ones. In addition, there will be a surprise: semidirect products in
(\ref{n:eq:aff4.5}) will not be enough!

\begin{exercise} \label{n:ex:invfo}
 Recall that a bilinear form $\inpr$ on a Lie algebra $L$ is
called \emph{invariant} if $([l_1, l_2] \mid l_3) = (l_1 \mid [l_2,
l_3])$ holds for all $l_i \in L$. Show:


(a) The set $\IF(L)$ of invariant bilinear forms on $L$ is a vector
space with respect to the obvious scalar multiplication and addition
defined by $(\be_1 + \be_2)(l_1,l_2) = \be_1(l_1, l_2) + \be_2(l_1,
l_2)$ for $\be_i \in \IF(L)$. \sm


(b) If $L$ is perfect, any invariant bilinear form is symmetric. \sm

(c) Let $S$ be a unital associative $F$-algebra. A bilinear form $b$
on $S$ is called \emph{invariant} if $b(s_1 s_2 , s_3) = b(s_1,
s_2s_3) = b(s_2, s_3s_1)$ for $s_i \in S$. \begin{itemize}

\item[(i)] The set $\IF(S)$ of invariant bilinear forms on $S$ is a
vector space with respect to the obvious operations.

\item[(ii)] Any linear form $\la \in S^*$ with $\la([S, S])=0$ gives rise to an
invariant bilinear form $b_\la$ on $S$, defined by $b_\la(s_1, s_2)
= \la(s_1s_2)$.

\item[(iii)]  The map $(S/[S,S])^* \to \IF (S)$, given by $\la \mapsto b_\la$,
is a vector space isomorphism. \end{itemize}

(d) Let $L$ be a perfect Lie algebra with a $1$-dimensional space
$\IF(L)$, say $\IF(L) = F \ka$. Also, let $S$ be a unital
associative commutative $F$-algebra. We consider $L\ot S$ as Lie
algebra with respect to the product $[l_1 \ot s_1, l_2 \ot s_2] =
[l_1, l_2] \ot s_2s_2$, cf. (\ref{n:eq:aff2.5}). For $\la \in
\IF(S)$ define a  bilinear form $\ka \ot \la$ on $L \ot S$ by
\[(\ka\ot \la)\,(l_1 \ot s_2, \, l_2 \ot s_2)= \ka(l_1, l_2) \,
\la(s_1, s_2).\] Then $\ka \ot \la \in \IF(L\ot S)$ and the map
$\IF(S) \to \IF(L \ot S)$, given by $\la \mapsto \ka \ot \la$, is an
isomorphism of vector spaces.
 \end{exercise}

\begin{exercise} \label{n:ex:derloop} Define the i$^{\rm th}$ degree
derivation $\pa_i$ of the Laurent polynomial ring $S=F[t_1 ^{\pm 1},
\ldots, t_n^{\pm 1}]$ by $\pa_i(t^\la) = \la_i t^\la$, so that the
$\pa_i$ of (\ref{n:eq:aff10}) becomes $\pa_i(u \ot t^\la) = u \ot
\pa_i(t^\la) = (\Id \ot \pa_i)(u\ot t^\la)$ (this double meaning of
$\pa_i$ should not create any confusion). Show:

(a) The derivation algebra $\Der_F(S)$ of $S$ is given by \[
\Der_F(S) = S \euD =
   \ts \bigoplus_{\la \in \ZZ^n} Ft^\la \euD
 \]
where, as above, $\euD = \Span_F\{ \pa_i : 1\le i \le n\}$. The
derivation algebra is a $\ZZ^n$-graded Lie algebra with Lie algebra
product determined by
\[[t^\la \pa_i, t^\mu \pa_j] = t^{\la + \mu} (\mu_i \pa_j - \la_j
\pa_i).\] Thus,  for $n=1$ we obtain the usual \emph{Witt algebra},
see for example \cite[1.4]{mp}.

 (b) $(t^\la \mid t^\mu)= \de_{\la + \mu, 0}$ defines a
nondegenerate symmetric bilinear form $\inpr$ on $S$ which is
\emph{invariant} in the sense that $(ab | c) = (a \mid bc)$ for all
$a,b,c\in S$.

(c) Let $\SDer_F(S)$ be the subalgebra of derivations of $S$, which
are skew-sym\-metric with respect to the form $\inpr$ of (b). Then
\[\SDer_F(S) = \ts \bigoplus_{\la \in \ZZ^n} Ft^\la \big\{ \tsum_{i=1}^n s_i
\pa_i : \sum_i s_i \la_i = 0 \big\}.\] In particular, for $n=1$ we
get $\SDer_F(F[t^{\pm 1}]) = Fd$ for $d=\pa_1$.
\end{exercise}

\subsection{Appendix on central extensions of Lie algebras}\label{n:appcen}

Central extensions will turn out to be an important tool in the
construction of extended affine Lie algebras. Although this provides
one with a bigger and hence potentially more complicated Lie
algebra, central extensions turn up naturally in the general theory
and the biggest of them (the universal central extension) is in fact
quite ``nice". For example, universal central extensions often have
a simpler presentation and a much richer representation theory than
the original Lie algebra. In this appendix we review the necessary
background. \ms

\begin{definition}[Extensions] \label{n:defexten}
An \textit{extension of a Lie algebra $L$\/} is a surjective
homomorphism $f: K \to L$ of Lie algebras. A \textit{homomorphism}
from an extension $f: K \to L$ to another extension $f': K' \to L$
is a Lie algebra homomorphism $g: K \to K'$ satisfying $f = f' \circ
g$. In other words, the diagram below is commutative.
\begin{equation} \label{n:exthomdi} \xymatrix{ K \ar[rr]^g \ar[dr]_f
&& K' \ar[dl]^{f'}\cr
            & L }
\end{equation} We will use \emph{abelian extensions},
i.e., extensions $f : K \to L$ with $\Ker f$ an abelian ideal in the
construction of an extended affine Lie algebra in section
\ref{n:sec:genconstr}.
\end{definition}

\begin{definition}[central extensions] \label{n:ucedef} A \textit{central extension\/} of $L$ is an extension $f: K \to
L$ whose kernel $\Ker f$ is contained in the centre $Z(K)$ of $K$. A
central extension $f: K \to L$ is called a \textit{covering} if $K$
is perfect, {i.e.,} $K=[K,K]$. It is traditional (but not always
advisable) to not specify the morphism $f$ and simply say that
\textit{$K$ is a central extension of $L$} or a \textit{covering}.

A central extension $\fru : \frL \to L$ is called a
\textit{universal central extension\/} if there exists a unique
homomorphism from $\fru:  \frL \to L$ to any other central extension
$f:  K \to L$ of $L$. It is obvious from the universal property that
two universal central extensions of $L$ are isomorphic as central
extensions and hence in particular their underlying Lie algebras are
isomorphic. We denote the universal central extension of $L$ by
$\fru : \uce(L)\to L$ or simply $\uce(L)$. \end{definition}

\begin{theorem}[{\cite[Prop.~1.3]{vdK}, \cite[\S1]{gar}}] \label{n:sub:ucethm}
A Lie algebra $L$ has a universal central extension if and only if
$L$ is perfect. In this case, the universal central extension $\fru
: \uce(L)\to L$ is perfect too, i.e., $\fru$ is a covering.
\end{theorem}

The process of taking universal central extensions stops at
$\uce(L)$, due to the following equivalent conditions for a Lie
algebra $L$: \begin{itemize}

\item[(i)] $\Id : L \to L$ is a universal central extension, i.e.,
$\uce(L) = L$,

\item[(ii)] every central extension $f: K \to L$ is direct product
$K=\tilde L \times \Ker f$ such that $f|_{\tilde L}$ is an
isomorphism between $\tilde L$ and $L$. \end{itemize} If (i) and
(ii) hold, one calls $L$ \textit{centrally closed}.

\begin{examples} (a) It is an immediate corollary of the Levi-Malcev
Theorem that every finite-dimensional semisimple Lie algebra is
centrally closed (\cite[VII, \S6.8, Cor.~3]{bou:lie78} or
\cite[Cor.~7.9.5]{wei}).

(b) An example of a universal central extension is the Virasoro
algebra: It is the universal central extension of the Witt algebra
$\Der_F(F[t^{\pm 1}])$, see for example \cite[I.9, Prop.~4]{mp}.
Hence the Virasoro algebra is centrally closed, while
$\Der_F(F[t^{\pm 1}])$ is not. On the other hand, the higher rank
Witt algebra $\Der_F(F[t^{\pm 1}_1, \ldots, t_n^{\pm 1}])$, $n>1$,
is centrally closed (\cite[V, Th.~5.1]{RSS}). \end{examples}

\begin{definition}[Central extensions via $2$-cocycles.]
\label{n:sub:2coc} We have already seen in \S\ref{n:sec:aff} that
one can construct central extensions of a Lie algebra $L$ by using
\textit{$2$-cocycles}, which, we recall, are bilinear maps $\psi : L
\times L \to C$ into a vector space $C$ satisfying for all $l, l_1,
l_2, l_3 \in L$
\begin{equation} \label{n:sub:2cocy1}
 \psi(l,l) = 0 \quad\hbox{and} \quad
  \psi([l_1, l_2], l_3) + \psi([l_2,l_3], l_1) + \psi([l_3, l_1], l_2)
=0.\end{equation}  The first equation is of course equivalent to
$\psi(l_1, l_2) = - \psi(l_2, l_1)$. Given a $2$-cocycle $\psi: L
\times L \to C$, the algebra
\begin{equation}\label{n:eq:2coc2} K= L \oplus C\quad \hbox{by}
\quad
    [l_1 \oplus c_1, \, l_2 \oplus c_2]_K = [l_1, l_2]_L \oplus
\psi(l_1, l_2)
\end{equation}($l_i \in L$, $c_i \in C$) is a Lie algebra and
$\pr_L : K \to L$, $\pr_L(l\oplus c) = l$, is a central extension of
$L$, which we will denote by $\rmE(L,C, \psi)$ or $\rmE(L,\psi)$ for
short.

Conversely, given a central extension $f: K \to L$, let $s : L \to
K$ be a \textit{section of $f$\/} in the category of vector spaces,
i.e.~a linear map $s : L \to K$ such that $f\circ s = \Id_L$. Such a
section always exists: We can choose a subspace $L'$ of $K$, which
is complementary to $C = \Ker f$, and take $s= (f|_{L'})^{-1}$ which
makes sense since $(f|L') : L' \to L$ is an invertible linear map
(but in general not a Lie algebra homomorphism since $L'$ need not
be a subalgebra). Given a section $s$, the map
\begin{equation} \label{n:appcen1}
    \psi_s : L \times L \to C, \quad
         \psi_s(l_1, l_2) =  [s(l_1), s(l_2)]_K - s([l_1, l_2]_L)
\end{equation} turns out to be a $2$-cocycle. Moreover, the map
$$
  K \to L\oplus C, \quad x \mapsto
        f(x) \oplus \big( x- (s \circ f)(x) \big)= f(x) \oplus x_C,
$$
where $x_C$ is the $C$-component of $x\in K$, is an isomorphism from
the central extension $f: K \to L$ to the central extension
$\rmE(L,\Ker f, \psi_s)$. To summarize, modulo some verifications
left as an exercise, we have proven the following well-known result.
\end{definition}

\begin{proposition} \label{n:appcenprop} For any $2$-cocycle $\psi$ the
construction {\rm (\ref{n:eq:2coc2})} is a central extension
$\rmE(L,\psi)$ of $L$ and, conversely, every central extension $L$
is isomorphic as central extension to some $\rme(L,\psi)$.
\end{proposition}

\begin{exercise} \label{n:uebcover} Let $\psi : L \times L \to C$ be a $2$-cocycle and let
$C'$ be a subspace of $C$ satisfying $\psi(L,L) := \Span_F
\{\psi(l_1, l_2) : l_i \in L\}\subset C'$. Then $\rmE(L,C', \psi)$
is also a central extension, and if $\rmE(L,C,\psi)$ is a covering
then $C=\psi(L,L)$. \end{exercise}

\begin{examples} \label{n:sub:examcoc}
(a) Any Lie algebra $L$ has many uninteresting central extensions.
One can simply take the direct product of $L$ with an abelian Lie
algebra, i.e., $L\times C$ with product $[(l_1, c_1), \, (l_2,
c_2)]= ([l_1, l_2], 0)$ for $l_i \in L$, $c_i \in C$, and consider
the canonical projection $\pr_L : L \times L \to L$, which is a
central extension (but not a covering, unless $L$ is perfect and
$C=\{0\}$). Observe that the canonical inclusion $\inc : L \to
L\times C$ is a section of $\pr_L$, not only in the category of
vector spaces, but even in the category of Lie algebras. Its
associated $2$-cocycle $\psi_{\inc} = 0$. \sm

(b)  Let $h: L \to C$ be a linear map into some vector space $C$.
Then $\be_h : L \times L \to C$, $\be_h(l_1, l_2) = h([l_1, l_2])$
is a $2$-cocycle, a so-called \textit{$2$-coboundary}. \sm

The two examples are related in the following exercise.
\end{examples}

\begin{exercise}\label{n:uebcen} For a central extension $f: K \to L$ of
$L$ with $C=\Ker f$ the following are equivalent: \begin{itemize}
 \item[(i)] The extension $f: K \to L$ is split in the category of Lie algebras,
i.e, there exists a section $L \to K$ of $f$, which is a Lie algebra
homomorphism.

 \item[(ii)] For any section $s$ of $f$ the associated $2$-cocycle
$\psi_s$ is a $2$-coboundary.

\item[(iii)] There exists a section $s$ of $f$, for which the associated
$2$-cocycle $\psi_s$ is a $2$-coboundary.

\item[(iv)] As central extension, $f$ is isomorphic to the central
extension $\pr_L : L \oplus C \to L$.
\end{itemize}
If these conditions are fulfilled, one calls $f$ a \textit{split
extension}.
\end{exercise}

\begin{exercise} \label{n:uebcent1} Let $\psi : L \times L \to C$ be
a $2$-cocycle and let $\pi : C \to C'$ be a linear map. Show:

(a) $\psi'=\pi \circ \psi$ is a $2$-cocycle of $L$ and the map
\[\rmE(\pi) : \rmE(L,C,\psi) \to \rmE(L,C', \psi'),
   \quad l\oplus c \mapsto l \oplus \pi(c)\] is a homomorphism of central extensions of
$L$:
\[ \xymatrix{ \rmE(L,C,\psi) \ar[rr]^{\rmE(\pi)} \ar[dr]_{\pr_L}
 && \rmE(L,C',\psi') \ar[dl]^{\pr_L} \\ &L }\]

(b) If $\pi$ is surjective, the map $\rmE(\pi)$ is a central
extension of $L'=\rmE(L,C', \pi \circ \psi)$, which as central
extension of $L'$ has the form $\rmE(L',C'', \psi'')$ for
\[\psi''(l_1\oplus c_1', l_2 \oplus c_2') = \big((\Id - \ga \circ \pi)
\circ \psi \big)(l_1, l_2),  \] where $\ga : C' \to C$ is a section
of $\pi$ with $\ga(C')=C''$. \sm

(c) Conversely, suppose $f' : L'\to L$ is a central extension and
$f: \rmE(L,C,\psi) \twoheadrightarrow L'$ is a surjective
homomorphism of central extensions. Then $\pi = f|C$ maps $C$ onto
$C'=\Ker f'$ and there exists a unique isomorphism of extensions
$\Phi : L' \to \rmE(L,C', \psi')$, $\psi'=\pi \circ \psi$ such that
all triangles in the diagram below commute:
\[ \xymatrix{
    \rmE(L,C,\psi) \ar[rr]^f \ar[dr]^{\rmE(\pi)} \ar[ddr]_{\pr_L}
          && L' \ar[dl]_\Phi \ar[ddl]^{f'} \\
          & \rmE(L,C',\psi') \ar[d]_{\pr_L} \\ &L
}\]
\end{exercise}

\begin{exercise} \label{n:uebcent2}
Let $C= C_1 \oplus C_2$ be a vector space direct sum and denote by
$\pi_i : C \to C_i$ the canonical projections. Let $\psi : L \times
L \to C$ be a $2$-cocycle with the property that $\pi_2 \circ \psi$
is a $2$-coboundary. Then for $\psi_1 = \pi_1 \circ \psi$,
\[\rmE(L,C,\psi) \cong \rmE(L,C_1, \psi_1) \times C_2\] as
central extensions of $L$ (even as central extensions of the Lie
algebra $\rmE(L, C_1, \psi_1)$).
\end{exercise}

\begin{example} \label{n:generic}
Let $\inpr : L \times L \to F$ be a symmetric bilinear form, which
is invariant, see Ex.~\ref{n:ex:invfo}. We denote by $\SDer_F(L)$
the subalgebra of $\Der_F (L)$ which consists of all skew-symmetric
derivations, where a derivation $d\in \Der_F(L)$ is called
\emph{skew-symmetric\/} if $(d(l_1) \mid l_2) + (l_1 \mid d(l_2)) =
0$ for all $l_1,l_2\in L$. Observe that
\[ \IDer(L) = \{ \ad l : l\in L\} \triangleleft \SDer_F(L).\]

Let $D$ be a subspace of $\SDer_F(L)$ and let $D^*$ be its dual
space. Then the map $\psi_D : L\times L \to D^*$, given by
\begin{equation} \label{n:examcoc1}
   \psi_D(l_1, l_2) (d) = \big( d(l_1) \mid l_2)
\end{equation}
for $l_i \in L$ and $d\in D$, is a $2$-cocycle of $L$.
\end{example}

\begin{exercise} \label{n:uebgeneric}
Show: (a) (\ref{n:examcoc1}) defines indeed a $2$-cocycle.

(b) $\psi_D$ for $D \subset \IDer(L)$ is a $2$-coboundary.

(c) If $\tilde D$ is a subspace of $D$, then the central extension
$\rmE(L,D^*, \psi_D)$ of $L$ factors through the central extension
$\rmE(L,\tilde D^*, \psi_{\tilde D})$ of $L$,
\[  \rmE(L,D^* , \psi_D) \twoheadrightarrow
       \rmE(L, \tilde D^*, \psi_{\tilde D}) \twoheadrightarrow L.\]
\end{exercise}

\begin{definition}[Graded central extensions] Let $\La$ be an abelian
group, and let $L= \bigoplus_{\la \in \La} L^\la$ be a $\La$-graded
Lie algebra. We say that $f : K \to L$ is a \emph{$\La$-graded
central extension of $L$\/} if $K$ is a $\La$-graded Lie algebra and
$f$ is a central extension which is at the same time a homomorphism
of $\La$-graded algebras: $f(K^\la) \subset L^\la$ for all
$\la\in\La$.  A $\La$-graded central extension $f : K \to L$ is
called a {\it $\La$-covering}, if $f$ is a covering, i.e., $K$ is
perfect. We note that an arbitrary central extension of a graded Lie
algebra need not be a graded central extension.

A {\it homomorphism\/} of a $\La$-graded central extension $f : K
\to L$ to another $\La$-graded central extension $f' : K' \to L$ is
a homomorphism $g : K \to K'$ of $\La$-graded Lie algebras
satisfying $f= f' \circ g$, cf.~\ref{n:exthomdi}.

To define graded central extensions of a $\La$-graded Lie algebra
$L$ via a $2$-cocycle, we need (obviously) a \emph{$\La$-graded
$2$-cocycle}, i.e., a $2$-cocycle $\psi: L \times L \to C$ into a
$\La$-graded vector space $C=\bigoplus_{\la \in \La} C^\la$ which is
graded of degree $0$, \[ \psi(L^\la, L^\mu) \subset C^{\la + \mu}
\quad \hbox{for all $\la, \mu$.}
\] For a graded $2$-cocycle $\psi$ the Lie algebra $K=L\oplus C$ of
(\ref{n:eq:2coc2}) is naturally $\La$-graded by
\[
    K^\la = L^\la \oplus C^\la
\]
and the central extension $\pr_L : K \to L$ is a $\La$-graded
central extension. Conversely, if $f : K \to L$ is a $\La$-graded
central extension, we can choose a section $s: L \to K$ of the
underlying vector spaces of degree $0$, meaning $s(L^\la) \subset
K^\la$. The $2$-cocycle associated to $s$ in (\ref{n:appcen1}) is
then a graded $2$-cocycle. Thus, Prop.~\ref{n:appcenprop} holds in
an analogous way for graded central extensions.

The following proposition also shows that one does not have to
introduce a new object of a ``graded universal central extension''.
\end{definition}

\begin{proposition}[{\cite[1.16]{n:superuce}}] \label{n:appcenth} Let
$L=\bigoplus_{\la\in \La} L^\la $ be a $\La$-graded perfect Lie
algebra. Then its universal central extension $\fru : \uce(L) \to L$
is $\La$-graded, hence a $\La$-covering. Moreover, $\Ker \fru $ is a
graded subspace of $\uce(L)$.
\end{proposition}

\begin{example} \label{n:grcocy} We also have the graded versions of the
Example~\ref{n:sub:examcoc} (details left to the reader) and the
Example~\ref{n:generic}, whose details follow.

Let $L = \bigoplus_{\la \in \La} L^\la$ be a $\La$-graded Lie
algebra and let $\inpr$ be an invariant bilinear form on $L$, which
is \textit{$\La$-graded\/} in the following sense:
\[
   (L^\la \mid L^\mu) = 0 \quad \hbox{ if } \la + \mu \ne 0.
\]
We define the $\La$-graded subalgebra of $\grEnd_F(L)$
\begin{equation}
 \grSDer_F(L) = \grEnd_F(L) \cap \SDer_F(L) = \ts
   \bigoplus_{\la \in \La} \big(\SDer_F(L)\big)^\la
\end{equation}
where $(\SDer_F(L))^\la$ consists of all skew-symmetric derivations
of degree $\la$. If $D\subset \grSDer_F(L)$ is a graded subspace of
$\grSDer_F(L)$, the $2$-cocycle $\psi_D$ of (\ref{n:examcoc1}) is
$\La$-graded and maps $L \times L$ into $D^{\gr *}$, thus giving
rise to a graded central extension $\rmE(L, D^{\gr *}, \psi_D)$ of
$L$.
\end{example}

\begin{exercise} \label{n:uebcent3} Show that the $2$-cocycles $\psi$ of (\ref{n:eq:aff2.5}),
(\ref{n:eq:aff6}) and (\ref{n:eq:aff7}) can be obtained in the form
(\ref{n:examcoc1}), i.e., find an invariant bilinear form on
$\scL=L(\g,\si)$ resp. $\scL=L(\g,\boldsi)$ and a subspace $D
\subset \SDer_F(\scL)$ such that $\psi$ and $\psi_D$ yield
isomorphic central extensions of $\scL$.
\end{exercise}

It is not so surprising that the $2$-cocycles we used in sections
\ref{n:sec:aff} and \ref{n:sec:toroidal} can all be obtained in the
form $\psi_D$ for $D\subset \grSDer_F(L)$. This is a special case of
the following general result.

\begin{theorem}[{\cite{n:uce}}] \label{n:thgencen}
Let $L=\bigoplus_{\la \in \La} L^\la$ be a $\La$-graded Lie algebra,
which \begin{itemize}

\item[\rm (i)] is perfect and finitely generated as Lie algebra,

\item[\rm (ii)] has finite homogeneous dimension: $\dim L^\la <
\infty$ for all $\la \in \La$, and

\item[\rm (iii)] has an invariant nondegenerate $\La$-graded
symmetric bilinear form. \end{itemize}

{\rm (a)} Then $\Der_F(L)= \grDer_F(L)$ is $\La$-graded and has
finite homogeneous dimension, whence the same is true for
$\SDer_F(L)$.

{\rm (b)} The universal central extension $\uce(L)$ has finite
homogenous dimension with respect to the $\La$-grading of {\rm
\ref{n:appcenth}}. Moreover, \[ \uce(L) \cong \rmE(L,D^{\gr
*},\psi_D) \] as central extensions of $L$, where $D$ is any graded
subspace of\/ $\SDer_F(L)$ which complements $\IDer(L)$ in
$\SDer_F(L)$, and $\psi_D$ is the $2$-cocycle of {\rm
(\ref{n:examcoc1})}.\end{theorem}

\begin{remarks} (a) Th.~\ref{n:ucemm} is an application of
Th.~\ref{n:thgencen}, as is Th.~\ref{n:torfg}(c).

(b) The Exercise~\ref{n:uebcent2} gives some indication why it is
sufficient to take a subspace of $\SDer_F(L)$ complementing
$\IDer(L)$ and not an arbitrary subspace of $\SDer_F(L)$.
\end{remarks}

\begin{exercise} \label{n:uebcent4}
In the setting of Th.~\ref{n:thgencen}, every $\La$-graded central
covering of $L$ is isomorphic as central extension to a central
extension $\rmE(L,B^{\gr *}, \psi_B)$ for some graded subspace $B$
of $D$.
\end{exercise}
%

%
\section{Extended affine Lie algebras: Definition and first examples
  }
\label{n:ch:eala-def-exam}

Rather than constructing Lie algebras in a concrete way as we have
done in Lecture~\ref{n:ch:affgen}, in this chapter we will define
extended affine Lie algebras by a set of axioms and give examples.
We will see that these examples encompass all the examples of
Lecture~\ref{n:ch:affgen} (with the exception of the choice 2. for
$D$ in \ref{n:sec:toroidal}). \sm

As before we will consider Lie algebras over an arbitrary  field $F$
of characteristic $0$, but we will no longer assume that $F$ has
enough roots of unity (multiloop algebras will not be play a role
here), except in \S\ref{n:sec:ealaone} where $F=\CC$).

\subsection{Definition of an extended affine Lie algebra} \label{n:sec:eala-def}

An \textit{extended affine Lie algebra\/}, or EALA for short, is a
pair $(E,H)$ consisting of a Lie algebra $E$ over $F$ and subalgebra
$H$ satisfying the following axioms (EA1) -- (EA6). \sm

\begin{description}
\item[(EA1)] {\it $E$ has an invariant nondegenerate symmetric bilinear
form $\inpr$.} \sm

\item[(EA2)] \textit{$H$ is nontrivial finite-dimensional toral and
self-centralizing subalgebra of $E$.}
\end{description}

Before we can state the other four axioms, we need to draw some
consequences of the axioms (EA1) and (EA2). But first we give
explanations of some of the notions used. The term
\textit{invariant} (= \textit{associative}) means that $\inpr$
satisfies $([e_1,e_2]\mid e_3) = (e_1\mid [e_2,e_3])$ for all $e_i
\in E$, and $\inpr$ is \textit{nondegenerate} if $(e\mid E) = 0
\implies e=0$. In the context of above, a \textit{toral
subalgebra\/}, sometimes also called an \textit{$\ad$-diagonalizable
subalgebra\/} is a subalgebra $H$ which induces a decomposition of
$E$ via the adjoint representation of $H$:
\begin{equation}\label{n:eq:eala-def0} \begin{split}
E  &= \tplus_{\al \in H^*} E_\al,  \\
 E_\al &= \{ e\in E: [h,e] = \al(h)e \hbox{ for all } h\in H\}.
\end{split}\end{equation} Such a subalgebra is necessarily abelian, whence
$H\subset E_0=\{ e\in E : [h,e]=0 \hbox{ for all } h\in H\}$. That
$H$ is also required to be \textit{self-centralizing\/} means
$$ 
H=E_0.$$
 Now to the consequences of (EA1) and (EA2). Because of invariance of
the bilinear form $\inpr$, we have \begin{equation}
\label{n:eala-def5}
    (E_\al \mid E_\be) = 0 \quad \hbox{if $\al + \be \ne 0$},
\end{equation} in particular the restriction of the bilinear form to $E_0=H$
is nondegenerate. Because of this and finite-dimensionality of $H$,
every linear form $\al \in H^*$ is represented by a unique $t_\al
\in H$, defined by the condition that $(t_\al\mid h) = \al (h)$
holds for all $h\in H$. This allows us to transport the restricted
form $\inpr \mid H \times H $ to a symmetric bilinear form on $H^*$,
also denoted $\inpr$ and defined by
\begin{equation}\label{n:ealadef2}
 (\al \mid \be) = (t_\al \mid
t_\be), \quad \al,\be\in H^*. \end{equation} This transport of
bilinear forms is a standard procedure in the theory of semisimple
Lie algebras, see for example \cite[\S8]{hum}.
We can now define \begin{equation}\begin{split} \label{n:eq:eala-def4}
   R &= \{ \al \in H^* : E_\al \ne 0 \} \quad
          (\hbox{\it set of roots of $(E,H)$}), \\
  R^0 &= \{ \al \in R : (\al \mid \al) = 0 \}
      \quad(\hbox{\it null roots}), \\
 R\an &= \{ \al \in R : (\al \mid \al)\ne 0 \}
       \quad(\hbox{\it anisotropic roots}).
\end{split}\end{equation}
We prefer to call $R$ the set of roots of $(E,H)$ and not the ``root
system'' since we want to restrict the latter term for root systems
in the usual sense, see \ref{n:arsexfin}. We point out that by
definition $0$ is a root,
$$ 0 \in R^0 \subset R.
$$
This is the customary convention for EALAs and has some notational
advantages.

We define the \textit{core of $(E,H)$\/} as the subalgebra $E_c$ of
$E$ generated by all anisotropic root spaces:
$$
  E_c = \lan \, \textstyle \bigcup_{\al \in R\an} E_\al \, \ran_{\rm
subalg}
$$
 We can now state the remaining four axioms. \sm

\begin{description}
\item[(EA3)] \textit{For every $\al \in R\an$ and $x_\al \in E_\al$, the
operator $\ad x_\al$ is locally nilpotent on $E$}. \sm

\item[(EA4)] \textit{$R\an$ is connected\/} in the sense that for any
decomposition $R\an = R_1 \cup R_2$ with $(R_1 \mid R_2)= 0$ we have
$R_1 = \emptyset$ or $R_2 = \emptyset$. \sm

\item[(EA5)] \textit{The centralizer of the core $E_c$ of $E$ is contained
in $E_c$}: $\{e \in E : [e, E_c] =0 \} \subset E_c$.

\sm

\item[(EA6)] \textit{The subgroup $\La = \Span_\ZZ(R^0) \subset H^*$ generated
by $R^0$ in $(H^*,+)$ is a free abelian group of finite rank.} In
other words, $\La \cong \ZZ^n$ for some $n\in \NN$ (including
$n=0$!).
\end{description} \sm

The term \textit{locally nilpotent} means that for every $e\in E$
there exists an $n\in \NN$, possibly depending on $e$, such that
$(\ad x_\al)^n (e) = 0$. The property (EA5) is called
\textit{tameness}. The condition $[e,E_c]=0$ is of course equivalent
to $[e,E_\al]=0$ for all $\al \in R\an$. The rationale for this
axiom is the following. The subalgebra $E_c$ is in fact an ideal of
$E$ (Th.~\ref{n:ealcor}). Hence we have a representation $\rho$ of
$E$ on $E_c$, given by $\rho(e) (x_c) = [e,x_c]$ for $e\in E$ and
$x_c\in E_c$. The kernel of the representation $\rho$ is the
centralizer of $E_c$ in $E$. Hence tameness means that $\Ker \rho
\subset E_c$. The idea here is that the core $E_c$ should control
$E$. We will make this more precise in section
\ref{n:sec:genconstr}. The rank of the free abelian group $\La$ in
axiom (EA6) is called the \textit{nullity} of $(E,H)$. It is
invariant under isomorphisms. We will describe EALAs of nullity $0$
and $1$ below.\sm

Although the structure of an EALA requires the existence of an
invariant nondegenerate symmetric bilinear form $\inpr$ in the axiom
(EA1), which is then used to define the anisotropic roots, it turns
out that this bilinear form is really not so important. Because of
this, we have defined an EALA as a pair $(E,H)$ and not as a triple
$(E,H,\inpr)$ as it is for example done in \cite{AF:isotopy}.
Consequently, an \textit{isomorphism\/} from an EALA $(E,H)$ to
another EALA $(E',H')$ is a Lie algebra isomorphism $f : E \to E'$
such that $f(H) = H'$. It is immediate that any isomorphism induces
a bijection $f'$ between the set of roots $R$ and $R'$ of $(E,H)$
and $(E',H')$ respectively. It then follows that $f'$ maps $R\an$
onto ${R'}\an$, whence also $R^0$ onto ${R'}^0$. One can then show
that $f'$ preserves the forms on $X=\Span_F(R)$ and $X'=\Span_F(R')$
up to scalars. \sm

For $F=\CC$ one can define a special class of EALAs. We call a pair
$(E,H)$ a \textit{discrete EALA} if it satisfies the axioms (EA1) --
(EA5) and in addition

\begin{description} \item[(DE)] $R$ is a discrete subset of $H^*$ with respect to the
natural topology of the finite-dimensional complex vector space
$H^*$. \end{description}

\noindent It is justified to call a discrete EALA an EALA, since one
can show that a discrete EALA also satisfies (EA6). Indeed, this
follows from Prop.~\ref{n:earsstrut} and Th.~\ref{n:ears&eala}.
However, not every EALA over $\CC$ is a discrete EALA (see
\cite[6.17]{n:persp}). \sm

\textbf{Some historical comments.} Although there were some
precursors (papers by Saito and Slodowy for nullity $2$), it was in
the paper \cite{HKT} by the physicists H{\o}egh-Krohn and
Torr{\'e}sani that the class of discrete extended affine Lie
algebras was introduced, however not under this name. Rather, they
were called ``irreducible quasi-simple Lie algebras" and later
(\cite{bgk,bgkn}) ``elliptic quasi-simple Lie algebras". The stated
goal of the paper \cite{HKT} was applications in quantum gauge
theory. The theory developed there did however not stand up to the
scrutiny of mathematicians. The errors of \cite{HKT} were corrected
in the AMS memoir \cite{aabgp} by Allison, Azam, Berman, Gao and
Pianzola. There also the name ``extended affine Lie algebras''
appears for the first time. But not in the sense as defined above.
Rather, the authors develop the basic theory of what here are called
discrete EALAs. Nevertheless,  \cite{aabgp} has become the standard
reference even for the more general extended affine Lie algebras,
since many of the results presented there for discrete extended
affine Lie algebras easily extend to the more general setting. The
definition of an extended affine Lie algebra given above is due to
the author (\cite{n:eala}) and was motivated by the fact that all
the examples presented in \cite{aabgp} did make sense over an
arbitrary base field $F$ and not just over $\CC$ only. Before
\cite{n:eala} the tameness axiom (EA5) was not part of the
definition of an EALA. However, as examples show (\cite[\S3]{bgk} or
\cite[6.10]{n:persp}), it seems impossible to classify EALAs without
(EA5). After \cite{n:eala}, several generalizations of EALAs have
been proposed. They are surveyed in \cite{n:persp}. \sm

\subsection{Some elementary properties of extended affine Lie algebras}
 \label{n:sec:ealaelem}

The following chapters will (hopefully) show that extended affine
Lie algebras share many properties with familiar Lie algebras, like
finite-dimensional split simple Lie algebras or affine Kac-Moody Lie
algebras. Some of these properties are immediate consequences of the
axioms. The following (strongly recommended!) exercise gives an
incomplete list of such properties.

\begin{exercise} \label{n:ex:eala-def1} Let $(E,H)$ be an EALA. We
use the notation of  above. Show: \begin{enumerate}

\item[(a)] For $\al,\be \in R$ we have \begin{equation} \label{n:ealagrad}
[E_\al, E_\be] \subset E_{\al + \be}.  \end{equation} Thus the root
space decomposition (\ref{n:eq:eala-def0}) is a grading by the
abelian group $\Span_\ZZ(R)$.

\item[(b)] $H$ is a \textit{Cartan subalgebra}, defined as a nilpotent
subalgebra which is self-normalizing: $H= \{ e\in E : [e,H] \subset
H \}$.

\item[(c)] For $\al,\be \in R$ we have $(E_\al \mid E_\be) = 0$
unless $\al + \be = 0$. The restriction of the bilinear form $\inpr$
to $E_\al \times E_{-\al}$ is nondegenerate, i.e., if $x_\al \in
E_\al$ satisfies $(x_\al \mid E_{-\al}) = 0$  then $x_\al = 0$. In
particular, $R=-R$.

\item[(d)] For $\al \in R$ and $x_\al \in E_\al$ and $y_{-\al} \in E_{-\al}$,
\begin{equation} \label{n:eala-def2}
 [x_\al, \, y_{-\al}] = (x_\al \mid y_{-\al}) \, t_\al.
\end{equation}
In particular, $[E_\al, E_{-\al}] = Ft_\al$, and if $\al \in R\an$
then
\begin{equation} \label{n:eala-def4}
  [[E_\al , E_{-\al}] , \, E_\al] = E_\al .
\end{equation}

\item[(e)] The core $E_c$ satisfies
\begin{equation}
 E_c = \textstyle \Big( \oplus_{\al \in R\an} E_\al \Big) \oplus
    \Big( \bigoplus_{\al \in R^0} (E_c \cap E_\al) \Big).
\end{equation}
\end{enumerate} \end{exercise}

We will now show that EALAs are built out of ``little'' $\lsl_2$'s
and Heisenberg's (albeit in a complicated way).

\begin{proposition} \label{n:ealafact} Let $(E,H)$ be an extended affine Lie algebra, with
anisotro\-pic root $R\an$ and null roots $R^0$. \sm

{\rm (a)} Let $\al \in R\an$. Then $\dim E_\al = 1$, and for any
$e_\al \in E_\al$ there exists $f_\al \in E_{-\al}$ such that
$(e_\al, h_\al = [e_\al, f_\al], f_\al) \in E_\al \times H \times
E_{-\al}$ is an $\lsl_2$-triple:
$$
   E_\al \oplus [E_\al, E_{-\al}] \oplus E_{-\al}
     = F e_\al \oplus Fh_\al \oplus Ff_\al \cong \lsl_2(F).
$$

{\rm (b)} Let $\al\in R^0$. Then for any $0\ne x_\al \in E_\al$
there exists $y_\al \in E_{-\al}$ such that $[x_\al, y_\al] = t_\al$
and
$$
    Fx_\al \oplus Ft_\al \oplus Fy_\al \cong \frh_3,
$$
the $3$-dimensional Heisenberg algebra.
\end{proposition}

It is not true that $\dim E_\al = 1$ if $\al \in R^0$ (this is
already not true in the examples of sections \ref{n:sec:eala0} and
\ref{n:sec:ealaone}). But we will show in Th.~\ref{n:bdd} that all
root spaces $E_\al$ are finite-dimensional in a rather strong way.

The following exercise shows that one can ``extend'' the
$3$-dimensional Heisenberg subalgebras in (b) above.

\begin{exercise} In the setting and notation of Prop.~\ref{n:ealafact}(b) show
that there exists $d_\al \in H$ such that
$$
   [d_\al, x_\al]= x_\al \quad \hbox{and} \quad [d_\al, y_\al] = -
y_\al. $$ Hence
$$
    Fx_\al \oplus Ft_\al \oplus Fd_\al \oplus Fy_\al
$$
is a $4$-dimensional subalgebra. It is $2$-step solvable, not
nilpotent and isomorphic to the subalgebra
$$
  \Big \{ \left( \begin{smallmatrix} 0 & a & b \\ 0 & c & d \\ 0 & 0 & 0
           \end{smallmatrix} \right):   a,b,c,d\in F \Big\}
$$
of $\gl_3(F)$.
\end{exercise}

And now an exercise which implies that in an EALA one can produce
many so-called elementary automorphisms.

\begin{exercise} \label{n:ex:eala-def2} Let $M$ be an $F$-vector
space. Once calls an endomorphism $f\in \End_F(M)$ \textit{locally
nilpotent} if for every $m\in M$ there exists $n\in \NN$, possibly
depending on $m$,  such that $f^n(m)=0$. \sm

(a) Show that the following conditions are equivalent for $f\in
\End_F(M)$:
\begin{enumerate}
 \item[]
\begin{enumerate}
\item[(i)] $f$ is locally nilpotent,

\item[(ii)] for every finitely spanned subspace $N$ of $M$ there exists
a finite-dimensional subspace $P$ of $M$ such that $N\subset P$ and
$f(P) \subset P$,

\item[(iii)] $f$ is nilpotent on every finite-dimensional and
$f$-invariant subspace of $M$.
\end{enumerate} \end{enumerate} \sm

(b) Let $f\in \End_F(M)$ be locally nilpotent and define the
\textit{exponential $\exp f$ of $f$} by
$$
    (\exp f)(m) = \ts \sum_{n\in \NN} \, \frac{1}{n!} f^n(m),
$$
for $m\in M$ (note that the sum on the right is always finite). Show
that $\exp f$ is an invertible endomorphism of $M$ with inverse
given by $(\exp f)^{-1} = \exp (-f)$. \sm

(c) Let $L$ be a Lie algebra and let $d$ be a locally nilpotent
derivation of $L$. Show that then $\exp d$ is an automorphism of
$L$.
\end{exercise}

We will next present some examples of EALAs.

\subsection{Extended affine Lie algebras of nullity $0$}
\label{n:sec:eala0}

Let $\g$ be a finite-dimensional split simple Lie algebra with
splitting Cartan subalgebra $\frh$, for example $\lsl_l(F)$ or a
finite-dimensional simple Lie algebra over an algebraically closed
field. We will show that then $(\g,\frh)$ is an EALA of nullity $0$.
The facts needed to prove this can be found in \cite[VIII,
\S2]{bou:lie78} or in \cite[\S8]{hum} for $F$ algebraically closed.
\sm

(EA1) Up to a scalar, there exists only one invariant nondegenerate
symmetric bilinear form on $\g$, the Killing form $\ka$. Hence we
can (and will) take $\inpr = \ka$. \sm

(EA2) By definition of a splitting Cartan subalgebra, the Lie
algebra $\g$ has a root space decomposition $$\g = \textstyle \g_0
\oplus \big( \bigoplus_{\al \in \Phi} \, \g_\al\big), \quad \g_0 =
\frh,
$$ where $\Phi$ is the root system of $(\g,\frh)$ (which is a
reduced root system in the usual sense, see \ref{n:arsexfin}) and
where the root spaces $\g_\al$ are defined as in
\ref{n:eq:eala-def0}. Hence the set of roots $R$ of $(\g,\frh)$ is
\begin{equation} \label{n:sec:eala01}   R = \{0\} \cup \Phi.
 \end{equation}
It is a basic fact that $\ka(t_\al, t_\al) \ne 0$ for $t_\al \in
\frh$ representing $\al \in \Phi$ via $\ka(t_\al, h) = \al(h)$ for
all $h\in \frh$. Hence, the anisotropic and null roots are
$$R\an = \Phi \quad \hbox{and} \quad R^0 = \{0\}. $$

(EA3) is now obvious: From $[\g_\al, \g_\be] \subset \g_{\al + \be}$
for $\al\in \Phi$ and $\be \in R$ and finite-dimensionality of $\g$,
it is clear that $\ad x_\al$ for $x_\al \in \g_\al$ is not only
locally nilpotent but even (globally) nilpotent.

(EA4) is another way of saying that $\Phi$ is an irreducible root
system. This is indeed the case and follows from simplicity of $\g$.

(EA5) We first need to determine the core $\g_c$ of $\g$. By
definition, $\g_c$ is the subalgebra of $\g$ generated by
$\bigoplus_{\al \in \Phi} \g_\al$. Since $\frh = \sum_{\al \in \Phi}
[g_\al, \g_{-\al}]$ we have
$$ \g_c = \g.$$
It is now a tautology that (EA5) holds, i.e., that the centralizer
of the core $\g_c$ is contained in $\g_c=\g$. Of course, we know
even more: The centralizer of the core equals the centre of $\g$,
and is therefore $\{0\}$.

(EA6) We have $\La= \lan R^0\ran = \lan \{0\}\ran = \{0\}$. \sm

\noindent We have now shown:
\begin{equation} \label{n:eala-def:prop}
\hbox{\it A finite-dimensional split simple Lie algebra is an EALA
of nullity $0$.}
\end{equation}
We will see in Prop.~\ref{n:fdeala} that the converse of
(\ref{n:eala-def:prop}) is true too. We thus know all the nullity
$0$ examples of EALAs, and can therefore focus on the higher nullity
examples. We will answer the case of nullity $1$ in the next section. %

\subsection{Affine Kac-Moody Lie algebras again} \label{n:sec:ealaone}

To justify the name extended \textit{affine Lie algebra\/}, we will
now show that any affine Kac-Moody Lie algebra is an extended affine
Lie algebra. To do so, we will need some basic facts about affine
Kac-Moody Lie algebras. All of them can be found in Kac's book
\cite{kac}. Since this reference uses $\CC$ as base field, we will
do the same in this section. But everything we say here holds true
for arbitrary algebraically closed fields of characteristic $0$.
Thus we let \begin{align*}
   \scL &=  \ts \bigoplus_{n\in \ZZ} \g_{\bar n}\ot \CC t^n \\
   \hat \scL  &= \hat \scL(\g, \si) = \scL \oplus \CC c \oplus \CC d
 \end{align*} be the complex Lie algebra described in (\ref{n:eq:aff8}) and
(\ref{n:eq:aff9}). Recall that $\g$ is a finite-dimensional simple
Lie algebra over $\CC$ and $\si$ is a diagram automorphism of $\g$.
We let $m\in \{1,2,3\}$ be the order of $\si$, and denote the
canonical map $\ZZ \to \ZZ/m\ZZ$ by $n\mapsto \bar n$. Recall from
(\ref{n:eq:aff1}) and (\ref{n:eq:aff1.5}) that $\si$ induces a
$\ZZ/m \ZZ$-grading of $\g$, namely
$$
     \g= \g_{\bar 0} \oplus \cdots \oplus \g_{\overline{m-1}}
$$
 where $\g_{\bar n} = \{x\in \g: \si(x) = \ze^n x\}$ for a primitive
$m$th root of unity $\ze$. For example, for $m=2$ we get a
$\ZZ/2\ZZ$-grading $\g=\g_{\bar 0} \oplus \g_{\bar 1}$ with
$\g_{\bar 0} = \{ x\in \g : \si(x) = x\}$ and $\g_{\bar 1} = \{ x\in
\g: \si(x) = - x\}$. We  identify $\g_{\bar 0} \equiv \g_{\bar 0}
\ot \CC t^0$. \sm

We now verify the axioms (EA1) -- (EA5) and (DE) which, we recall,
implies (EA6). \sm

(EA1) We let $\ka$ be the Killing form of $\g$ and define a bilinear
form $\inpr$ on $\hat \scL$, using the notation of
(\ref{n:eq:aff9}),
\begin{equation}\label{n:eq:ealaone1}
   \begin{split}
  &\big( u_{\bar \la} \ot t^\la \oplus s_1 c \oplus s'_1 d  \mid
        v_{\bar \mu} \ot t^\mu \oplus s_2 c \oplus s'_2 d \big)
 \\ &\qquad = \ka(u_{\bar \la}, v_{\bar \mu})\de_{\la, - \mu} \, + \, s_1 s_2' + s_2 s'_1.
\end{split} \end{equation} The form is visibly symmetric. The reader is
invited in Exercise~\ref{n:ex:ealaone1} to show that it is in fact
an invariant nondegenerate symmetric bilinear form on $\hat \scL$,
as required in (EA1). In anticipation of the later developments, we
point out that $\inpr$ has the following features:
\begin{itemize}

\item  $\hat \scL$ is an orthogonal sum of $\scL$ and $\CC c
\oplus \CC d$: $ \hat \scL = \scL \perp (\CC c \oplus \CC d)$,

\item  $\CC c \oplus \CC d$ is a hyperbolic plane, i.e.,
$(c\mid c) = 0 = (d \mid d)$ while $(c \mid d) = 1$.

\item The Laurent polynomial ring $\CC[t^{\pm 1}]$ has a
nondegenerate symmetric bilinear form $\eps$ given by $\eps(t^\la,
t^\mu) = \de_{\la, -\mu}$. It is \textit{invariant} in the sense
that $\eps(pq,r) = \eps(p,qr)$ for $p,q,r\in \CC[t^{\pm 1}]$, and is
\textit{graded} in the sense that $\eps(t^\la, t^\mu) = 0$ unless
$\la +\mu = 0$. For $\si = \Id_\g$, the bilinear form on the loop
algebra $\g \ot \CC[t^{\pm 1}]$ is simply the tensor product form
$\ka \ot \eps$, and for a general $\si$ the form is obtained by
restriction. \end{itemize} \sm

(EA2) To construct a subalgebra $H$ as required in axiom (EA2) we
start with a Cartan subalgebra $\frh$ of $\g$. Since $\si$ is a
diagram automorphism, it leaves $\frh$ invariant. We let
$$   \frh_{\bar 0} = \frh \cap \g_{\bar 0}
         =\{ h \in \frh : \si(h) = h\}
$$
and put
$$
     H = \frh_{\bar 0} \oplus \CC c \oplus \CC d.
$$
One knows that $\g_{\bar 0}$ is a simple Lie algebra with Cartan
subalgebra $\frh_{\bar 0}$ (\cite[Prop.~7.9]{kac}). The grading
property implies that $[\g_{\bar 0}, \g_{\bar n}] \subset \g_{\bar
n}$ for $n\in \ZZ$. Hence $\g_{\bar 0}$ acts on $\g_{\bar n}$ by the
adjoint action. Let $\De_{\bar n}$ be the set of weights of the
$\g_{\bar 0}$-module $\g_{\bar n}$ with respect to $\frh_{\bar 0}$:
 \begin{align*}
    \g_{\bar n} &= \ts \bigoplus_{\ga \in \De_{\bar n}}
           \g_{\bar n,  \ga} \\
   \g_{\bar n, \ga} &= \{ x\in \g_{\bar n} : [h_{\bar 0}, x ]
     = \ga(h_{\bar 0}) x \hbox{ for all } h_{\bar 0} \in
              \frh_{\bar 0} \}.
\end{align*}
In particular, $\De_{\bar 0} \setminus \{0\}$ is the root system of
$\g_{\bar 0}$ with respect to $\frh_{\bar 0}$ and $\frh_{\bar 0} =
\g_{\bar 0, 0}$.

We extend $\De_{\bar n} \subset \frh_{\bar 0}^*$ to a linear form on
$H$ by zero, i.e., for $\ga \in \De_{\bar n} $ we put
$$
     \ga(h_{\bar 0} \oplus sc \oplus s'd) = \ga(h_{\bar 0})
$$
and define a linear form $\de$ on $H$ by
$$
    \de(h_{\bar 0} \oplus sc \oplus s'd) = s'.
$$
Then for $\ga \in \De_{\bar n}$, $n \in \ZZ$, we have
\begin{equation} \label{n:eq:ealaone2} \begin{split}
   \hat \scL_{\ga \oplus n\de} &= \{ u \in \hat \scL :
         [h ,u ] = (\ga\oplus n \de)(h) u \hbox{ for all } h\in H\}
    \\ &= \begin{cases} \g_{\bar n, \ga} \ot t^n,
                                 & \ga\oplus n\de \ne 0, \\
                       H,   & \ga \oplus n \de = 0, \end{cases}
\end{split} \end{equation}
whence $\hat \scL = \bigoplus_{\al \in R} \hat \scL_\al$ has a root
space decomposition with respect to $H$ with set of roots
 \begin{equation} \label{n:eq:ealaone3}
   R = \{ \ga \oplus n \de : \ga \in \De_{\bar s},\, \bar n = \bar
s, 0 \le s < m\}. \end{equation} This establishes (EA2).

To check the other axioms we first need to determine which of the
roots in $R$ are the null respectively anisotropic roots. Following
the procedure in \S\ref{n:sec:eala-def}, we consider the restriction
of the bilinear form $\inpr$ to $H$. With obvious notation this is
$$
   \big(h_{\bar 0} \oplus s_1c \oplus s_1' d \mid
         h'_{\bar 0} \oplus s_2 c \oplus s_2' d) =
     \ka(h_{\bar 0}, h'_{\bar 0}) + s_1 s_2' + s_2 s'_1.
$$
Since $\ka |_{\frh_{\bar 0} \times \frh_{\bar 0}}$ is nondegenerate,
this is indeed a nondegenerate symmetric bilinear form on $H$, as it
should be. Let $t_\ga \in \frh_{\bar 0}$ be the element representing
$\ga \in \frh^*_{\bar 0}$: $\ka(t_\ga, h_{\bar 0}) = \ga(h_{\bar
0})$ for all $h_{\bar 0} \in \frh_{\bar 0}$. For the canonical
extension of $\ga$ to a linear form of $H$, also denoted by $\ga$,
we then get $(t_\ga \mid h) = \ga(h)$ for all $h\in H$. Moreover $(c
\mid h_{\bar 0} \oplus sc \oplus s'd) = s' = \de(h_{\bar 0} \oplus
sc \oplus s'd)$ shows that $\de$ is represented by $t_\de=c\in H$.
Therefore $\al = \ga \oplus n \de \in R$ is represented by
$$
   t_{\ga \oplus n \de} = t_\ga \oplus nc.
$$
Now observe $(t_{\ga \oplus n\de} \mid t_{\ga \oplus n\de}) =
  (t_\ga \oplus nc \mid t_\ga \oplus nc) = \ka(t_\ga, t_\ga)$.
It is of course well-known that $\ka(t_\ga, t_\ga)\ne 0$ for $0\ne
\ga\in \De_{\bar 0}$. But one can (easily) show that this also holds
for any $0 \ne \ga\in \De_{\bar n}$. We therefore get
\begin{equation} \label{n:eq:ealaone4}
   R\an = \{ \ga \oplus n \de \in R: \ga \ne 0\}\quad\hbox{and}
\quad
    R^0 = \ZZ \de,
\end{equation} which in the theory of affine Kac-Moody algebras are usually
called \textit{real} and \textit{imaginary roots}. We are now set
for the verification of the  remaining axioms. \sm

(EA3) holds in the stronger form: $\ad \hat \scL_\la$, $\al \in
R\an$, is nilpotent. (We have already seen the same phenomenon in
the Example~\ref{n:sec:eala0} of a finite-dimensional split simple
Lie algebra. Perhaps the reader wonders if this is true in general.
The answer is yes.)\sm

(EA4) The verification of (EA4) is left to the reader. \sm

(EA5) The core of $\hat \scL$ is $\hat \scL_c = \big(
\bigoplus_{n\in \ZZ} \g_{\bar n} \ot \CC t^n\big) \oplus \CC c$, and
therefore equals the derived algebra $[\hat \scL, \hat \scL]$ of
$\hat \scL$. The centralizer of $\hat \scL_c$ in $\hat \scL$, in
fact the centre of $\hat \scL$ is $\CC c \subset \hat \scL_c$, see
Exercise~\ref{n:ueb1}. \sm

(DE) In this example the subgroup $\La= \lan R^0\ran$ equals $R^0 =
\ZZ d$ and is a discrete subset of $H^*$. \sm

We have now shown one implication of the following result.

\begin{theorem}[\cite{abgp}] \label{n:ealaone:th} A complex Lie algebra $E$ is a discrete EALA of
nullity $1$ if and only if $E$ is an affine Kac-Moody Lie algebra.
 \end{theorem}

\begin{exercise} \label{n:ex:ealaone1} Check the following details of the
construction above.

(a) (\ref{n:eq:ealaone1}) defines an invariant symmetric bilinear
form on $\hat \scL$.

(b) $\hat \scL$ has a root space decomposition whose root spaces are
given by (\ref{n:eq:ealaone2}) and whose set of roots is
(\ref{n:eq:ealaone3}).

(c) $\ka(t_\ga, t_\ga) \ne 0$ for any $0\ne \ga\in \De_{\bar n}$.

(d) (EA4) holds for $(\hat \scL, H)$.
\end{exercise}

\subsection{Higher nullity examples} \label{n:sec:ealan}

We have seen all examples of EALAs of nullity $0$ and $1$. In this
section we will construct examples of higher nullity. To simplify
things we consider untwisted algebras (no non-trivial finite order
automorphism are involved). We can therefore go back to our standard
setting: $\g$ is a split simple Lie algebra over a field $F$ of
characteristic $0$. \sm

As in \S\ref{n:sec:toroidal} let $F[t_1^{\pm 1}, \ldots, t_n^{\pm
n}]$ be the Laurent polynomial ring in $n$ variables and let
$$
  L =   L(\g) = \g \ot F[t_1^{\pm 1}, \ldots, t_n^{\pm n}]
$$
be the associated untwisted multiloop algebra. We have seen in
Exercise~\ref{n:ueb3} that $L$ has a $2$-cocycle $\psi : \scL \times
\scL \to F^n = : \euC$, given by (\ref{n:eq:aff6}): $\psi(u \ot
t^\la, v\ot t^\mu) =  \de_{\la +\mu, 0} \, \ka(u,v) \,\la$. We can
therefore define the central extension
$$
   K = L \oplus \euC
$$
with product (\ref{n:eq:2coc2}). In (\ref{n:eq:aff10}) we have
defined degree derivations $\pa_i$, $i=1,\ldots, n$, of $K$. Let
\begin{equation} \label{n:ealandeg}
   \euD= \Span_F \{ \pa_1, \ldots, \pa_n\}\end{equation} and define the Lie
algebra $E$ as the semidirect product,
$$
   E = \big( L(\g) \oplus \euC \big)\rtimes \euD.
$$
Let $\frh$ be a Cartan subalgebra of $\g$ and put
$$   H = \frh \oplus \euC \oplus \euD.   $$
We claim  that \textit{$(E,H)$ is an EALA of nullity $n$.} \sm

(EA1) We will mimic the construction of an invariant nondegenerate
symmetric bilinear form in \S\ref{n:sec:ealaone} and require

\begin{itemize}

\item $( L(\g) \mid \euC \oplus \euD)=0$.

\item $\euC \oplus \euD$ is a hyperbolic space with $(\euC
\mid \euC) = 0 = (\euD \mid \euD)$ and $$ (\tsum_i  s_i c_i \mid
\sum_i s'_i \pa_i) = \sum_i s_i s'_i, $$ where $c_1, \ldots, c_n$ is
the canonical basis of $F^n$. Thus $\euC \oplus \euD$ is the
orthogonal sum of the $n$ hyperbolic planes $Fc_i \oplus F\pa_i$.

\item On $L(\g)$ the form is the tensor product form of the
Killing form $\ka$ of $\g$ and the natural invariant bilinear form
on $F[t_1^{\pm 1}, \ldots, t_n^{\pm 1}]$. \end{itemize}

\noindent Putting all these requirements together, we arrive at the
global formula which is completely analogous to
(\ref{n:eq:ealaone1}):
\begin{align} \label{n:eq:ealan0} \begin{split}
  &\textstyle\big( u \ot t^\la \,\oplus \, \sum_i s_i  c_i \oplus \sum_j s'_j \pa_j
  \, \mid \,
        v \ot t^\mu \,\oplus \,\sum_i t_i c_i \oplus \sum_j t'_j \pa_j \big)
 \\ & \qquad \qquad \textstyle = \ka(u, v)\de_{\la, - \mu} \, + \, \sum_i ( s_i t_i' + t_i s'_i).
\end{split} \end{align}

(EA2) Let $\frh$ be a splitting Cartan subalgebra and let $\Phi$ be
the usual root system of $(\g,\frh)$, thus $0\not\in \Phi$. We put
$\De = \{0\} \cup \Phi$ and then have the root space decomposition
$\g = \bigoplus_{\ga \in \De} \g_\ga$ with $\g_0 = \frh$. We embed
$\De \hookrightarrow H^*$ by requiring $\ga \mid  \euC \oplus \euD =
0$ for $\ga \in \De$. Also we embed $\La = \ZZ^n \hookrightarrow
H^*$ by $\la(\frh \oplus \euC) = 0$ and $\la (\pa_i) = \la_i$ for
$\la = (\la_1, \ldots, \la_n)\in \La$. Then $E$ has the root space
decomposition $E=\bigoplus_{\al \in R} E_\al$ with root spaces
\begin{align}\label{n:eq:ealan1}
 E_{\ga \oplus \la} &= \g_\ga \ot t^\la \quad (\ga \oplus
\la \ne 0),  & E_0 = H, \\ \label{n:eq:ealan2}
   R\an &= \Phi \times \La, & R^0 = \La.
\end{align} \sm

\noindent It is now not difficult to verify (EA3) -- (EA5) and (DE).
Thus:

\begin{lemma} The pair $(E,H)$ constructed above is a discrete EALA of nullity $n$. \end{lemma}
\sm

\noindent There is however no analogue of
Prop.~\ref{n:eala-def:prop} and Th.~\ref{n:ealaone:th}: There are
many more EALAs of nullity $n\ge 2$. We have just seen the ``tip of
the iceberg''! Other examples can be found in Ch.III of
\cite{aabgp}, some of them involving heavy-duty nonassociative
algebras, like octonion algebras and Jordan algebras over Laurent
polynomial rings!
 \sm

\begin{exercise}\label{n:ueb4} Supply the missing details of the proof that
$(E,H)$ above is a discrete EALA of nullity $n$. In particular,
prove:

(a) (\ref{n:eq:ealan0}) defines an invariant nondegenerate symmetric
bilinear form on $E$. \sm

(b) The root spaces of $(E,H)$ and the anisotropic and null roots
are as stated in (\ref{n:eq:ealan1}) and (\ref{n:eq:ealan2}).
\end{exercise}

%
%
\section{The structure of the roots of an EALA}\label{n:sec:roots}

In this chapter we will describe the structure of the set of roots
$R$ of an EALA $(E,H)$, defined in (\ref{n:eq:eala-def4}). We have
already seen some examples: $R$ can be a finite irreducible reduced
root system, (\ref{n:sec:eala01}), $R$ can be an affine root system
(\ref{n:eq:ealaone3}), i.e., the set of roots of an affine Kac-Moody
Lie algebra, or $R$ can be of the form $R=S \times \ZZ^n$ where
$S\setminus \{0\}$ is a finite irreducible reduced root system
(\ref{n:eq:ealan2}). Thus any description of the general case has to
encompass all these different examples. \sm

It turns out that the roots of an EALA form an extended affine root
system and that the latter is naturally described as a special case
of affine reflection systems. We therefore first introduce the
latter, describe their structure and then specialize later to
extended affine root systems. Affine reflection systems are
themselves special cases of reflection systems, whose theory is
developed in \cite{prs}.

\subsection{Affine reflection systems: Definition} \label{n:sec:ars}

Throughout this section we work with a triple $(R,X,\inpr)$ where
\begin{itemize}

\item  $X$ is a finite-dimensional vector space
over a field $F$ of characteristic $0$,

\item $\inpr$ is a symmetric bilinear form on $X$ and

\item $ R \subset X$.
\end{itemize}
\noindent For any such triple $(R,X,\inpr)$ we define
\begin{align}
 X^0 &= \{ x \in X : (x\mid X)= 0\}, \hbox{ the radical of } \inpr,
   \nonumber \\
 R^0 &= \{\al\in R : (\al | \al) =0\},  \quad (\hbox{\it null roots})\nonumber \\
R\an & = \{\al\in R : (\al | \al) \ne 0\}, \quad(\hbox{\it anisotropic roots})\nonumber  \\
\lan x, \al\ch\ran &= 2 \frac{(x|\al)} {(\al | \al)}, \quad
    (\hbox{$x\in X$ and $\al \in R\an$})  \nonumber \\
s_\al (x) &= x -\lan x,\al\ch\ran \al. \label{n:sec:ars1}
\end{align}

\noindent By definition we therefore have $R=R^0 \cup R\an$. The map
$s_\al :X \to X$ is a \textit{reflection in $\al$}, i.e.,
$s_\al^2=\Id_X$ and $\{x\in X : s_\al(x) = -x\} = F \al$. It is also
orthogonal with respect to $\inpr$: $(s_\al(x) \mid s_\al(y)) =
(x\mid y)$ for all $x,y\in X$. \sm

We call $(R,X,\inpr)$, or just $R$ for short, an \textit{affine
reflection system} if \begin{description} {\it

\item[\bf (AR1)] $0\in R$ and $R$ spans $X$,

\item[\bf (AR2)] $s_\al(R)=R$ for all $\al \in R\an$,

\item[(AR3)] for every $\al \in R\an$ the set $\lan R, \al\ch\ran$
is finite and contained in $\ZZ$, and

\item[(AR4)] $R^0 = R \cap X^0$.}
\end{description}
An affine reflection system is said to be \begin{itemize}

 \item{} \textit{reduced} if for every $\al \in R\an$ and $c\in
F$:  $ c\al \in R\an \iff c=\pm 1 $,

\item{} \textit{connected} if for any decomposition $R\an = R_1
\cup R_2$ with $(R_1 \mid R_2) = 0$ we have $R_1 = \emptyset$ or
$R_2 = \emptyset$.
\end{itemize}
The \textit{nullity\/} of $(R,X)$ is the rank of the torsion-free
abelian group $\ZZ[R^0]=\Span_\ZZ(R^0)$ generated by $R^0$ in
$(X,+)$. Thus, by definition, $$
 \text{nullity of  } (R,X) = \dim_\QQ( \ZZ[R^0]\ot_\ZZ \QQ)
    = \dim_F ( \ZZ[R^0] \ot_\ZZ F).$$
Since the vector space $\ZZ[R^0]\ot_\ZZ F$ maps onto $\Span_F(R^0)$,
the nullity of $(R,Z)$ is bounded below by $\dim_F \Span_F(R^0)$. It
is in general not equal to it. But this is of course so for nullity
$0$: $(R,X)$ has nullity $0$ if and only if $R^0=\{0\}\iff \dim_F
\Span_F(R^0) = 0$.\sm

\begin{remarks} - For a large part of the theory it is not necessary that $X$ be
finite-dimensional, see \cite{prs}. But assuming this right from the
start, simplifies the presentation.

- We need the bilinear form $\inpr$ to define $R^0$ and the
reflections. But although we will sometimes write $(R,X,\inpr)$, we
will not consider $\inpr$ as part of the structure of an affine
reflection system. For example, in the definition of an isomorphism
below we will not require that the bilinear forms are preserved. See
\cite{prs}, where this point of view is emphasized.

- The requirement $0\in R$ is in line with the previous chapter, in
which $0$ was considered a root of an EALA. This conflicts with the
traditional approach to root systems in which $0$ is not a root, see
for example \cite{brac}, \cite{hum} or \cite{kac}. The question
whether $0$ is a root or is not a root, has lead to heated debates.
In the author's opinion, there are some advantages of considering
$0$ as a root, which however can only be fully seen when one
develops the  theory for affine reflection systems. But perhaps the
reader can be convinced by the natural) example $(R,X)=(\{0\},
\{0\})$ of an affine reflection system.

- The condition $\lan \be, \al\ch\ran \in \ZZ$ in axiom (AR3) makes
sense since every field of characteristic $0$ contains (an
isomorphic copy of) the field of rational numbers, which allows us
to identify $\ZZ \equiv \ZZ 1_F$.

- By definition $\lan X^0, \al\ch\ran=0$ for all $\al\in R\an$.
Hence $s_\al(x^0) = x^0$ for $x^0 \in X^0$. Also, the inclusion $R
\cap X^0 \subset R^0$ in (AR4) is always true. Therefore the axioms
(AR2)--(AR4) can be replaced by the following conditions \sm

\begin{tabular}{cl}
   (AR2)$'$  & $s_\al(R\an) = R\an$ for all $\al \in R\an,$\\
    (AR3)$'$ & for every $\al \in R\an$ the set $\lan R\an,\al\ch\ran\subset \ZZ$ is finite, \\
   (AR4)$'$ & $R^0 \subset X^0.$
\end{tabular}
\sm

\noindent This new set of axioms makes it (even more) clear that the
conditions on $R^0$ are rather weak: We (may) need $R^0$ to span $X$
from (AR1), we need $R^0 \subset X^0$ for (AR4)$'$ and we need $0\in
R$, which is no condition since one can always add $0$ to $R^0$. We
will see this phenomena re-appearing in the examples, e.g., in
Example~\ref{n:arsexreal}, and in the definition of an extension
datum in \ref{n:arsed}.

- The definition of a connected affine reflection system is the same
as the axiom (EA4) in the definition of an EALA.

- The definition of an affine reflection system given in \cite{prs}
is not the same as the one given here. The equivalence of two
definitions follows from \cite[Prop.~5.4]{prs}. \end{remarks} \sm

An \textit{isomorphism\/} from an affine reflection system
$(R,X,\inpr)$ to another affine reflection system $(R',X'\inpr')$ is
a vector space isomorphism $f: X \to X'$ satisfying
$$
     f(R\an)=R'{}\an  \quad\hbox{and} \quad  f(R^0) = R'{}^0
$$
If such a map exists, $(R,X,\inpr)$ and $(R',X,\inpr')$ are called
\textit{isomorphic\/}. One can show, as a corollary of the Structure
Theorem~\ref{n:arsstructh}, that an isomorphism $f$ also satisfies $
     f \circ s_\al = s_{f(\al)} \circ f \quad \hbox{for all
          $\al \in R\an$,}
$ equivalently, \begin{equation*} 
\lan x,\al\ch\ran = \lan f(x), f(\al)\ch\ran \end{equation*} for all
$x\in X$ and $\al \in R\an$. This is always fulfilled if $f$ is an
isometry for $\inpr$ and $\inpr'$ respectively. But in general an
isomorphism is not necessarily an isometry. For example, one can
always multiply the bilinear form $\inpr$ by a non-zero scalar
without changing $\lan x, \al\ch\ran$.

Since a reflection $s_\al$ is an isometry, it follows from (AR2) and
(AR4) that $s_\al$ leaves $R\an$ and $R^0$ invariant and is thus an
automorphism of $(R,X)$. The subgroup $W(R)$ of the automorphism
group of $(R,X)$ generated by all reflections $s_\al$, $\al \in
R\an$, is (obviously) called the \textit{Weyl group\/} of $(R,X)$.
(It will not play a big role in this chapter.)

\subsection{Examples of affine reflection systems} \label{n:sec:arsex}

We will now give some immediate examples of affine reflection
systems.

\begin{example}[\textit{The real part of an affine reflections system}] \label{n:arsexreal}
Let $(R,X,\inpr)$ be an affine reflection system. Then
\begin{align*}
 \Rea(R) &= \{0\} \cup R\an, \quad \Rea(X) = \Span_F(R\an), \\
 \inpr_{\Rea} &= \inpr_{\Rea(X) \times \Rea(X)}
\end{align*}
defines an affine reflection system, called the \textit{real part of
$(R,X)$}, with
$$
    \Rea(R)\an = R\an, \quad \Rea(R)^0 = \{0\},
$$
in particular $\Rea(R)$ has nullity $0$.

Observe that $\inpr_{\Rea}$ need not be nondegenerate, see
Example~\ref{n:arsexn} for an example.

The fact that one can ``throw away'' the non-zero null roots and
still have an affine reflection system indicates that one has little
control over the null roots in a general affine reflection system.
This will be made even more evident in the concept of an extension
datum \ref{n:arsed}, used in the general Structure Theorem
\ref{n:arsstructh} for affine reflection systems. It is therefore
natural to define subclasses of affine reflection system by imposing
conditions on the null roots. For example, we will do so when we
define extended affine root systems in \ref{n:sec:ears}.

In \cite[3.6]{n:persp} the author claimed that an affine reflection
system of nullity $0$ is a finite root system. The example above
show that this is far from being true. But what remains true is the
converse, also claimed in \cite[3.6]{n:persp}: A finite root system
is an affine reflection system of nullity $0$, as we will show now.
\end{example}

\begin{example}[\textit{Finite root systems}] \label{n:arsexfin} Let $\Phi$ be a (finite)
root system \`a la Bourbaki \cite[VI, \S1.1]{brac}. Recall that this
means that $\Phi$ is a subset of an $F$-vector space $Y$ satisfying
the axioms (RS1)--(RS3) below.
\begin{description}

\item[(RS1)] $\Phi$ is finite, $0\not\in \Phi$ and $\Phi$ spans $Y$.

\item[(RS2)] For every $\al \in \Phi$ there exists a linear form
$\al\ch \in Y^*$ such that $\al\ch(\al) = 2$ and $s_\al(\Phi) =
\Phi$, where $s_\al $ is the reflection of $Y$ defined by $s_\al (y)
= y - \al\ch(y)\al$,

\item[(RS3)] for every $\al \in \Phi$ the set $\al\ch(\Phi)$
is contained in $\ZZ$.
\end{description}
Observe that the reflection $s_\al$ defined in (RS2) satisfies
$s_\al(\al) = - \al$ and $s_\al(y)= y$ for $\al\ch(y)=0$. It
therefore seems to depend on $\al$ and the linear form $\al\ch$.
However, since $\Phi$ is finite, there exists at most one reflection
$s$ with $s(\Phi)=\Phi$ and $s(\al) = - \al$ (\cite[VI, \S1.1,
Lemme~1]{brac}). It is therefore not necessary to indicate $\al\ch$
in the notation of $s_\al$. \sm

Note that we do not assume that $\Phi$ is reduced. This more general
concept of a root system is necessary for the Structure Theorem of
affine root systems (\ref{n:arsstructh}). The reader who is only
familiar with the theory of reduced finite root systems, as for
example developed in \cite[Ch.~III]{hum}, can perhaps be comforted
by the fact that the difference is not very big. Indeed, every
finite root system is a direct sum of connected (= irreducible) root
systems and there is only one irreducible non-reduced root system of
rank $l$, namely
$$\rmbc_l = \rmb_l \cup \rmc_l = \{ \pm \veps_i : 1 \le i \le l\} \cup
\{ \pm \veps_i \pm \veps_j : 1\le i,j\le l\}$$ where here and in the
following $\veps_1, \ldots, \veps_l$ is the standard basis of $F^l$.
(Note $0\in \rmbc_l$ in anticipation of the convention introduced
below.)\sm

In the context of finite-dimensional Lie algebras, non-reduced root
systems arise naturally as the roots of a finite-dimensional
semisimple Lie algebra $L$ with respect to a maximal
$\ad$-diagonalizable subalgebra $H\subset L$ which is not
self-centralizing, hence not a Cartan algebra. In particular,
non-reduced root systems do not occur over an algebraically closed
field. However, they do occur in the context of infinite-dimensional
Lie algebras, even over algebraically closed fields, see
Ex.~\ref{n:arsexone}. \sm

Given a finite root system $(\Phi,Y)$, define
\begin{equation}\label{n:arsexfin1}
 S=\{0\} \cup \Phi \quad\hbox{and} \quad
(x\mid y) = \tsum_{\al\in \Phi}\, \al\ch(x) \, \al\ch(y)
\end{equation} for $x,y\in Y$. Then $\inpr$ is a nondegenerate
symmetric bilinear form on $Y$ with respect to which all reflections
$s_\al$ are isometric (\cite[VI, \S1.1, Prop.~3]{brac}). Moreover,
$(\al \mid \al)$ is a positive integer for every $\al \in \Phi$
(viewing $\QQ \subset F$ canonically) and
$$
   \lan y,\al\ch \ran = \al\ch(y) = 2 \, \textstyle \frac{
(y\mid \al)}{(\al\mid \al)}$$ for all $y\in Y$. Hence $s_\al$ as
defined in (RS2) is also given by the formula (\ref{n:sec:ars1}). We
have $S^0=\{0\} = X^0 = X^0 \cap S$. Since $\lan \Phi, \al\ch\ran
\subset \ZZ$ we have shown that
$$
 \textit{$(S,Y,\inpr)$ as defined in {\rm (\ref{n:arsexfin1})}
is a finite affine reflection system of nullity $0$\/}.$$ We will
characterize finite root systems within the category of affine
reflection systems in Cor.~\ref{n:arsclass0}. \sm

In the following we will always assume that a finite root system
contains $0$. We will usually use the symbol $S$ for a finite root
system, and put
$$   S^\times = S\setminus \{0\} = \Phi.
$$
We will also need the following subsets of roots of a finite root
system $S$:
\begin{itemize}

\item[$S\div$] is the set of divisible roots, where $\al \in S$ is
called \textit{divisible\/} if $\al/2\in S$. In particular $0\in
S\div$. We put $S\div^\times = S\div \cap S^\times = S\div \setminus
\{0\}$.

\item[$S\ind$] $= S\setminus S\div^\times$, the subsystem of
\textit{indivisible roots}.
\end{itemize}
We also need the fact that there exists a unique symmetric bilinear
form $\inpr_u$ on $Y$ which is invariant under the Weyl group $W(S)$
and which satisfies $2\in \{ (\al | \al)_u : 0 \ne \al \in C\}
\subset \{2, 4, 6, 8\}$ for every connected component $C$ of $S$.
This follows easily from \cite[Prop.~7]{brac}. Observe that
$$S\div^\times = \{ \al \in S : (\al |\al)_u=8\}. $$ We use $\inpr_u$
to define short and long roots:
\begin{itemize}
\item[$S\sh$] $= \{ \al \in S : (\al | \al)_u = 2\}$ is the set of
\textit{short roots\/}.

\item[$S\lg$] $= \{ \al \in S : (\al | \al)_u \in \{4, 6\}\}$ is the set of
\textit{long roots} in $S$.
\end{itemize}
Thus $S\lg = S \setminus (S\sh \cup S\div)$. For example, for
$S=\rmbc_l$ we have
\begin{align*}
\rmbc_{l, {\rm sh}} &= \{ \pm \veps_i : 1 \le i \le l\}, \\
\rmbc_{l, {\rm div}}^\times &= \{ \pm 2\veps_i : 1\le i \le l\}, \\
\rmbc_{l, {\rm lg}} &= \{ \pm \veps_i \pm \veps_j : 1 \le i \ne j
\le l\}, \end{align*} in particular $\rmbc_{1, {\rm lg}} =
\emptyset$, and if $S$ is \textit{simply laced}, i.e., $S^\times =
S\sh$, then $S\div = \{0\}$ and $S\lg = \emptyset$.
\end{example}

\begin{example}[\textit{Untwisted affine reflection systems}]
\label{n:arsexn} Let $(S, Y, \inpr_Y)$ be a finite root system.
Hence $0\in S$ and $\Phi=S\setminus \{0\}$, as stipulated in
Example~\ref{n:arsexfin}. Also, let $Z$ be an $n$-dimensional
$F$-vector space, say with a basis $\veps_1, \ldots, \veps_n$. We
define
\begin{align*}
    X &= Y \oplus Z, \\
    \La &= \ZZ \veps_1 \oplus \cdots \oplus \ZZ \veps_n \; \subset Z, \\
    R &= \textstyle \bigcup_{\xi \in S} \{ \xi \oplus \la : \la \in \La \}
           \; \subset Y \oplus Z,\\
   (x_1 \mid x_2)_X &= (y_1 \mid y_2)_Y \quad
 \hbox{for $x_i= y_i \oplus z_i$ with $y_i \in Y$ and $z_i \in Z$.}
\end{align*}
By construction we then have
$$
   X^0 = Z, \quad R^0 = \La, \quad R\an = \textstyle
         \bigcup_{\xi\in \Phi} \, \xi \oplus \La
$$
where of course $\xi \oplus \La = \{ \xi \oplus \la : \la \in
\La\}$. For $\al= \xi \oplus \la \in R\an$ with $\xi\in S$ and $\la
\in \La$ the reflection $s_\al$ satisfies
\begin{equation} \label{n:eq:arsexn1}
  s_\al(y \oplus z) = s_\xi(y) \oplus (z - \lan y, \xi\ch\ran \la).
\end{equation}
We will leave it to the reader to verify that
\begin{equation} \label{n:arsexn2}
\hbox{\it $(R,X)$ is an affine reflection system of nullity $n$.}
\end{equation}
Observe that $(R,X)$ is the set of roots of the EALA constructed in
\ref{n:sec:ealan}, see in particular (\ref{n:eq:ealan2}). \sm

Observe that $\Span_F(R\an) = X= \Rea(X)$ in case $S\ne \{0\}$. This
shows that the form $\inpr_{\Rea}$ of the real part $\Re(R)$ of $R$
need not be nondegenerate.\end{example}

\begin{exercise} Show the claim in (\ref{n:arsexn2}), and also
that $(R,X)$ is reduced resp. connected if and only if $(S,Y)$ is
so. \end{exercise}

\begin{example}[\textit{Affine root systems}] \label{n:arsexone}
By definition, an \textit{affine root system} is the set of roots of
an affine Kac-Moody Lie algebra, which we studied in
\S\ref{n:sec:aff} and then again in \S\ref{n:sec:ealaone}, where we
showed that an affine Kac-Moody algebra is an EALA of nullity $1$.
Our goal here is not surprising. We want to show that
\begin{equation} \label{n:arsexone1} \hbox{\it an affine root system
is an affine reflection system of nullity $1$.}
\end{equation}
Let us first collect the data necessary to prove this. We use the
notation established in \ref{n:sec:ealaone}. Thus, $\hat \scL = \hat
\scL(\g, \si)$ is an affine Kac-Moody Lie algebra over $\CC$, $\si$
is a diagram automorphism of the simple finite-dimensional Lie
algebra $\g$ of order $m\in \{1,2,3\}$, and $\De_{\bar s}$ denotes
the set of weights of the $(\g_{\bar 0}, \frh_{\bar 0})$-module
$\g_{\bar s} \subset \g$, $s=0, \ldots , m-1$. One knows that
$\De_{\bar 0}$ is a reduced irreducible root system in $\frh_{\bar
0}^* =:Y$. The roots of $\hat \scL$ with respect to $H=\frh_{\bar 0}
\oplus \CC c \oplus \CC d$ are
$$
  R=\{ \ga \oplus n\de : \ga \in \De_{\bar s}, \, \bar n = \bar s,
            \,   0 \le s < m\},
$$
see (\ref{n:eq:ealaone3}), hence
$$  X=\Span_\CC (R) = Y \oplus \CC \de.$$
The bilinear form $\inpr_X$ used to determine the (an)isotropic
roots in $R$ has the form
$$ (x_1 \mid x_2)_X = (y_1 \mid y_2)_Y$$
where $x_i = y_i \oplus a_i \de$ with $y_i \in Y$ and $a_i \in \CC$,
and where $\inpr_Y$ is the nondegenerate symmetric bilinear form on
$Y$, obtained by transporting the Killing form $\ka\mid_{\frh_{\bar
0} \times \frh_{\bar 0}}$ from $\frh_{\bar 0}$ to $Y$. It follows
that
$$  X^0 = \CC \de \quad \hbox{and} \quad R\an=
         \{ \ga \oplus n \de \in R: \ga\ne 0\}.$$
We can now verify the axioms (AR1)--(AR4). \sm

(AR1) holds by definition. (AR2) is a consequence of
\cite[Prop.~3.7(b)]{kac}. Concerning (AR3), it follows from the
structure of $\inpr_X$ that \begin{equation} \label{n:arsexone3}
 \lan x, \al\ch \ran = \lan y, \ga\ch\ran \quad \hbox{for
     $x=y\oplus a \de\in X$ and $\al = \ga \oplus n \de \in R\an$.}
\end{equation} This implies that $\lan R, \al\ch\ran $ is a finite set since
$$
   S = \De_{\bar 0} \cup \cdots \cup \De_{\overline{m-1}}
$$
is a finite set ($S$ is actually a finite root system; for $m>1$ see
the table below). Moreover $\lan R, \al\ch\ran \subset \ZZ$ because
$\hat \scL$ is an integrable $\hat \scL$-module
(\cite[Lemma~3.5]{kac}). Thus (AR3) holds, and (AR4) follows from
(\ref{n:arsexone3}) and $(\ga \mid \ga) = 0 \Leftrightarrow \ga=0$
for $\ga \in S$. This proves (\ref{n:arsexone1}). \sm

To motivate the definition of extension data in Def.~\ref{n:arsed}
and the Structure Theorem \ref{n:arsstructh} for affine reflection
systems, we will now look at $R$ and $S$ more closely. In the
untwisted case, i.e., $m=1$, we have of course
$$  \De_{\bar 0} = S , \quad R= S \times \ZZ \de \quad (m=1).
$$
Thus $R$ is an untwisted affine root system of nullity $1$, a
special case of the Example~\ref{n:arsexn}. For $m=2,3$ the
structure of $\De_{\bar s}$ and $S$ is summarized in the table
below. Proofs can be extracted from \cite[7.9, 7.8, 8.3]{kac}.
\begin{equation} \label{n:table1} {\renewcommand{\arraystretch}{1.5}
\begin{array}{|c|| c| c| c |} \hline
 (\g, m) &  \De_{\bar 0} & \De_{\bar 1} &  S \\ \hline \hline
  (\rma_{2l}, 2), l\ge 1  &  \rma_1 \text{ or } \rmb_l (l \ge 2)
    &  \De_{\bar 0} \cup \{\pm 2\veps_i : 1 \le i \le l \} & \rmbc_l \\ \hline
 (A_{2l-1}, 2), l\ge 2   & \rmc_l & \{0\} \cup \rmc_{l, {\rm sh}} & \rmc_l \\ \hline
(\rmd_{l+1}, 2), l\ge 3 & \rmb_l& \{0\} \cup \rmb_{l, {\rm sh}} & \rmb_l \\
\hline
 (\rme_6,2) & \rmf_4 & \{0\} \cup \rmf_{4, {\rm sh}}  & \rmf_4 \\ \hline
 (\rmd_4, 3) &  \rmg_2 & \{0\} \cup \rmg_{2, {\rm sh}}  & \rmg_2 \\ \hline
\end{array}}
\end{equation} We can now rewrite $R$. For subsets $T \subset S$ and $\Xi
\subset \ZZ \de$ we put $$T \oplus \Xi = \{ \ta \oplus n \de : \ta
\in T, \, n\de \in \ZZ \de\}$$ and abbreviate $\rmb_1 = \rma_1 =
\{0,\pm \al\}$. For $m=2$ we get
\begin{align*}
  R &= (\De_{\bar 0} \oplus 2 \ZZ \de) \cup
       \big(\De_{\bar 1} \oplus (1 + 2\ZZ) \de\big) \\
 &= \begin{cases}
    (\{0\} \oplus \ZZ \de) \cup \big( (\rmb_l \setminus \{0\} \oplus \ZZ \de\big)
            \cup
          \big(\rmbc^\times \div \oplus (1 + 2 \ZZ) \de\big),
 & \g = \rma_{2l} \\
   (\{0\} \oplus \ZZ \de) \cup \quad (S\sh \oplus \ZZ \de) \quad \cup
      \quad   (S\lg \oplus 2 \ZZ \de ),  &\g \not = \rma_{2l}.
 \end{cases}
 \end{align*}
For $(\g, m)=(\rmd_4, 3)$ one knows $\De_{\bar 1} = \De_{\bar
2}=\{0\} \cup \rmg_{2, {\rm sh}}$, whence \begin{align*}
  R&=(\De_{\bar 0} \oplus 3 \ZZ \de) \cup
       \big(\De_{\bar 1} \oplus (1 + 3\ZZ) \de\big)  \cup \big(
         \De_{\bar 2}\oplus (2 + 3\ZZ)\de\big) \\
  &= (\{0\} \oplus \ZZ \de) \cup (S\sh \oplus \ZZ \de) \cup
         (S\lg \oplus 3\ZZ\de).
\end{align*}
In all three cases $R$ has a simultaneous description in terms of
the root system $S$ and subsets $\La\sh, \La\lg, \La\div \subset \ZZ
\de$ as
\begin{equation} \label{n:arsexone4}
 R = R^0 \cup (S\sh \oplus  \La\sh) \cup (S\lg \oplus \La\lg) \cup
        (S\div^\times \oplus \La\div)
\end{equation}
where 
\begin{equation} \label{n:arsexone5}
 \La\sh = \ZZ \de = R^0, \quad \La\div  = (1 +2 \ZZ) \de, \quad
   \La\lg =  \begin{cases} \ZZ\de, & \g=\rma_{2l}, \, m=2, \\
                            2\ZZ \de, & \g \ne \rma_{2l}, \, m=2, \\
                            3\ZZ \de, & m=3. \end{cases}
\end{equation}
If we define $\La_\xi$ for $\xi \in S$ by $\La_\xi \in \{ \La_0=
R^0, \La\sh, \La\lg, \La\div\}$ according to $\xi$ belong to the
corresponding subset of $S$, then (\ref{n:arsexone4}) becomes
\begin{equation}\label{n:arsexone7}
  R = \textstyle\bigcup_{\xi \in S} \xi \oplus \La_\xi.
\end{equation}
Note that we also recover \cite[Th.~5.6(b)]{kac}:
$$   R \cap X^0 = R \cap \ZZ \de.
$$

\end{example}

\begin{example}[\textit{Type $\rma_1$ generalized}] \label{n:arsexaone}
We consider a final example of an affine reflection system to
motivate the definition of an extension datum in \ref{n:arsed}
below. \sm

Let $Z$ be a finite-dimensional $F$-vector space and define the
vector space $X$ and a symmetric bilinear form on $X$ by
$$ X= F \al \oplus Z, \quad (a_1 \al \oplus z_1 \mid a_2 \al \oplus
z_2) = a_1 a_2,
$$
where $0\ne \al$ and $a_i \in F$. We define $R\subset X$ in terms of
three non-empty subsets $\La_0, \La_\al, \La_{-\al} \subset Z$ as
follows: \begin{equation} \label{n:arsexaone0}     R = \La_0 \cup
(\al \oplus \La_\al) \cup (-\al \oplus \La_{-\al}). \end{equation}
It is then immediate that
$$
   X^0 = Z, \quad R^0 = \La_0, \quad
  R\an = (\al  \oplus \La_\al) \cup ( -\al  \oplus \La_{-\al}).
$$
We will now discuss under which conditions $(R,X,\inpr)$ is an
affine reflection system. Let us start with (AR2). For $s_i \in
\{\pm 1\}$, $\mu \in \La_{s_1 \al}$ and $\la \in \La_{s_2 \al}$ we
have
\begin{align*}
\lan s_1 \al \oplus \mu, (s_2 \al \oplus \la)\ch\ran &= 2 \, s_1s_2,
\\
 s_{s_2 \al \oplus \la} (s_1 \al \oplus \mu) &= -s_1 \al \oplus
   (\mu - 2 s_1 s_2 \la)
\end{align*}

Hence, all reflections $s_{s_2\al \oplus \la}$ leave $R$ invariant
if and only if $\mu - 2 s_1 s_2 \la \in \La_{-s_1 \al}$ for
$\mu,\la$ as above, i.e., in obvious short form
$\La_{s_1\al} - 2 s_1 s_2 \La_{s_2 \al} \subset \La_{-s_1 \al}$.
In particular,
$$ \begin{tabular}{lcc}
    $\La_\al - 2\La_\al \subset \La_{-\al},$
       & $\La_{-\al} - 2 \La_{-\al} \subset \La_\al$ & for $s_1=s_2$, \\
 $\La_{-\al} + 2 \La_\al \subset \La_\al$ &  for $s_1 = -1 = - s_ 2.$
\end{tabular}
$$
For $\la \in \La_\al$ we therefore get $\la - 2 \la = - \la \in
\La_{-\al}$, whence $-\La_\al \subset \La_{-\al}$ and, analogously,
$\La_{-\al} \subset -\La_\al$. We therefore obtain
\begin{equation} \label{n:arsexaone2}
\La_{-\al} = - \La_\al,  \end{equation} or with the notation of
above $\La_{s_1 \al} = s_1 \La_\al$. It is now easy to see that
(AR2) is equivalent to the two conditions (\ref{n:arsexaone2}) and
\begin{equation} \label{n:arsexaone4}
2\La_\al - \La_\al  \subset \La_\al. \end{equation} It then follows
that $R$ is an affine reflection system if and only if
\begin{enumerate}
 \item[(i)] (\ref{n:arsexaone2}) and (\ref{n:arsexaone4}) hold,
 \item[(ii)] $0\in \La_0$, and
 \item[(iii)] $Z= \Span_F ( \La_0 \cup \La_\al \cup \La_{-\al})$.
\end{enumerate}
Observe the similarity with the previous examples: $R$ has the form
$$R= \textstyle \bigcup_{\xi \in S} \xi \oplus \La_\xi$$ where $S=\{0,
\pm \al\}$ is a finite root system and $(\La_\xi : \xi \in S)$ is a
family of subsets in $X^0$. However, in the previous examples the
$\La_\xi$ were subgroups of $(Z,+)$ while here we only have the
condition (\ref{n:arsexaone4}). Does this imply that $\La_\al$ is a
subgroup? The answer is no! For example, in $Z=F$ the subset
$\La_\al = 1 + 2 \ZZ \subset F$ satisfies (\ref{n:arsexaone4}). \sm

A subset $A$ of an abelian group $(Z,+)$ is called a
\textit{reflection subspace} if $2a_1-a_2 \in A$ for all $a_i \in A$
(see \cite{l:sp} or \cite[3.3]{n:persp} for a justification for this
terminology). Hence, (\ref{n:arsexaone4}) just says that $\La_\al$
is a reflection subspace. The structure of two special types of
reflection subspaces is described in Exercise~\ref{n:arsexlem}
below. \sm

While in general $\La_\al$ is far from being a subgroup, one can
always ``re-coordina\-tize'' $R$ to at least get $0\in \La_\al$.
Namely, for a fixed $\la \in \La_\al$ we have $\al + \La_\al = (\al
+ \la) + (\La_\al - \la)$. Hence, with $\tilde \al = \al + \la$ and
$\La_{\tilde \al} = \La_\al - \la$, we obtain \begin{equation}
\label{n:arsexaone5}
    R=  \La_0 \cup (\tilde \al + \La_{\tilde \al}) \cup
        \big( - (\tilde \al + \La_{\tilde \al}) \big),
\end{equation} where now $\La_{\tilde \al}$ not only satisfies
(\ref{n:arsexaone4}) but also $0\in \La_{\tilde \al}$. In other
words, $\La_{\tilde \al}$ is a pointed reflection subspace as
defined in Lemma~\ref{n:arsexlem} and therefore also satisfies
$\La_{\tilde \al} = - \La_{\tilde \al}$. \sm

The process of re-coordinatization works well in this example. The
reason is that the finite root system $S$ in (\ref{n:arsexaone5}) is
reduced. Re-coordinatization will not work if $S$ is not reduced, as
for example in the case $(\g,m)=(\rma_{2l},2)$ of
Example~\ref{n:arsexone}. This ``explains'' why in the property
(ED2) of an extension datum in \ref{n:arsed} we require $0\in
\La_\xi$ only for an indivisible root $\xi \in S$.
\end{example}

\begin{exercise} \label{n:arsexlem}
Let $A$ be a subset of an abelian group $(Z,+)$. As above we put
$2A-A = \{2a_1-a_2 : a_i \in A\}$. We denote by $\La=\Span_\ZZ (A)$
the $\ZZ$-span of $A$ in $Z$. A subset $A\subset Z$ is called
\emph{symmetric\/} if $A=-A$.
\begin{itemize}

\item[(a)] The following equivalent conditions characterize
symmetric reflection subspaces $A \subset Z$:
\begin{itemize}

\item[(i)]  $2A - A \subset A$ and $A=-A$,

\item[(ii)]  $2\la + a \in A$ for every $\la\in \La$ and $a\in A$,

\item[(iii)] $A$ is a union of cosets modulo $2\La$,

\item[(iv)] $a_1 - 2 a_2 \in A$ for all $a_i \in A$.
\end{itemize}
\sm

\item[(b)] The following are equivalent for $A \subset Z$: \begin{enumerate}

\item[(i)] $0\in A$ and $A - 2A \subset A$,

\item[(ii)] $0\in A$ and $2A - A \subset A$,

\item[(iii)] $2\ZZ[A]\subset A$ and $2\ZZ[A] - A \subset A$,

\item[(iv)] $A$ is a union of cosets modulo $2\ZZ[A]$, including the
trivial coset $2\ZZ[A]$.
\end{enumerate}
 In this case $A$ is called a \textit{pointed} reflection subspace.
\sm

\item[(c)] Every pointed reflection subspace is symmetric. \sm

\item[(d)] If $A$ is a symmetric reflection subspace then $A+A$ is
a  pointed reflection subspace.
\end{itemize} \end{exercise}

\subsection{The Structure Theorem of affine reflection systems} \label{n:sec:arstrut}

After the many examples in \ref{n:sec:arsex}, the following
definition should not be too surprising.

\begin{definition} \label{n:arsed} Let $S$ be a finite root system as defined in
\ref{n:arsexfin}. Recall $S^\times = S\setminus \{0\}$ and $S\ind =
\{0\} \cup \{ \al \in S : \al/2 \not\in S\} = S\setminus
S\div^\times$. Also, let $Z$ be a finite-dimensional $F$-vector
space. An \textit{extension datum of type $(S,Z)$\/}, sometimes
simply called an \textit{extension datum\/}, is a family $(\La_\xi :
\xi \in S)$ of subset $\La_\xi \subset Z$ satisfying the axioms
(ED1)--(ED3) below.

\begin{description}

\item[(ED1)] For $\eta, \xi \in S\ti$, $\mu \in \La_\eta$ and $\la
\in \La_\xi$ we have $\mu - \lan \eta, \xi\ch\ran \xi \in
\La_{s_\xi(\eta)}$, in obvious short form
$$
    \La_\eta - \lan \eta, \xi\ch\ran \La_\xi \subset
\La_{s_\xi(\eta)}.
$$

\item[(ED2)] $0\in \La_\xi$ for $\xi \in S\ind$, and $\La_\xi \ne
\emptyset$ for $\xi \in S\setminus S\ind = S\div^\times$.

\item[(ED3)] $Z = \Span_F \big( \bigcup_{\xi \in S} \La_\xi\big) $.
\end{description}
\end{definition}
The axiom (ED1) is trivially true for $\eta=0$ since $\lan \eta,
\xi\ch\ran = 0$ and $s_\xi(0)= 0$. Also, if $S\div^\times =
\emptyset$, then there is no $\La_\xi$ for $\xi\in S\div^\times$ and
so the second condition in (ED2) is trivially fulfilled. (ED3)
simply serves to determine $Z$. If it does not hold, one can simply
replace $Z$ by $\Span_\ZZ(\bigcup_{\xi \in S} \La_\xi)$. \sm

The definition of an extension datum above is a special case of the
notion of an extension datum for a pre-reflection system, introduced
in \cite[4.2]{prs}. (The reader will note that the axiom (ED1) in
\cite{prs} simplifies since in our setting the subset $S\re$ of
\cite{prs} is $S\re= S\an = S\setminus \{0\}$.) The Structure
Theorem~\ref{n:arsstructh} below is proven in \cite[Th.~4.6]{prs}
for extensions of pre-reflection systems. Affine reflection systems
are special types of such extensions, namely finite-dimensional
extensions of finite root systems. \sm

The rationale for the concept of an extension datum is the following
Structure Theorem for affine reflection systems.

\begin{theorem}[\textbf{Structure Theorem for affine reflection systems}]
\label{n:arsstructh} \ \newline \noindent {\rm (a)} Let $(S,Y,
\inpr_Y)$ be a finite root system and let $\frL = (\La_\xi : \xi \in
S)$ be an extension datum of type $(S,Z)$. Define $(R,X,\inpr_X)$ by
\begin{align*}
     X & = Y \oplus Z \\
     R &= \textstyle \bigcup_{\xi \in S} \,  \xi \oplus \La_\xi
     \; \subset \; Y \oplus Z = X, \\
   (y_1 \oplus z_1 \mid y_2 \oplus z_2)_X &= (y_1 \mid y_2)_Y
\end{align*}
for $y_i \in Y$ and $z_i \in Z$. Then $(R,X,\inpr_X)$ is an affine
reflection system, denoted $\scA(S,\frL)$,  with
$$
  R^0 = \La_0, \quad X^0=Z \quad\hbox{and} \quad R\an = \textstyle \bigcup_{0 \ne \xi \in S} \,
     \xi \oplus \La_\xi.
$$
For $\al = \xi \oplus \la \in R\an$ and $x=y \oplus z \in X$ the
reflection $s_\al$ is given by
$$
    s_\al (x) = s_\xi(y) \oplus (z - \lan y, \xi\ch\ran \la)
$$

\noindent {\rm (b)} Conversely, let $(R,X,\inpr_X)$ be an affine
reflection system. \begin{itemize}

\item[\rm (i)] Let $f : X \to X/X^0=:Y$ be the canonical map, put
$S=f(R)$ and let $\inpr_Y$ be the induced bilinear form on $Y$, that
is $(f (x_1) \mid f(x_2))_Y = (x_1 \mid x_2)_X$. Then $(S,Y,
\inpr_Y)$ is a finite root system, the so-called {\rm quotient root
system} of $(R,X)$.

\item[\rm (ii)] There exists a linear map $g : Y \to X$ satisfying
$f\circ g = \Id_Y$ and $g(S\ind) \subset R$.

\item[\rm (iii)] For $g$ as in {\rm (ii)} and $\xi \in S$ define
$\La_\xi \subset  \Ker(f) =: Z$ by \begin{equation}
\label{n:arsstructh1}
   R\cap f^{-1}(\xi) =  g(\xi) \oplus \La_\xi.  \end{equation}
Then $\frL = (\La_\xi : \xi \in S)$ is an extension datum of type
$(S,Z)$.

\item[\rm (iv)] $(R,X)$ is isomorphic to the affine reflection
system $\scA(S,\frL)$ constructed in {\rm (a)}.
\end{itemize}
 \end{theorem}

Let us note that it is not reasonable to expect $g(S) \subset R$ in
(b.ii) above, since $R$ may be reduced while $S$ is not, see for
example the case $(\g,\rma_{2l})$ in \ref{n:arsexone}. The quotient
root system $S$ is uniquely determined, but not so the extension
datum, see \cite[Th.~4.6(c)]{prs}.\sm

\begin{corollary} \label{n:arsclass0} An affine reflection system $(R,X,\inpr)$ is
nondegenerate in the sense that $\inpr$ is nondegenerate if and only
if $R$ is a finite root system.\end{corollary}

\begin{proof} If $(R,X,\inpr)$ is an affine reflection system with
a nondegenerate form $\inpr$, then $\{0\}= X^0 = \Ker f$, so $f$ is
the identity. We have seen the other direction in
Example~\ref{n:arsexfin}.\end{proof}

\begin{corollary}[{\cite[Cor.~5.5]{prs}}] \label{n:arsrealcl}
Let $(R,X,\inpr)$  be an affine reflection system over $F=\RR$. Then
there exists a positive semidefinite symmetric bilinear form
$\inpr_{\ge}$ on $X$ such that $(R,X,\inpr_\ge)$ is an affine
reflection system with the same (an)isotropic roots and reflections.
\end{corollary}
\sm

\noindent The morale of the Structure Theorem is

\begin{center} \begin{tabular}{|lcr|} \hline
 affine reflection system &$\quad = \quad $&
  finite root system $+$ extension datum \\ \hline\end{tabular}
\end{center}
Thus properties of an affine reflection system can be described in
terms of properties of its quotient root system and the associated
extension datum. Some examples of this philosophy are given in the
Proposition~\ref{n:ex:arsstrut3} and the
Exercise~\ref{n:ex:arsstrut1} below.

\begin{proposition}[{\cite[Cor.~5.2]{prs}}] \label{n:ex:arsstrut3} Let $R$ be an affine reflection
system, let $S$ be its quotient root system and let $(\La_\xi: \xi
\in S)$ be the associated extension datum. We define $$\textstyle
\La_{\diff} = \bigcup_{0\ne \xi \in S} \, \La_\xi - \La_\xi.$$ Then
$\ZZ\La_{\diff} = \La_{\diff}$. Moreover:

{\rm (a)} $R$ is {\rm tame} in the sense that $R^0 \subset R\an -
R\an$ if and only if $R^0 \subset \La_{\diff}$.

{\rm (b)} All {\rm root strings} $$\mathbb{S}(\be,\al)=
   R \cap  (\be + \ZZ \al), \quad (\be \in R, \al\in R\an)
$$ are {\rm unbroken}, i.e., $\ZZ(\be,\al)= \{ n\in \ZZ : \be +n\al \in
R\}$ is either a finite interval in $\ZZ$ or equals $\ZZ$, if and
only if $\La_{\diff} \subset R^0$.

{\rm (c)} A tame affine reflection system with unbroken root strings
is symmetric.
\end{proposition}

\begin{exercise} \label{n:ex:arsstrut1} Let $R$ be an affine reflection
system, let $S$ be its quotient root system and let $(\La_\xi: \xi
\in S)$ be the associated extension datum. We use the notation of
Prop.~\ref{n:ex:arsstrut3}. Prove: \sm

\begin{itemize}
\item[(a)] $R$ is reduced if and only if for all $0\ne \xi \in S$ with
$2\xi \in S$ we have $$\La_{2\xi} \cap 2 \La_\xi = \emptyset.$$ In
particular, $S$ need not be reduced for $R$ to be reduced!

\item[(b)] $R$ is connected iff $S$ is connected (= irreducible).

\item[(c)] $R$ is symmetric, i.e., $R=-R$, iff $\La_0$ is symmetric.

\item[(d)] For all $\al\in R\an$ and $\be\in R$ the $\al$-string through
$\be$, i.e., $\mathbb{S}(\be,\al)$ has  length
$|\mathbb{S}(\be,\al)|\le 5$.

\item[(e)] Let $(\al,\be) \in R\an \times R$ and define $d,u\in \NN$ by
put $-d=\min \ZZ(\be,\al)$ and $u = \max \ZZ(\be,\al)$. Then
$u-d=\lan \be,\al\ch\ran$.

\end{itemize}  \end{exercise}
 \sm

We will now describe how the examples of affine reflection systems
of section \ref{n:sec:arsex} fit into the general scheme of the
Structure Theorem above.

\begin{examples} \label{n:arsedexam} (a) Let $\frL=(\La_\xi : \xi \in S)$ be an extension
datum of type $(S,Z)$. Observe that the only conditions on $\La_0$
are $0\in \La_0$ from (ED2) and that $\La_0$ together with the other
$\La_\xi$'s spans $Z$ from (ED3). This is in line with our earlier
observation that one has little control over the null roots $R^0$ of
an affine reflection system. Following the Example~\ref{n:arsexreal}
we define a new extension datum $\Rea(\frL) = (\Rea(\La_\xi) : \xi
\in S)$ of type $(S,\Rea(Z))$ by
$$
   \Rea(Z) = \textstyle \Span_F \big( \bigcup_{0\ne \xi \in S} \La_\xi\big), \quad
  \Rea(\La_\xi) = \begin{cases} \{0\} & \hbox{for $\xi=0$,} \\
                             \La_\xi &  \hbox{for $\xi \ne 0$.}
               \end{cases}
$$
If $\frL$ is the extension datum associated to the affine reflection
system $(R,X)$, then $\Rea(\frL)$ is the extension datum associated
to the affine reflection system $\Rea(R)$. \sm

(b) All $\La_\xi = \{0\}$, whence $Z=\{0\}$, defines a
\emph{trivial} extension datum for any root system $S$. It is
``used'' when we view $S$ as an affine reflection system, as done in
Example~\ref{n:arsexfin}. \sm

(c) Let $\La$ be a subgroup of a finite-dimensional vector space $Z$
such that $\Span_F(\La)\allowbreak =Z$. Then for any finite root
system $S$ the family $(\La_\xi \equiv \La: \xi \in S)$ is an
extension datum of type $(S,Z)$. It is used to construct the
untwisted affine reflection systems of Example~\ref{n:arsexn}. \sm

(d) Let $R$ be an affine root system.  We have seen that $R$ is an
affine reflection system. Its quotient root system $S$ and
associated extension datum $(\La_\xi: \xi \in S)$ are described in
Example~\ref{n:arsexone} using the very same symbols, see the
formulas (\ref{n:arsexone7}) and (\ref{n:arsstructh1}). \sm

(e) The family $\tilde \frL= (\La_0, \La_{\tilde \al}, \La_{-\tilde
\al})$ in Example~\ref{n:arsexaone} is an extension datum, but not
necessarily $\frL=(\La_0, \La_\al, \La_{-\al})$ since $0$ need not
lie in $\La_{\pm \al}$. In fact, replacing $\frL$ by $\tilde \frL$
was the rationale for the re-coordinatization in \ref{n:arsexaone}.
\end{examples}

To describe the classification of affine reflection systems we need
some more properties of the subsets $\La_\xi$ of an extension datum.
They are given in the following exercise (just do it!). Recall from
Exercise~\ref{n:arsexlem} that a reflection subspace $A$ is called
symmetric if $A=-A$ and is called pointed if $0\in A$.

\begin{exercise} \label{n:arsedlem}
Let $(\La_\eta: \eta \in S)$ be an extension datum of type $(S,Z)$.
Show:

\begin{itemize}

\item[(a)] Every $\La_\xi$ for $0\ne \xi \in S$ is a
symmetric reflection subspace and is even a pointed reflection
subspace if $\xi \in S\ind^\times$.

\item[(b)] For $w\in W(S)$, the Weyl group of $S$, we have
\begin{equation} \label{n:arsedlem1} \La_\xi = \La_{w(\xi)}.
\end{equation} In particular, $\La_\xi =
\La_{-\xi} = - \La_\xi$.

\item[(c)]  Whenever $0\ne \xi \in S$ and $2\xi \in S$, then
$$
  \La_{2\xi} \subset \La_\xi .
$$

\item[(d)] $\ZZ\La_\xi \subset \La_\xi$ for $\xi \in S\ind^\times$.
\end{itemize}
\end{exercise}

\sm

Let $(\La_\xi : \xi \in S)$ be an extension datum where $S$ is
irreducible. Then $W(S)$ acts transitively on the roots of the same
length (\cite[VI, \S1.3, Prop.~11]{brac}), i.e., on $\{0\}$, $S\sh$,
$S\lg$ and $S\div$ (some of these sets might be empty). Because of
(\ref{n:arsedlem1}), there are therefore at most four different
subsets $\La_0$, $\La\sh$, $\La\lg$ and $\La\div$ among the
$\La_\xi$, defined by \begin{equation} \label{n:arsed1}
 \La_\xi = \begin{cases}
       \La_0,  & \xi = 0; \\
       \La\sh, & \xi \in S\sh; \\
        \La\lg, & \xi \in S\lg; \\
       \La\div, & \xi \in S\div^\times. \end{cases}
\end{equation}
Of course, $\La\lg$ or $\La\div$ only exists if the corresponding
subset of roots exits. The assertions below referring to $\La\lg$ or
$\La\div$ should be interpreted correspondingly. \sm

We have seen in Exercise~\ref{n:arsedlem} (did you do it?), that the
subsets $\La\sh$ and $\La\lg$ are pointed reflection subspaces and
that $\La\div$ is a symmetric reflection subspace. Assuming only
these properties, does however not give an extension datum, since
only parts of the axiom (ED1) are fulfilled, namely those with $\eta
= \pm \xi$. We also need to evaluate what happens for $\eta \ne \pm
\xi$ with $\lan \eta, \xi\ch\ran \ne 0$. We will do this in the
following examples. \sm

\begin{examples} Let $S$ be an irreducible root system. We suppose that
we are given a pointed reflection subspace $\La\sh$ of a
finite-dimensional vector space, and if $S\lg \ne \emptyset$ or
$S\div^\times \ne \emptyset$ then also a pointed reflection subspace
$\La\lg$ and a symmetric reflection subspace $\La\div$. We define
$\La_\xi$, $\xi \in S$, by (\ref{n:arsed1}) and ask, when is the
family $\frL_{\min}$ defined in this way an extension datum in
$Z=\Span \big(\bigcup_{\xi \in S} \La_\xi \big)$? Note that we only
have to check (ED1). We will consider some examples of $S$. \sm

(a) $S=\rma_1$: In this case there are no further conditions, so
$\frL_{\min}$ describes all possible extension data for $\rma_1$
with $\La_0 = \{0\}$. \sm

(b) $S=\rma_2$: In this case there exists roots $\eta, \xi$ with
$\lan \eta, \xi\ch\ran =1$, namely those for which
$\angle(\eta,\xi)= \frac{\pi}{3}$. Evaluating (ED1) for those gives
$\La\sh - \La\sh \subset \La\sh$, forcing $\La\sh$ to be a subgroup
of $Z$. Thus, $\frL_{\min}$ is an extension datum for $S=\rma_2$ iff
$\La\sh$ is a subgroup.

(c) $S$ simply laced, $\rank S \ge 2$: The argument in (b) works
whenever $S\sh$ contains roots $\eta, \xi$ with $\lan \eta,
\xi\ch\ran =1$. Since this is the case here, we get that
$\frL_{\min}$ is an extension datum iff $\La\sh$ is a subgroup of
$Z$. \sm

(d) $S=\rmb_2= \{\pm \veps_1, \pm \veps_2, \pm \veps_1 \pm
\veps_2\}$:  Here we have pointed reflection subspaces $\La\sh$ and
$\La\lg$. Since non-zero roots of the same length are either
proportional or orthogonal, (ED1) is fulfilled for them. Because
(ED1) is invariant under sign changes, we are left to evaluate the
case of two roots $\eta, \xi$ of different lengths forming an obtuse
angle of $\frac{3\pi}{4}$, for example $\eta = \veps_1$, $\xi =
\veps_2 - \eps_1$.

If $\lan \eta, \xi\ch\ran =-2$, then $\eta$ is long, $\xi$ is short
and so (ED1) becomes $\La\lg + 2 \La\sh \subset \La\lg$. If $\eta$
is short, $\xi$ is long, we have $\lan \eta, \xi\ch\ran = -1$ and
thus get the condition $\La\sh + \La\lg \subset \La\sh$. To
summarize: \emph{$\frL_{\min}$ is an extension datum for $S=\rmb_2$
iff $\La\sh$ and $\La\lg$ are pointed reflection subspaces
satisfying
\begin{equation} \label{n:arsed2}
  \La\lg + 2 \La\sh \subset \La\lg \quad \hbox{and} \quad
    \La\sh + \La\lg \subset \La\sh.
\end{equation}}
Note that (\ref{n:arsed2}) implies $2 \La\sh \subset \La\lg
 \subset \La\sh$.

(e) $S=\rmb_l$, $l\ge 3$. Recall $S = \{ \pm \eps_i : 1\le i \le l\}
\cup \{ \pm \veps_i \pm \veps_j: 1 \le i \ne j \le l\}$. Since the
short roots in $S$ are either proportional or orthogonal, (ED1) is
fulfilled for all short roots $\eta, \xi$. But there exist long
roots $\eta, \xi \in S$ with $\angle(\eta,\xi) = \frac{\pi}{3}$,
whence $\lan \eta, \xi\ch\ran= 1$ and so (ED1) reads $\La\lg -
\La\lg \subset \La\lg$. This forces $\La\lg $ to be a subgroup. As
for $S=\rmb_2$, (ED1) for roots of different lengths leads to the
condition (\ref{n:arsed2}). It is then easy to check that
\emph{$\frL_{\min}$ is an extension datum for $S=\rmb_l$, $l \ge 3$,
iff $\La\sh$ is a pointed reflection subspace, $\La\lg$ is a
subgroup and {\rm ({\ref{n:arsed2}})} holds.}
\end{examples}

Continuing in this way, one arrives at the following.

\begin{theorem}[{\bf Structure of extension data}]
Let $S$ be an irreducible finite root system and define a family
$\frL_{\min}$ as in {\rm (\ref{n:arsed1})} with $\La_0=\{0\}$. Then
$\frL_{\min}$ is an extension datum if and only if $\La\sh$ and
$\La\lg$ are pointed reflection subspaces, $\La\div$ is a symmetric
reflection subspace and the following conditions, depending on $S$,
hold.
\begin{itemize}

\item[\rm (i)] $S$ is simply laced, $\rank S \ge 1:$ No further
condition for $S=\rma_1$, but $\La\sh$ is a subgroup if $\rank S \ge
2$.

\item[\rm (ii)] $S=\rmb_l (l\ge 2)$, $\rmc_l (l\ge 3)$, $\rmf_4:$
$\La\sh$ and $\La\lg$ satisfy $$\La\lg + 2 \La \sh \subset \La\lg
\quad \hbox{and}\quad \La\sh + \La \lg \subset \La\sh.$$ Moreover,

    \begin{itemize}
        \item[$\bullet$] $\La\lg $ is a subgroup if $S=\rmb_l$, $l\ge 3$ or
           $S=\rmf_4$, and

        \item[$\bullet$] $\La\sh$ is a subgroup if $S=\rmc_l$ or $S=F_4$.
    \end{itemize}

\item[\rm(iii)] $S=\rmg_2 :$ $\La\sh$ and $\La\lg$ are subgroups
satisfying
$$ \La\lg + 3 \La \sh \subset \La\lg
\quad \hbox{and}\quad \La\sh + \La \lg \subset \La\sh.$$

\item[\rm (iv)] $S=\rmbc_1:$ $\La\sh$ and $\La \div$ satisfy
$$ \La\div + 4 \La \sh \subset \La\div
\quad \hbox{and}\quad \La\sh + \La \div \subset \La\sh.$$

\item[\rm (v)] $S=\rmbc_l (l\ge 2):$ $\La\sh$, $\La\lg$ and $\La\div
$ satisfy
\begin{center} \begin{tabular}{ccc}
 $\La\lg + 2 \La \sh \subset \La\lg$, &\quad &
  $\La\sh + \La \lg \subset \La\sh,$ \\
 $\La \div + 2 \La\lg \subset \La\div$, &\quad &
     $ \La \lg +\La \div \subset \La\lg$, \\
  $\La\div + 4 \La \sh \subset \La\div$, &\quad &
 $\La\sh + \La \div \subset \La\sh.$
\end{tabular} \end{center}
In addition, if $l\ge 3$ then $\La\lg$ is a subgroup.
\end{itemize}
\end{theorem}

The inclusions $\La\div + 4 \La \sh \subset \La\div$ and $\La\sh +
\La \div \subset \La\sh$ in case (v) above are consequences of the
other inclusions. Since $0$ lies in $\La\sh$ and also in $\La\lg$ if
it exists, the displayed inclusions in the Structure Theorem above
imply
\begin{equation} \label{n:arsed3}
  \La\div \subset \La\lg \subset \La\sh.
\end{equation}

The details of this theorem are given in \cite[II, \S2]{aabgp} for
the special case of extended affine root systems and then in
\cite{y:ext} in general (it follows from the Structure
Theorem~\ref{n:arsstructh} that an affine reflection system is the
same as a ``root system extended by a torsion-free abelian group of
finite rank'' in the sense of \cite{y:ext}). The reference
\cite{aabgp} also contains a classification of discrete extension
data for extended affine root systems of low nullity.

\begin{exercise} Without looking at \cite{aabgp} or \cite{y:ext}, work out some of
the cases above.
\end{exercise}

\subsection{Extended affine root systems} \label{n:sec:ears}

Let us come back to the beginning of this chapter. Our goal was to
describe the structure of the set of roots occurring in an extended
affine Lie algebra. After all the preparations in
\ref{n:sec:ars}--\ref{n:sec:arstrut}, this is now easy.

We start with the same setting as in \ref{n:sec:ars}, i.e., $X$ is a
finite-dimensional vector space over a field $F$ of characteristic
$0$, $R$ is a subset of $X$ and $\inpr$ is a symmetric bilinear form
on $X$. As in \ref{n:sec:ars} we define $R^0 = \{ \al \in R:
(\al|\al)=0\}$, $R\an = \{ \al \in R : (\al |\al) \ne 0\}$ and
$$
\lan x, \al\ch\ran = 2 \frac{(x|\al)} {(\al | \al)}, \quad
    (\hbox{$x\in X$ and $\al \in R\an$}).
$$

\begin{definition} \label{n:earsdef} A triple $(R,X,\inpr)$ as above is
called an  \textit{extended affine root system\/} or EARS for short,
if the following seven axioms (EARS1)--(EARS7) are fulfilled.
\begin{description}

\item[(EARS1)] $0\in R$ and $R$ spans $X$,

\item[(EARS2)] $R$ has \textit{unbroken finite root strings\/}, i.e.,
for every $\al \in R\an$ and $\beta \in R$ there exist $d,u\in
\NN=\{0,1,2, \ldots\} $ such that $$
 \{\be + n \al : n\in \ZZ \} \cap R =
     \{ \be - d \al, \ldots, \be + u \al\} \quad\hbox{and} \quad
d-u = \lan \be,\al\ch\ran.
$$
($d$ stands for ``down'' and $u$ for ``up''.)

\item[(EARS3)] $R^0 = R \cap X^0$.

\item[(EARS4)] $R$ is reduced as defined in \ref{n:sec:ars}:
for every $\al \in R\an$ we have $F\al \cap R\an = \{\pm \al\}$.

\item[(EARS5)]  $R$ is connected in the sense of
\ref{n:sec:ars}: whenever $R\an = R_1 \cup R_2$ with $(R_1\mid R_2)
= 0$, then $R_1 = \emptyset$ or $R_2= \emptyset$.

\item[(EARS6)] $R$ is tame, i.e., $R^0 \subset R\an +
R\an$.

\item[(EARS7)] The abelian group $\Span_\ZZ(R^0)$ is free of finite rank.
\end{description}

In analogy with the concept of discrete EALAs we call $(R,X, \inpr)$
for $F=\CC$ or $F=\RR$ a \textit{discrete} extended affine root
system if (EARS1)--(EARS6) hold and in addition

\begin{description}

\item[(DE)] $R$ is a discrete subset of $X$, equipped
with the natural topology.

\end{description}
\end{definition}
\noindent As for EALAs, a discrete extended affine root system
necessarily satisfies (EARS7), see Proposition~\ref{n:earsstrut}(c)
below, so that it is justified to call it an EARS.

We will immediately connect EARS to affine reflection systems:

\begin{proposition} \label{n:earsstrut}
{\rm (a)} A pair $(R,X)$ satisfying {\rm (EARS1)--(EARS3)} is an
affine reflection system. In particular:
\begin{itemize}
\item[\rm (i)]
An extended affine root system is an affine reflection system which
is reduced, connected, symmetric, tame and which has unbroken root
strings.

\item[\rm (ii)] If $F=\RR$ we can assume that $\inpr$ is positive
semidefinite.
\end{itemize}

{\rm (b)} Let $(R,X)$ be an affine reflection system with quotient
root system $S$ and extension datum $\scL$. Then $(R,X)$ is an
extended affine root system if and only if
\begin{itemize}

\item[\rm (i)] $S$ is irreducible, hence $\scL=(\La_0, \La\sh,
\La\lg, \La\div)$,

\item[\rm (ii)] $\La_0 = \La\sh + \La\sh$, $\La\div \cap 2 \La\sh
= \emptyset$,  and

\item[\rm (iii)] {\rm (EARS7)} holds. \end{itemize}

{\rm (c)} For an extended affine root system $(R,X)$ over $F=\RR$ or
$F=\CC$ the following are equivalent:
\begin{itemize}

\item[\rm (i)] $R$ is discrete;

\item[\rm (ii)] $R^0$ is a discrete subset of $X$;

\item[\rm (iii)] $\Span_\ZZ(R^0)$ is a discrete subgroup of $X$.
\end{itemize}
In this case, all reflection subspaces $\La_\xi$ of\/ {\rm (b)} are
discrete too, and $\Span_\ZZ(R^0)$ is a free abelian group of finite
rank.
\end{proposition}

\begin{proof} (a) Obviously (AR1) = (EARS1) and (AR4) = (EARS3). The axiom
(AR2), i.e., $s_\al(R)=R$, follows from (EARS2): $s_\al(\be) = \be +
(u-d)\al$ and $ -d \le u-d \le u$. Also (AR3) is immediate from
(EARS2).

In any affine reflection system the root strings $R \cap (\be + \ZZ
\al)$ for $(\beta, \al) \in R\times R\an$ are finite. This is
Exercise~\ref{n:ex:arsstrut1}(d) and an immediate consequence of the
Structure Theorem~\ref{n:arsstructh}. Also, a tame affine reflection
system with unbroken roots strings is necessarily symmetric by
Prop.~\ref{n:ex:arsstrut3}. The characterization of an EARS in (i)
is now clear. (ii) follows from Cor.~\ref{n:arsrealcl}.

(b) follows from Prop.~\ref{n:ex:arsstrut3} and
Exercise~\ref{n:ex:arsstrut1}, since for a connected $R$ =
irreducible $S$ the formula (\ref{n:arsed3}) implies $\La_{\rm diff}
= \La\sh + \La\sh$.

(c) (i) $\Rightarrow$ (ii) is obvious. Suppose (ii) holds. We know
from (b) that $R^0=\La_0 = \La\sh +\La\sh$. Since $\La\sh$ is a
pointed reflection subspace, so is $\La_0$
(Exercise~\ref{n:arsexlem}). Hence $2\Span_\ZZ(\La_0) \subset \La_0$
is discrete. But then so is $\Span_\ZZ(\La_0)$. Thus (ii)
$\Rightarrow$ (iii).

By (\ref{n:arsed3}) and (b.ii), $\La\div \subset \La\lg\subset
\La\sh \subset \La_0$. Hence, if (iii) holds, then all $\La_\xi$ are
discrete subsets of $X$. But then so is is $R$, as a finite union of
discrete subsets. This shows (iii) $\rightarrow$ (i).

It is well-known fact that every discrete subgroup of a
finite-dimensional real vector space is free of finite rank.
\end{proof} \sm

A quick comparison of \cite[Definition~2.1]{aabgp} and our
Definition~\ref{n:earsdef} together with Prop.~\ref{n:earsstrut}(a)
will convince the reader that an extended affine root system in the
sense of \cite{aabgp} is the same as a discrete extended affine root
system over $\RR$ in our sense. The reason for the generalization
and the change of name is the same as the one justifying our more
general notion of extended affine Lie algebras: We are considering
EALAs over arbitrary fields of characteristic $0$ and the set of
roots of an EALA will not be an extended affine root system in the
sense of \cite{aabgp}. \sm

Finally, here is the result which brings us back to EALAs.

\begin{theorem} \label{n:ears&eala} Let $(E,H)$ be an EALA over $F$
and let $R\subset H^*$ be its set of roots. Put $X=\Span_F(R)$ and
let $\inpr_X$ be the restriction of the bilinear form {\rm
(\ref{n:ealadef2})} to $X$. Then $(R,X,\inpr_X)$ is an extended
affine root system. If $F=\CC$, then $(E,H)$ is a discrete EALA if
and only if $(R,X,\inpr_X)$ is a discrete extended affine root
system. \end{theorem}

\begin{remarks} (a) For $(E,H)$ a discrete EALA over $F=\CC$, the
theorem is proven in \cite[I, Th.~2.16]{aabgp}, using discreteness.
The generalization to arbitrary EALAs is due to the author, see
\cite[Prop.~3]{n:eala}. It has been further generalized to other
classes of Lie algebras, the so-called invariant affine reflection
algebras, see \cite[Th.~6.6 and Th.~6.8]{n:persp}. Special cases
have also been proven in \cite{az:grla} and \cite{morita-yoshii}.
That $R$ is symmetric, is an easy exercise, namely
Exercise~\ref{n:ex:eala-def1}(c).

(b)  In view of the theorem above, one can ask if every extended
affine root system is the set of roots of some extended affine Lie
algebra. This is however not the case, see \cite[Th.~6.2]{AG2} for a
detailed discussion of this question.
\end{remarks}

As a first application of this theorem, we can now completely
characterize EALAs of nullity $0$.

\begin{proposition} \label{n:fdeala}
The following are equivalent: \begin{itemize}
 \item[\rm (i)] $(E,H)$ is an EALA of nullity $0$,

\item[\rm (ii)] $(E,H)$ is an EALA with a finite-dimensional $E$,

\item[\rm (iii)] $E$ is a finite-dimensional split simple Lie algebra with
splitting Cartan subalgebra $H$.
\end{itemize}
In this case, $E$ equals its core and the set of roots $R$ coincides
with the quotient root system $S$ of $R$ and is an irreducible
reduced finite root system.
\end{proposition}


\section{The core and centreless core of an EALA}\label{n:sec:core}

In the previous chapter we have studied affine reflection systems
per se. The rationale for doing so became clear only in the end,
when we saw in Th.~\ref{n:ears&eala} that the set of roots $R$ of an
EALA $(E,H)$ is an extended affine root system, a special type of an
affine reflection system.

In this chapter we start by drawing consequences of the Structure
Theorem~\ref{n:arsstructh} of affine reflection systems and the
description of extended affine root systems in
Prop.~\ref{n:earsstrut}. The examples in \S\ref{n:sec:ealaone} and
\S\ref{n:sec:ealan} indicate that the core $E_c$ and centreless core
$E_{cc}= E_c / Z(E_c)$ of an extended affine Lie algebra $(E,H)$
really are the ``core'' of the matter. We will show in
Th.~\ref{n:ealcor} and in Cor.~\ref{n:ccore} that both are so-called
Lie tori, a new class of Lie algebras which we will introduce in
\ref{n:sec:lietordef}. We will present some basic properties of Lie
tori in \ref{n:lietorprop} and describe some examples in
\ref{n:sec:lietypeA} and \ref{n:sec:lietorex}.

With some justification, this chapter could therefore also be
entitled ``On Lie tori''. But the reader can be re-assured that we
are not getting side-tracked too much: In the next chapter we will
see that Lie tori are precisely what is needed to construct EALAs.

\subsection{Lie tori: Definition} \label{n:sec:lietordef}
\sm

Lie tori are special objects in the following category of graded Lie
algebras.

\begin{definition}\label{n:sladef} 
Let $(S,Y)$ be a finite irreducible, but not necessarily reduced
root system, as defined in Example~\ref{n:arsexfin}. We denote by
$\scQ(S)=\Span_\ZZ(S)\subset Y$ the root lattice of $S$. To avoid
some degeneracies we will always assume that $S\ne \{0\}$. Let $\La$
be an abelian group. \sm

A \textit{$(\scQ(S),\La)$-graded\/} Lie algebra is a Lie algebra $L$
with compatible $\scQ(S)$- and $\La$-gradings. It is convenient (and
helpful) to use subscripts for the $\scQ(S)$-grading and
superscripts for the $\La$-grading. Thus,
$$
   L = \ts \bigoplus_{q\in \scQ(S)} L_q  = \bigoplus_{\la \in \La} L^\la
$$
are $\scQ(S)$- and $\La$-gradings of $L$, and compatibility means
$$
 \ts L = \bigoplus_{q\in \scQ(S), \, \la \in \La} L_q^\la \quad
 \hbox{for } L_q^\la = L_q \cap L^\la.
   $$
Hence for $\la,\mu \in \La$ and $p,q\in \scQ(S)$
$$  L^\la = \ts \bigoplus_{q\in \scQ(S)} L_q^\la, \quad
     L_q = \bigoplus_{\la \in \La} L^\la_q \quad \hbox{and}\quad
   [L_q^\la, L_p^\mu] \subset L_{q+p}^{\la + \mu}.
$$
Thus, $L$ has three gradings, by $\scQ(S)$, $\La$ and $\scQ(S)
\oplus \La$ whose interplay will be crucial in the following.
Corresponding to these three different gradings are three
\textit{support sets} :
  $\supp_{\scQ(S)} L = \{ q\in \scQ(S) : L_q \ne 0\}$,
  $\supp_\La L  = \{ \la \in L : L^\la \ne 0\}$, and
  $\supp_{\scQ(S)\oplus\La} L = \{ (q,\la) \in (\scQ(S), \La): L_q^\la
   \ne 0 \}$.
 \end{definition}

\begin{definition} \label{n:lietorax} We keep the notation of the Def.~\ref{n:sladef}.
A \textit{Lie torus of type $(S,\La)$} is a $(\scQ(S), \La)$-graded
Lie algebra $L$ over $F$, a field of characteristic $0$, satisfying
the axioms (LT1)--(LT3) below. \begin{description}
\item[(LT1)]  $\supp_{\scQ(S)} L \subset S$, hence
 $L=\bigoplus_{\xi \in S} L_\xi$.

\item[(LT2)]  If $L_\xi^\la \ne 0$ and $\xi \ne 0$, then there
exist $e_\xi^\la \in L_\xi^\la$ and $f_\xi^\la \in L_{-\xi}^{-\la}$
such that \begin{equation} \label{n:lietordef1}
 L_\xi^\la = F e_\xi^\la, \quad L_{-\xi}^{-\la} = F f_\xi^\la,
\end{equation}
and for $x_\ta \in L_\ta $ we have
\begin{equation} \label{n:lietordef2}
 [[e_\xi^\la, f_\xi^\la],\, x_\ta] = \lan \ta, \xi\ch\ran x_\ta.
\end{equation}

\item[(LT3)] (a) $L_\xi^0 \ne 0$ if $\xi \in S\ind^\times$, i.e., $0\ne \xi
  \in S$ and $\xi/2 \not\in S$.
 \begin{itemize}
  \item[(b)] As a Lie algebra, $L$ is generated by $\bigcup_{0\ne
  \xi \in S} L_\xi$.

  \item[(c)] $\La = \Span_\ZZ( \supp_\La L)$.
  \end{itemize}
\end{description}
We will say that $L$ is a Lie torus if $L$ is a Lie torus for some
pair $(S,\La)$.

 A Lie torus is called \textit{invariant}, if $L$ has
an invariant nondegenerate symmetric bilinear form $\inpr$ which is
\textit{graded} in the sense that
\begin{equation} \label{n:lietordef9}
  (L_\xi^\la \mid L_\ta^\mu) = 0 \quad \hbox{if $\la + \mu \ne 0$
              or $\xi + \ta \ne 0$.}
\end{equation}

Two Lie tori $L$ and $\tilde L$, both of type $(S,\La)$, are called
\textit{graded-isomorphic} if there exists a Lie algebra isomorphism
$f: L \to \tilde L$ such that $f(L_\xi^\la) = \tilde L^\la_\xi$ for
all $(\xi,\la) \in S\times \La$. Thus, a graded-isomorphism is an
isomorphism in the category of graded Lie algebras. But we will use
the term ``graded-isomorphism'' to emphasize that Lie tori are
graded algebras.
\end{definition}

\begin{remarks} (a) Let $L$ be a Lie torus. Hence, by (LT1),
$L= \bigoplus_{\xi \in S} L_\xi =  \bigoplus_{\xi \in S, \, \la \in
\La} L_\xi^\la$. We will determine $\supp_{\scQ(S)} L$ in
Cor.~\ref{n:lietorsupp} below.  The axiom (LT2) implies that
\begin{equation} \label{n:lietordef3}   \dim L_\xi^\la = 1 \quad
\hbox{if $0\ne \xi$ and $L_\xi^\la \ne 0$}
\end{equation}    and  that
\begin{equation} \label{n:lietordef5}
(e_\xi^\la, h_\xi^\la, f_\xi^\la) \quad  \hbox{with} \quad
 h_\xi^\la  = [e_\xi^\la, f_\xi^\la] \in L_0^0\end{equation} is an $\lsl_2$-triple. The
condition (LT3.a) together with (LT2) ensures that a Lie torus has
enough $\lsl_2$-triples.

The other two conditions in (LT3) are not really serious; they just
serve to normalize things: If (LT3.c) does not hold, one can simply
replace $\La$ by $\Span_\ZZ(\supp_\La L)$.  Also,
\begin{equation} \label{n:lietordef4}
\hbox{(LT3.b)} \quad \iff \quad
 L_0^\la= \tsum_{0\ne \xi \in S} \tsum_{\mu \in \La} \,
       [L_\xi^\mu, \, L_{-\xi}^{\la - \mu}]
\end{equation}
for all $\la\in \La$. If one has a Lie algebra, for which all axioms
except (\ref{n:lietordef4}) hold, one can replace the subspaces
$L_0^\la$ by the right hand side of (\ref{n:lietordef4}) and then
gets a Lie torus. Observe that (\ref{n:lietordef4}) for $\la=0$
together with (LT2) yields
\begin{equation} \label{n:lietordef6}
 L_0^0 = \tsum \, F h_\xi^\la
\end{equation}
where the sum in (\ref{n:lietordef6}) is taken over all pairs
$(\xi,\la)$ for which $h_\xi^\la$ exists, i.e., those with
$L_\xi^\la \ne 0$ and $\xi \ne 0$.
 \sm

(b) A Lie torus is a special type of a so-called
division-$(S,\La)$-graded Lie algebras, or more generally of a
root-graded Lie algebra. This and also the different approaches to
root-graded Lie algebras are discussed in \cite[\S5]{n:persp}. Lie
tori were first defined by Yoshii in \cite{y:ext,y:lie}, using the
notion of a root-graded Lie algebra. The definition above is due to
the author \cite{n:tori}.

Viewing a Lie torus as a special type of a root-graded Lie algebra
is the approach used in the classification of Lie tori. \sm

(c) Why was a Lie torus christened a ``Lie torus''? The historically
correct answer is: Because of pure analogy with already existing
names like a quantum torus, defined in \ref{n:quantordef}, or an
alternative or Jordan torus. All of these are graded algebras, in
which every non-zero homogeneous element is invertible. If one
interprets the elements $e$ and $f$ of the $\lsl_2$-triple
(\ref{n:lietordef5}) as invertible elements of $L$, then a Lie torus
is a graded Lie algebra in which most of the non-zero homogenous
elements are invertible. It is certainly unusual to speak of
``invertible elements'' in a Lie algebra. But the examples below
will provide some justification to that: We will see that the
invertible elements of $L$ are given by invertible elements of its
coordinate algebra.

Besides the analogy with the already existing concepts of ``tori''
in categories of (non)associative algebras, the fact that a
\textit{toroidal} Lie algebra is a Lie \textit{torus}, see
\ref{n:lieuntwist}, reinforces the choice of the name Lie torus.
\end{remarks}

\subsection{Some basic properties of Lie tori} \label{n:lietorprop}

Throughout this section $L$ is a Lie torus of type $(S,\La)$. We use
the notation of Def.~\ref{n:lietorax}. We describe some basic
properties of Lie tori and prove some of them, in particular those
for which there does not yet exist  a published proof.

We first show that the homogeneous subspaces of the
$\scQ(S)$-grading of $L$ are weight spaces for the
$\ad$-diagonalizable subalgebra
$$\frh=\Span_F\{h_\xi^0 : \xi \in S\ind^\times\}.$$

\begin{lemma}\label{n:weisp} The subspaces $L_\ta$, $\ta \in S$, are
given by
\begin{equation} \label{n:weisp1}
L_\ta = \{ l\in L : [h_\xi^0, l] = \lan \ta, \xi\ch\ran l \hbox{ for
all $\xi \in S\ind^\times$} \}.
\end{equation}
\end{lemma}

\begin{proof} The inclusion from left to right holds by (\ref{n:lietordef2}).
For the proof of the other inclusion we write $l\in L$ as
$l=\sum_\al l_\al$ with $l_\al \in L_\al$. Then $l$ satisfies
$[h_\xi^0, l] = \lan \ta, \xi\ch\ran l \hbox{ for all } \xi \in
S\ind^\times$ if and only if for every $\al \in S$ we have $\lan \al
- \ta, \xi\ch\ran l_\al = 0$ for all $\xi \in S\ind^\times$. Since
$\Span_F(S\ind)=Y$ and the bilinear form on $Y$ associated with the
root system  $S$ is nondegenerate, for every pair $(\al, \ta) \in
S^2$ with $\al \ne \ta$ there exists $\xi \in S\ind^\times$ with
$\lan \al - \ta, \xi\ch\ran \ne 0$. Hence, any $l$ belonging to the
set on the right hand side of (\ref{n:weisp1}) has  $l_\al = 0$ for
$\al \ne \ta$, proving $l\in L_\ta$.
\end{proof}

\begin{proposition}\label{n:weylref}  For every $(\xi, \la) \in \supp_{\scQ(S) \oplus
\La} L$ with $\xi \ne 0$ the map
$$
  \vphi_\xi^\la  = \exp\big( \ad(e_\xi^\la)\big) \,
   \exp\big( \ad( - f_\xi^\la)\big) \, \exp\big( \ad(e_\xi^\la)\big)
$$
is a well-defined automorphism of the Lie algebra $L$ with the
property \begin{equation} \label{n:weylref1} \vphi_\xi^\la (
L_\ta^\mu)  = L_{s_\xi(\ta)}^{\mu - \lan \mu, \xi\ch\ran \la}.
\end{equation}
Moreover, for every $w\in W(S)$, the Weyl group of $S$, there exists
an automorphism $\vphi_w$ of the Lie algebra $L$ such that
$$\vphi_w(L_\ta^\mu) = L_{w(\ta)}^\mu$$ for all $\ta \in S$ and $\mu \in
L$.
\end{proposition}

This proposition can be proven in the same way as
\cite[Prop.~1.27]{aabgp}.

\begin{corollary}[{\cite[Lemma~1.10]{abfp2}}] \label{n:lietorsupp}  The $\scQ(S)$-support of $L$ satisfies
\[
  \supp_{\scQ(S)} L =  \begin{cases} S & \hbox{if $S$ is reduced}, \\
              S \hbox{ or } S\ind  &  \hbox{if $S$ is non-reduced.}
   \end{cases}
\]
 \end{corollary}

As a consequence of this corollary, $\supp_{\scQ(S)}L$ is always a
finite irreducible root system. It is immediate that $L$ is also a
Lie torus of type $(\supp_{\scQ(S)}, \La)$. Without loss of
generality we can therefore assume that $S=\supp_{\scQ(S)}L$ if this
is convenient.

\begin{proposition} \label{n:divsupp}
For $\xi \in S$ define \[ \La_\xi = \{\la \in \La : L_\xi^\la \ne 0
\},  \] so that $\supp_\La L = \bigcup_{\xi \in S} \La_\xi$. Then
the family $(\La_\xi : \xi \in S)$ satisfies the axioms {\rm (ED1)}
and {\rm (ED2)} of\/ {\rm Def.~\ref{n:arsed}},
\begin{align}
  &\La_\eta - \lan \eta, \xi\ch\ran \La_\xi \subset
  \La_{s_\xi(\eta)}   \tag{ED1} \\
  &\hbox{$0 \in \La_\xi$ for $\xi \in S\ind$ and $\La_\xi \ne
  \emptyset$ for $\xi \in S^\times\div$}
    \tag{ED2}
\end{align}
Hence $\La_\xi$ is a pointed reflection subspace for $\xi \in
S^\times\ind$,  a symmetric reflection subspace for $\xi \in
S^\times\div$ and \begin{equation} \label{n:divsupp1}
   \La_\xi = \La_{w(\xi)} \quad \hbox{for all $w\in W(S)$}.
   \end{equation}
Defining $\La\sh$, $\La\lg$ and $\La\div$ as in {\rm
\ref{n:arsed1}}, we have
\begin{align}
  \La\sh &\supset \La\lg \supset \La\div, \label{n:divsupp2} \\
  \supp_\La L &= \La_0 = \La\sh + \La\sh, \label{n:divsupp3} \\
   \emptyset     &=   2\La\sh \cap \La\div  , \label{n:divsupp4} \\
  \La &= \Span_\ZZ(\La\sh). \label{n:divsupp5}
\end{align}
The support families $(\La_\xi: \xi \in S)$ are the same for $L$ and
$L/Z(L)$.
\end{proposition}

\begin{proof} (ED1) is a consequence of Prop.~\ref{n:weylref} and (ED2) of
(LT3.a). It then follows as in section~\ref{n:sec:roots} that
$\La_\xi$, $\xi \in S^\times$, are pointed respectively symmetric
reflection subspaces such that (\ref{n:divsupp1}) and
(\ref{n:divsupp2}) hold. (\ref{n:divsupp3}) and (\ref{n:divsupp4})
are proven in \cite[Th.~5.1]{y:ext} and \cite[Lemma~1.1.12]{abfp2}.
(\ref{n:divsupp5}) follows from (\ref{n:divsupp2}) and
(\ref{n:divsupp3}). The last claim is also proven in
\cite[Th.~5.1]{y:ext}.
\end{proof}

\begin{proposition} \label{n:liesplit} Define
\begin{align*}
   \g &= \hbox{subalgebra generated by $\{ L^0_\xi :
            \xi \in  S\ind^\times\}$}, \\
  \frh &= \Span_F \{ h_\xi^0 : \xi \in S\ind^\times\}.
 \end{align*}

{\rm (a)} Then $\frg$ is a finite-dimensional split simple Lie
algebra with splitting Cartan subalgebra $\frh$. \sm

{\rm (b)} The root system $S\ind$ and the root system of $(\g,
\frh)$ are canonically isomorphic. Namely, for every $\xi \in
S^\times \ind$ there exists a unique $\tilde \xi \in \frh^*$,
defined by $\tilde \xi (h_\eta^0) = \lan \xi, \eta\ch\ran$ for $\eta
\in S\ind^\times$, such that the map $\xi \mapsto \tilde \xi$
extends to an isomorphism between the root system $S\ind$ and the
root system of $(\g,\frh)$.
\end{proposition}

\begin{proof} This is a special case of a result for arbitrary root-graded
Lie algebras, see \cite[Remark 2 of \S2.1]{n:3g} and
\cite[Prop.~5.9]{n:persp}. It is essentially a corollary to the
Chevalley-Serre presentation of finite-dimensional split simple Lie
algebras. For Lie tori it was announced in \cite[\S3]{n:tori}. The
details of the proof are given in \cite[Prop.~1.2.2]{abfp2}.
\end{proof}

The following Ex.~\ref{n:ex:lietor} lists some more basic properties
of Lie tori. You will need Ex.~\ref{n:ex:centex}(a) in part (d) of
\ref{n:ex:lietor}.

\begin{exercise} \label{n:ex:centex} (a) Let $K$ be a perfect Lie algebra. Then $K/Z(K)$ is a
perfect and centreless Lie algebra.

(b) Let $E$ be a Lie algebra with an invariant nondegenerate
symmetric bilinear form $\inpr$, and let $K$ be an ideal of $E$ with
$K=[E,K]$. Then $\{z \in K : (z\mid K)=0 \} = Z(K)$.\end{exercise}

\begin{exercise} \label{n:ex:lietor} Let $L$ be a Lie torus of type $(S,\La)$. Show:

(a) $L_0^\la$ is given by the formula (\ref{n:lietordef4}).

(b) $L$ is perfect.

(c) The centre satisfies $Z(L) = \textstyle \bigoplus_{\la \in \La}
Z(L)^\la$ for $Z(L)^\la = Z(L) \cap L_0^\la$.

(d) Let $Y= \bigoplus_{\la \in \La} Y_0^\la$, $Y_0^\la = Y \cap
L_0^\la$, be a graded subspace of $Z(L)$. Then $L/Y$ is a Lie torus
with respect to the subspaces $(L/Y)_\xi^\la = L_\xi^\la /
Y_\xi^\la$, where for $\xi \ne 0$ we put $Y_\xi^\la = \{0\}$ and
thus have $(L/Y)_\xi^\la \cong L_\xi^\la$ as vector spaces. In
particular, $L/Z(L)$ is a centreless Lie torus.

(e) For $\la,\mu \in \La_\xi$, $\xi\in S^\times$, we have $h_\xi^\la
\equiv h_\xi^\mu \mod Z(L)$.

(f) $L^0 = \g \oplus Z(L)^0$ and $L_0^0 = \frh \oplus Z(L)^0$.

(g) Let $I$ be a $\La$-graded ideal of $L$, whence $I=\bigoplus_{\la
\in \La} I^\la$ for $I^\la = I \cap L^\la$. Then either $I=L$ or $I
\subset Z(L)$. In particular, a centreless Lie torus is
graded-simple with respect to the $\La$-grading of $L$.
\end{exercise}

\sm

Since a Lie torus is perfect by part (b) of the exercise above, it
has a universal central extension (Th.~\ref{n:sub:ucethm}).

\begin{theorem}[{\cite[\S5]{n:tori}, \cite{n:uce}}]\label{n:ucelie}
Let\/ $\fru : \uce(L) \to L$ be a universal central extension of a
Lie torus $L=\bigoplus_{\xi, \la} L_\xi^\la$ of type $(S,\La)$. Then
$\uce(L)$ is also a Lie torus of type $(S,\La)$, say $\uce(L) =
\bigoplus_{\xi, \la} \uce(L)_\xi^\la$, and $\fru$ maps
$\uce(L)_\xi^\la$ onto $L_\xi^\la$.
\end{theorem}

\begin{remark} \label{n:ucelierem}
It follows from this theorem and the exercise above that in order to
describe Lie tori up to graded isomorphism, one can proceed in two
steps: \begin{itemize}

 \item[(A)] Classify centreless Lie tori, up to graded isomorphism.
We will discuss some examples in \ref{n:sec:lietypeA} and
\ref{n:sec:lietorex}.

\item[(B)] Describe the universal central
extension of the centreless Lie tori from (A). They are unique up to
isomorphism. We will not say anything about this here. The reader
can find some results in \cite{bgk,bgkn,n:eala,n:uce} for Lie tori
arising from EALAs and in \cite{abg,abg2,BeSm,n:uce} for general
root-graded Lie algebras.
\end{itemize}
Once (A) and (B) completed, an arbitrary Lie torus of type $(S,\La)$
is then obtained as $\uce(L)/C$ where $L$ is taken from the list in
(A) and where $C$ is a graded subspace of the centre of $\uce(L)$.
\end{remark}

The following results only holds for special types of Lie tori.

\begin{theorem}[{\cite[Th.~5]{n:tori}, proven in \cite{n:uce}}]
\label{n:torfg} Let $L= \bigoplus_{\xi \in S, \, \la\in \La}
L_\xi^\la$ be a Lie torus of type $(S,\La)$ where $\La$ is a
finitely generated abelian group. \sm

{\rm (a)} Then $L$ is finitely generated as Lie algebra and has
bounded homogeneous dimension with respect to the $\scQ(S) \oplus
\La$-grading of $L$. \sm

{\rm (b)} Moreover, the Lie algebra $\Der_F(L)= \grDer_F(L)$ where
$\grDer_F(L)$ is naturally $\scQ(S) \oplus \La$-graded and has
bounded homogeneous dimension with respect to this grading. \sm

{\rm (c)} If $L$ is invariant, its universal central extension is
isomorphic to the central extension $\rmE(L,D^{\gr *}, \psi_D)$
where $D$ is any graded complement of $\IDer(L)$ in $\SDer_F(L)$.
\end{theorem}

Part (c) of this theorem is an immediate corollary of parts (a) and
(b) and of Th.~\ref{n:thgencen}.

\subsection{The core of an EALA}\label{n:betdes}

We will now connect extended affine Lie algebras and Lie tori, and
first introduce some notation. Let $(E,H)$ be an EALA with set of
roots $R$. We have seen in Th.~\ref{n:ears&eala} that $R$ is an
extended affine root system, hence an affine reflection system. We
can therefore apply the Structure Theorem~\ref{n:arsstructh}. Recall
the following data describing the structure of $R$:
\begin{itemize}
 \item[$\bull$] $X=\Span_F(R) \subset H^*$, $X^0 = \{ x\in X : (x
 \mid X)=0\} = \{x\in X : (x\mid R)=0\}$, $f: X \to X/X^0=Y$ the
 canonical projection,

 \item[$\bull$] $S=f(R)$ the quotient root system, a finite
irreducible but possibly non-reduced root system, and $S\ind = \{
\al  \in S : \al/2\not\in S \} \cup \{0\}$,

 \item[$\bull$] $g: Y \to X$ a linear map satisfying $f\circ g = \Id_Y$ and
 $g(S\ind) \subset R$,

 \item[$\bull$] $(\La_\xi : \xi \in S)$ the associated extension datum, defined
 by $R \cap f^{-1}(\xi)= g(\xi) \oplus \La_\xi$ and $\La_\xi \subset
 X^0$,

 \item $\La= \Span_\ZZ \big( \bigcup_{\xi \in S} \La_\xi \big)$, a free abelian
 group of finite rank (this is axiom (EARS7)).
\end{itemize}
Hence\begin{align*}
  R &= \textstyle \bigcup_{\xi \in S} \, \big(g(\xi) \oplus \La_\xi \big) \subset
            g(Y) \oplus X^0, \\
  R\an &= \textstyle\bigcup_{\xi \in S^\times}  \big(
                g(\xi) \oplus \La_\xi \big), \\
  R^0 &= 0 \oplus \La_0 = R \cap X^0.
\end{align*}

\begin{theorem}[\cite{AG2} for $F=\CC$] \label{n:ealcor} Let
$K=E_c$ be the core of an EALA $(E,H)$. We use the notation of above
and define subspaces
\begin{equation}\label{n:lietordef8}
 K_\xi^\la = K \cap E_{g(\xi) \oplus \la} = \begin{cases}
   E_{g(\xi) \oplus \la} & \xi \ne 0, \\
          K \cap E_{0 \oplus \la}, &\xi=0.
\end{cases}
\end{equation}

{\rm (a)} Then $K= \bigoplus_{\xi, \, \la} K_\xi^\la$ is a Lie torus
of type $(S,\La)$, where $\La$ is free abelian of finite rank. \sm

{\rm (b)} $K$ is a perfect ideal of $E$. \sm

{\rm (c)} Let $\inpr$ be a nondegenerate invariant bilinear form on
$E$, whose existence is guaranteed by the axiom {\rm (EA1)}. Then
the radical of the restricted bilinear form $\inpr|_{K\times K}$
equals the centre $Z(K)$, that is
\begin{equation} \label{n:betdes2}
   \{ z\in K : (z\mid K) = 0 \} = Z(K)
= \textstyle \bigoplus_{\la \in \La} Z(K) \cap K_0^\la.
\end{equation}
\end{theorem}

\begin{remark} \label{n:ealcorem} The subspaces $K_\xi^\la$ in
(\ref{n:lietordef8}) and hence the Lie torus structure of $K$ depend
on the section $g$. A different choice of $g$ leads to a so-called
\textit{isotope} of $K$, see \cite{AF:isotopy} and
\cite[Prop.~6.4]{n:persp}.
\end{remark}

\begin{corollary} \label{n:ccore} We use the notation of\/ {\rm Th.~\ref{n:ealcor}}, and put
$E_{cc}= K/Z(K)=L$, the {\rm centreless core of $(E,H)$}. Then $L$
is an invariant centreless Lie torus of type $(S,\La)$ with respect
to the homogeneous subspaces
$$
   L_\xi^\la =  K_\xi^\la \big/ \big(Z(K)\cap K_\xi^\la \big)
$$
and the bilinear form $\inpr_L$ defined by
$$
  (\bar x \mid \bar y)_L = (x \mid y)
$$
where $x,y\in K$, $\bar x$ and $\bar y$ are the canonical images in
$L$ and $\inpr$ is the bilinear form of {\rm \ref{n:ealcor}(c)}.
\end{corollary}

\begin{remark} \label{n:yosrem} Yoshii (\cite{y:lie}) has shown that any Lie torus of
type $(S,\La)$ with $\La$ a torsion-free abelian group admits a
non-zero graded invariant symmetric bilinear form. This implies that
the existence of a nondegenerate such form on $E_{cc}$. However,
Yoshii's proof uses the existence of invariant nondegenerate
symmetric bilinear forms on Jordan tori (\cite{NY}) and hence relies
on the classification of Jordan tori.
\end{remark}

We can now show that all root spaces of an EALA are
finite-dimensional in a strong form.

\begin{proposition}[{\cite[Prop.~3]{n:eala}}] \label{n:bdd}
An EALA has finite bounded dimension. The same is true for its core
and centreless core.
\end{proposition}

\begin{proof} Let $K=E_c$ be the core of the EALA $(E,H)$. By
Th.~\ref{n:ealcor}, $K$ is a Lie torus of type $(S,\La)$, where
$\La$ is a free abelian group of finite rank. Hence, by
Th.~\ref{n:torfg}, one knows that $K$ has finite bounded dimension
with respect to its double grading, say $\dim K_\xi^\la \le M_1$ for
all pairs $(\xi, \la)$. By the same reference, one also knows that
the Lie algebra $\Der_F(K)$ of all $F$-linear derivations of $K$ has
a double grading by $\scQ(S)$ and $\La$,
$$
  \Der_F (K) = \textstyle \bigoplus_{\xi \in S, \, \la \in \La}
    (\Der_F K)^\la_\xi,
$$
where $(\Der_F K)^\la_\xi$ is the subspace of those derivations
mapping $K_\ta^\mu$ to $K_{\xi+ \ta}^{\la + \mu}$, and that
$\Der_F(K)$ has finite bounded dimension with respect to this
grading, say $\dim_F (\Der_F K)^\la_\xi \le M_2$ for all pairs
$(\xi,\la)$.

Since $K$ is an ideal, we have a Lie algebra homomorphism $\rh : E
\to \Der_F(K)$, given by $\rh(e) = \ad e |_K$. It is homogenous of
degree $0$, i.e., $\rh(E_\al) \subset (\Der_F K)^\la_\xi$ for $\al =
g(\xi) \oplus \la$ as in (\ref{n:lietordef8}). Moreover, by the
tameness axiom (EA5) for an EALA we know that $\Ker \rh \subset K$
(whence $\Ker \rh = Z(K)$, but we won't need this). It now follows
that
 $\dim E_\al =    \dim \Ker (\rh|_{E_\al})    +  \dim \rh(E_\al)
    \le  \dim K_\xi^\la    +     \dim (\Der_F K)^\la_\xi \le M_1 + M_2$.
\end{proof}

\subsection{Lie tori of type $\rma_l$, $l \ge 3$} \label{n:sec:lietypeA}

As explained in Rem.~\ref{n:ucelierem}, in classifying Lie tori one
can restrict one's attention to the case of centreless Lie tori, at
least modulo the solution of problem (B) in \ref{n:ucelierem}. In
this section we will describe centreless Lie tori of type $\rma$.

The reader will expect that this will have something to do with
trace-$0$-matrices. This turns out to be correct, but only with the
proper interpretation of ``trace-$0$''. It will not be sufficient to
consider trace-$0$-matrices over $F$. We will see in \ref{n:lietonu}
that they will only lead to nullity $0$-examples. Rather, one must
allow matrices with entries from a possibly non-commutative algebra.
\sm

To avoid some degeneracies, in this section  we let $N$ be a natural
number with $N\ge 3$. We start with an arbitrary associative unital
$F$-algebra $A$. In particular, $A$ need not be commutative, and
hence
 \[ [A,A] = \Span_F \{ a_1 a_2 - a_2 a_1 : a_1, a_2 \in A\}
  \]
is in general non-zero. As usual, $\gl_N(A)$ is the Lie algebra of
all $N \times N$ matrices with entries in $A$ and Lie algebra
product $[x,y] = xy-yx$, the usual commutator of the matrices $x$
and $y$. We define the \textit{special linear Lie algebra\/}
$\lsl_N(A)$ as the derived algebra
\[
 \lsl_N(A)  = [ \gl_N(A), \gl_N(A)] .
     \]
of $\gl_N(A)$. In particular, $\lsl_N(A)$ is an ideal of $\gl_N(A)$.
To analyze the structure of $\lsl_N(A)$ we use the matrix units
$E_{ij}$, i.e., the $N\times N$ matrices with $1$ at the position
$(ij)$ and $0$ at all other positions. They satisfy the basic
multiplication rule \begin{equation} \label{n:quant2}
    [aE_{ij}, \, bE_{mn}] = \de_{jm}\, ab \,E_{in} - \de_{ni}\,ba \,E_{mj}
\end{equation}  where $\de_*$ is the usual Kronecker delta.
We put $E_N= \sum_{i=1}^N E_{ii}$. Some properties of the Lie
algebra $\lsl_N(A)$ are listed in the following (very worthwhile)
exercise.

\begin{exercise} \label{n:ex:slgen} (a) $\lsl_N(A)
= \{ x\in \gl_N(A) : \tr(x) \in [A,A]\}$.

(b) As a vector space, $\lsl_N(A)$ decomposes as \begin{align}
   \lsl_N(A) &= \lsl_N(A)_0 \oplus \ts \big(
          \bigoplus_{i\ne j} A E_{ij} \big), \quad\hbox{where}
      \label{n:quant1}   \\
  \lsl_N(A)_0 &= \lsl_N(A) \cap \big( \ts \bigoplus_{i=1}^N
            AE_{ii}\big) \nonumber \\
 &= \tsum_{i\ne j} [AE_{ij}, AE_{ji}]
  = \tsum_{i\ne j} \Span_F \{abE_{ii} - baE_{jj} :a,b\in A\}
     \nonumber \\
&= \{ cE_N : c\in [A,A] \}
       \oplus
  \big( \ts \bigoplus_{i=1}^{N-1} \{a(E_{ii} - E_{i+1, i+1}) : a\in A\}
      \big) \nonumber
\end{align}

(c) For a commutative $A$:
\begin{equation*}
 \lsl_N(A)  =
 \{ x\in \gl_N(A) : \tr(x) = 0 \} = \lsl_N(F) \ot_F A .
\end{equation*}

(d) The centre of $\lsl_N(A)$ is $
   Z(\lsl_N(A)) = \{ z E_N : z\in Z(A) \cap [A,A]\}
$ where $Z(A) = \{ z\in A: za=az \hbox{ for all } a\in A \}$ is the
centre of $A$.

(e) For $a,b\in A$ and $i\ne j$,
\[
    (aE_{ij}, E_{ii} - E_{jj}, bE_{ji})
\]
is an $\lsl_2$-triple if and only if $a$ is invertible and
$b=a^{-1}$.\end{exercise}

The exercise shows that the structure of a general $\lsl_N(A)$ is
quite similar to that of $\lsl_N(F)$. In particular, the
decomposition (\ref{n:quant1}) is a $\scQ(\rma_l)$-grading where
\[
   \rma_l = \{\veps_i - \veps_j : 1\le i,j\le N \} ,
    \quad l=N-1.
   \]
is the root system of type $\rma_l$ and
\[
   \lsl_N(A)_{\veps_i - \veps_j} = A E_{ij} \quad
            \hbox{for $i\ne j$.}
\]
In fact, $\lsl_N(A)$ is the prototype of an $\rma_l$-graded Lie
algebra (see \cite{bm}). At this level of generality we are far from
the structure of a Lie torus. Most importantly, we are missing a
compatible $\La$-grading of $\lsl_N(A)$. We will use gradings of
$A$, defined as follows.

\begin{definition}\label{n:asstor} Let $A=\bigoplus_{\la \in \La} A^\la$
be a unital associative $\La$-graded $F$-algebra. Then $A$ is called
an \textit{associative torus of type $\La$} if it satisfies
(AT1)--(AT3) below. \begin{description}

\item[(AT1)] if every non-zero $A^\la$ contains an invertible
element,

\item[(AT2)] $\dim A^\la \le 1$ for all $\la \in \La$, and

\item[(AT3)] $\Span_\ZZ (\supp_\La A ) = \La$. \end{description}
One calls $A$ simply an \emph{associative torus\/} if $A$ is an
associative torus of type $\La$ for some abelian group $\La$. See
\ref{n:twigrrev} for a short discussion of associative tori.
\end{definition}

These definitions are justified by the following exercise,
describing when $\lsl_N(A)$ is a Lie torus of type $(\rma_l,\La)$.

\begin{exercise} \label{n:ex:sltor} (a) The Lie algebra $\lsl_N(A)$ has a $\La$-grading
compatible with the $\scQ(\rma_l)$-grading (\ref{n:quant1}) if and
only if $A$ is $\La$-graded. In this case, the compatible
$\La$-grading of $\lsl_N(A)$ is given by $\lsl_N(A)= \bigoplus_{\la
\in \La} \lsl_N(A)^\la$ where $\lsl_N(A)^\la$ consists of matrices
in $\lsl_N(A)$, which have all their entries in $A^\la$.

(b) With respect to the compatible gradings of (a), the Lie algebra
$\lsl_N(A)$ is a Lie torus of type $(\rma_{N-1}, \La)$ if and only
if $A$ is an associative torus of type $\La$.

(c) The Lie torus $\lsl_N(A)$ is invariant with respect to the
bilinear form $\inpr_\lsl$ given by $ (\sum_{i,j} x_{ij} E_{ij} \mid
\sum_{p,q} y_{pq} E_{pq})_\lsl = \sum_{i,j} ( x_{ij} y_{ji})_0$
where $a_0$ for $a\in A$ denotes the $A^0$-component of $a$.
\end{exercise}

But we not only have an example of a Lie torus of type
$(\rma_l,\La)$, we actually have all centreless examples.

\begin{theorem} \label{n:typeAcl} Let $l\ge 3$. A Lie algebra $L$ is a centreless Lie torus of
type $(\rma_l, \La)$ if and only if $L$ is graded-isomorphic to
$\lsl_{l+1}(A)$ for $A$ an associative torus of type $\La$. In this
case, $L$ is an invariant Lie torus.
\end{theorem}

\begin{proof} This is a special case of the Coordinatization Theorem of
$\rma_l$-graded Lie algebras (\cite[Recognition Theorem~0.7]{bm}): A
centreless Lie algebra $L$ is $\rma_l$-graded ($l=N-1$) if and only
if $L$ is $\scQ(\rma_l)$-graded-isomorphic to
$\lsl_N(A)/Z(\lsl_N(A))$ for some associative $F$-algebra $A$. If
$L$ is a Lie torus, it follows as in the Exercise~\ref{n:ex:sltor}
above that $A$ is an associative torus. But $Z(\lsl_N(A))=\{0\}$ for
an associative torus (\cite[(3.3.2)]{NY} and
Exercise~\ref{n:ex:slgen}).
\end{proof}

Besides \cite{bm}, related results are proven in
\cite[Th.~2.65]{bgk} (see Cor.~\ref{n:bgkcor} below), \cite[2.11 and
3.4]{gn2} and \cite[Prop.~2.13]{y2}.

\begin{rev}[Associative tori versus twisted group algebras]
\label{n:twigrrev} In view of Th.~\ref{n:typeAcl} it is of interest
to know more about associative tori. First of all, the identity
$1_A$ of an associative torus $A$ satisfies $1_A \in A^0$. Hence
$a^{-1} \in A^{-\la}$ for every invertible $a\in A^\la$. Moreover,
since the product of two invertible elements in an associative
algebra is again invertible, it follows that $\supp_\La A$ is a
subgroup of $\La$, whence (AT3) is equivalent to
\begin{description}
 \item[(AT3)$'$] $\supp_\La A = \La$.
\end{description}
Next, choose a family $(u_\la : \la \in \La)$ of invertible elements
$u_\la\in A^\la$. This is then in particular an $F$-basis of $A$ so
that the algebra structure of $A$ is completely determined by the
equations
\begin{equation} \label{n:quant3}
 u_\la u_\mu = c(\la,\mu) u_{\la + \mu}
\end{equation}
for $\la,\mu \in \La$ and suitable non-zero scalars $c(\la,\mu) \in
F$. It is not necessary that $c(\la, \mu) = 1$, see for example
Ex.~\ref{n:ex:quntor}. Rather, given an $F$-vector space with basis
$(u_\la: \la \in \La)$ one can define a multiplication on $A$ by
(\ref{n:quant3}), and this multiplication is associative if and only
if
\begin{equation} \label{n:quant4}
   c(\la, \mu) \, c(\la + \mu, \nu) = c(\mu, \nu) \, c(\la, \mu+\nu)
\end{equation}
 In this case, the algebra is an associative
torus of type $\La$. It is clear from the construction that,
conversely, any associative torus is obtained in this way from a
family $(c(\la,\mu))_{\la,\mu}$ of non-zero scalars. The algebras
constructed in this way are called \textit{twisted group algebras}.
The reader with some knowledge in group cohomology will recognize
that the families $(c(\la,\mu))$ satisfying (\ref{n:quant4}) are
precisely the $2$-cocycles of $\La$ with values in $F\setminus
\{0\}$. One can show that two families define graded-isomorphic tori
if and only if their cohomology classes coincide. \end{rev}

\begin{example}[Group algebra] \label{n:groudef} Although, as we have
pointed out, the $c(\la, \nu)$ need not equal $1$ in general, the
family for which all $c(\la,\mu)=1$ satisfies (\ref{n:quant4}) and
so yields an example of a $\La$-torus, called the \textit{group
algebra of $\La$} and denoted $F[\La]$. In particular, this implies
that associative tori exist for all $\La$, and hence Lie tori exist
for all types $(\rma_l, \La)$, $l\ge 2$. \end{example}

The Lie tori that arise as cores of an EALA have type $(S,\La)$
where $\La$ is a free abelian group of finite rank. This condition
on $\La$ follows from the axiom (EA6) or, equivalently from the
axiom (EARS7). We therefore discuss this special case now.

\begin{definition}\label{n:quantordef} Let $\bq =(q_{ij})$ be an
$n\times n$ matrix such that the entries $q_{ij}\in F$ satisfy
$q_{ii} = 1 = q_{ij}q_{ji}$ for all $1\le i,j\le n$. The
\emph{quantum torus} associated to $\bq$ is the associative algebra
$\FF_\bq$ presented by the generators $t_i, t_i^{-1}$, $1\le i \le
n$ subject to the relations \[   t_i t_i^{-1} = t_i^{-1} t_i, \quad
\hbox{and} \quad
    t_i t_j = q_{ij}\, t_j t_i \hbox{ for all } 1\le i,j\le n.
\]
For example, if all $q_{ij} = 1$, then $\FF_\bq = F[t_1^{\pm 1},
\ldots, t_n^{\pm 1}]$ is the Laurent polynomial ring in $n$
variables. Thus, a general $\FF_\bq$ is a non-commutative version of
$F[t_1^{\pm 1}, \ldots, t_n^{\pm 1}]$, the coordinate ring of the
$n$-dimensional algebraic torus $(F\setminus \{0\})^n$, which
explains the name ``quantum torus''.
\end{definition}

\begin{exercise}\label{n:ex:quntor}  Let $\FF_\bq$ be a quantum torus. Show:

(a) $\FF_\bq = \bigoplus_{\la \in \ZZ^n} F t^\la$ for $t^\la =
t_1^{\la_1} \cdots t_n^{\la_n}$.

(b) The $t^\la$ satisfy the multiplication rule $t^\la t^\mu =
c(\la,\mu)t^{\la + \mu} $ with $$c(\la,\mu) = \ts \prod_{1\le j<i\le
n} \, q_{ij}^{\la_i \mu_j}.$$

(c) $\FF_\bq$ is an associative torus of type $\ZZ^n$.

(d) Every associative torus of type $\ZZ^n$ is graded-isomorphic to
some quantum torus $\FF_\bq$.

(e) The centre $Z(\FF_\bq) = \{ z \in \FF_\bq : [z,\FF_\bq]=0\}$ of
the associative algebra $\FF_\bq$ is a graded subspace of $\FF_\bq$,
namely $Z(\FF_\bq)= \bigoplus_{\ga \in \Ga} F t^\ga$, where $\Ga$ is
a subgroup of $\ZZ^n$ given by \begin{align*}
 \Ga &= \{ \ga\in \ZZ^n : c(\ga, \mu) = c(\mu,\ga) \hbox{ for all }
              \mu \in \ZZ^n \}
\\ &= \{ \ga \in \ZZ^n : \ts \prod_{j=1}^n \, q_{ij}^{\ga_j}=1
   \hbox{ for } 1\le i \le n\}. \end{align*}
Moreover, the following are equivalent: \begin{itemize}
 \item[(i)] All $q_{ij}$ are roots of unity.
 \item[(ii)] $[\ZZ^n : \Ga]< \infty$.
 \item[(iii)] $\FF_\bq$ is finitely generated as a module over its
centre $Z(\FF_\bq)$.
\end{itemize}
\end{exercise}
 \sm

Combining this exercise with Th.~\ref{n:typeAcl} we obtain the
classification of the cores of EALAs of type $\rma$:

\begin{corollary}[{\cite[Th.~2.65]{bgk}}] \label{n:bgkcor}  The Lie algebra $\lsl_N(\FF_\bq)$
for a quantum torus $\FF_\bq$ is a centreless Lie torus of type
$(\rma_{N-1}, \ZZ^n)$. Conversely, any centreless Lie torus of type
$(\rma_{N-1}, \ZZ^n)$ with $N\ge 4$ is graded-isomorphic to some
$\lsl_N(\FF_\bq)$. \end{corollary}

The attentive reader will have noticed that we didn't say anything
about Lie tori of type $(\rma_1, \La)$ and $(\rma_2,\La)$ in
Th.~\ref{n:typeAcl} and Cor.~\ref{n:bgkcor}. Of course, $\lsl_3(A)$
of an associative $\La$-torus will be centreless Lie tori of type
$(\rma_2,\La)$. The limitation $N\ge 4$ in Th.~\ref{n:typeAcl} is
justified, since for $N=3$ there are more examples: One needs
coordinate algebras $A$ which are no longer associative but only
alternative. Moreover, one needs to replace the matrix algebra
$\lsl_3(A)$ by something more general, a Tits-Kantor-Koecher algebra
or an abstractly defined Lie algebra, see \cite{bgkn} for details.

Analogous remarks apply for the $\rma_1$-case, in which the
coordinates come from certain Jordan algebras (called \emph{Jordan
tori}) and in which $\lsl_2$ has to be replaced by a Jordan algebra.
One now has classification theorems for Lie tori of all types. The
references up to 2007 for each type are listed at the beginning of
\S7 in \cite{AF:isotopy}. An additional recent reference is
\cite{NT} for $S=\rmb_2$.

\subsection{Some more easy examples of Lie tori} \label{n:sec:lietorex}

We describe some more easy examples, where easy means that they do
not require some knowledge of non-associative algebras, like Jordan
algebras, alternative or structurable algebras.  \sm

\begin{example}[$\La=\{0\}$]
\label{n:lietonu} Let $\g$ be a finite-dimensional split simple Lie
algebra with splitting Cartan subalgebra $\frh$. Then $\g$ has a
root space decomposition $\g= \bigoplus_{\xi \in S} g_\xi$ where
$\g_0=\frh$ and $S$ is the root system of $(\g, \frh)$, a finite
reduced root system. Since $\g$ is simple, $S$ is also irreducible.
Using standard properties of finite-dimensional split simple Lie
algebras, it is easy to check that \textit{$\g= \bigoplus_{\xi \in
S} \g_\xi$ is a Lie torus of type $(S, \{0\})$.}

Conversely, if $L$ is a Lie torus of type $(S,\{0\})$, then $L$ is a
finite-dimensional split simple Lie algebra. Indeed, $L=\g$ in the
notation of Prop.~\ref{n:liesplit}.

Note that this fits nicely the picture of EALAs of nullity $0$,
which we have characterized in Prop.~\ref{n:fdeala} as
finite-dimensional split simple Lie algebras.
\end{example}

\begin{example} \label{n:lieuntwistgen} As in the previous Example~\ref{n:lietonu} let $\g$ be a
finite-dimensional split simple Lie algebra with splitting Cartan
subalgebra $\frh$ and root system $S$. We would like to consider a
Lie algebra of the form $\g \ot A$ where $A$ is an associative
algebra. For $\g$ of type $\rma$ we could take non-commutative
``coordinates'' $A$ to get a Lie torus, see \S\ref{n:sec:lietypeA}.
However, for $\g$ not of type $\rma$ the algebra $A$ must be
commutative in order to get a Lie algebra.\sm

Therefore, we let $A= \bigoplus_{\la \in \La} A^\la$ be a
\emph{commutative} associative torus of type $\La$ and consider $\g
\ot A $, which becomes a Lie algebra (over $F$) by $[u_1 \ot a_1, \,
u_2 \ot a_2] = [u_1, u_2] \ot a_1 a_2$. It is a \emph{centreless Lie
torus of type $(S,\La)$ with respect to the homogeneous subspaces}
\[ (\g \ot A)_\xi^\la = \g_\xi \ot A^\la.  \]
Note that the support of the $\scQ(S) \oplus \La$-graded Lie algebra
$\g \ot A$ is the set \[
    \supp_{\scQ(S) \oplus \ZZ^n} \g \ot A = S \times \La,\]
and that $\g\ot A$ is an invariant Lie torus with respect to the
bilinear form
\[
     (x \ot a^\la \mid y \ot b^\mu) = \ka(x,y) \, (a^\la b^\mu)_0
\]
where $\ka$ is the Killing form of $\g$ and $c_0$ for $c\in A$ is
the $0$-component of $c$. This example works for any type of $S$.
But it yields all examples only for special types of $S$.
\end{example}

\begin{theorem} \label{n:typede} Any centreless Lie torus of type $(S,\La)$ for $S$ of type
$\rmd_l$, $l\ge 4$ or $\rme_l$, $l=6,7,8$, is graded-isomorphic to
an example as in {\rm \ref{n:lieuntwistgen}} for $\g$ of the
corresponding type and $A$ a commutative associative torus of type
$\La$. \end{theorem}

\begin{proof} The proof is analogous to the proof of
Th.~\ref{n:typeAcl}: One applies the Coordinatization Theorem of
\cite{bm} to get that $L$ has the form $\g \ot A$ for some
commutative associative $F$-algebra. One then has to discuss when
such a Lie algebra is a Lie torus. This is the case exactly when $A$
is a torus. \end{proof}

\begin{example}[Untwisted multiloop algebras] \label{n:lieuntwist}For EALAs it is of interest to describe the centreless Lie tori
of type $(S,\La)$ with $\La$ a free abelian group of finite rank,
say of rank $n$. Hence $\La\cong \ZZ^n$. It is immediate that $\g
\ot A$ is a Lie torus of type $(S,\ZZ^n)$ if and only if $A$ is a
commutative quantum torus, i.e., a Laurent polynomial ring in
several, say $n$ variables. In other words, these are the untwisted
multiloop algebra of (\ref{n:untloo}),
 \[ L(\g) = \g \ot_F F[t_1^{\pm 1}, \ldots, t_n^{\pm 1}]
 \]
Hence by Th.~\ref{n:ucelie} the universal central extension of
$L(\g)$, the toroidal Lie algebras of \S\ref{n:sec:toroidal} are
also Lie tori. Finally, Th.~\ref{n:typede} has the following
corollary. \end{example}

\begin{corollary}[{\cite{bgk}}] \label{n:bgk} Any centreless Lie
torus of type $(S,\ZZ^n)$ with $S=\rmd_l$, $l\ge 4$ or $S=\rme_l$,
$l=6,7,8$, is graded-isomorphic to an untwisted multiloop algebra
$L(\g)$ as in Example~{\rm \ref{n:lieuntwist}}.
\end{corollary}

Perhaps the reader now expects that the next example will be the
general multiloop algebras $L(\g, \boldsi )$ defined in
(\ref{n:multdeff}). However, an arbitrary multiloop algebra is in
general not a Lie torus, see \cite[Th.~3.3.1]{abfp2} and
\cite[Th.~5.1.4]{naoi} for a characterization of centreless Lie tori
which are multiloop algebras. But this phenomenon does not occur in
nullity $1$.

\begin{exercise}\label{n:ex:twiaff} Verify that the loop algebra $L(\g,\si)$
of (\ref{n:oneloop}) is an invariant Lie torus of type $(S,\ZZ)$
where $S$ is the root system of Table~\ref{n:table1}.
\end{exercise}

%

\section{The construction of all EALAs}\label{n:sec:const}

Recall Th.~\ref{n:ealcor}: If $(E,H)$ is an EALA, its core $E_c$ and
its centreless core $E_{cc}$ are Lie tori, the latter being an
invariant Lie torus. Moreover, if $(S,\La)$ is the type of $E_c$ and
$E_{cc}$ then $\La$ is a free abelian group of finite rank. Thus:
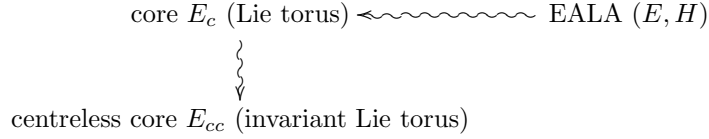
\[ \xymatrix{
    \hbox{core $E_c$ (Lie torus)} \ar@{~>}[d]
    & \hbox{EALA $(E,H)$} \ar@{~>}[l] \\
     \hbox{centreless core $E_{cc}$ (invariant Lie torus)}
}\] In this chapter we will reverse the process, starting from an
invariant Lie torus we will construct an EALA. \sm

To motivate the construction it is useful to look again at the
construction in Example~\ref{n:sec:aff}~and~\ref{n:sec:ealaone} of
an affine Kac-Moody Lie algebra. It can be summarized as follows: We
start with a twisted loop algebra $\scL = L(\g,\si)$, which as we
now know is an invariant Lie torus (Ex.~\ref{n:ex:twiaff}). We then
take a central extension $\tilde \scL$, which in this example is the
universal central extension and hence by Th.~\ref{n:ucelie} again a
Lie torus (of course, one can also verify this directly in this
example). Finally, we add some (not all) derivations to $\tilde L$
to get an affine Kac-Moody Lie algebra, and all affine Kac-Moody Lie
algebras are obtained in this way (Kac's Realization
Theorem~\ref{n:th-kac}). To summarize, using the EALA terminology:
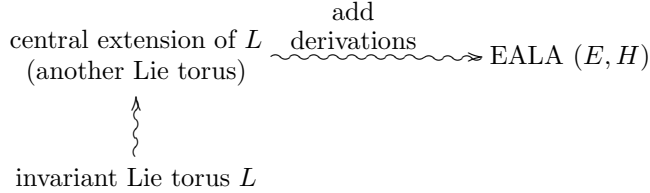
\[ \xymatrix@C=80pt{
   { \begin{matrix}\hbox{central extension of $L$}\\
              \hbox{(another Lie torus)} \end{matrix} }
    \ar@{~>}[r]^{\txt{add \\ derivations}}
          & \hbox{EALA $(E,H)$}  \\
     \hbox{invariant Lie torus $L$} \ar@{~>}[u]
}\] To do something like this in general, one faces the following
two problems.

(A) An invariant Lie torus has in general many central extension.
For example, the untwisted multiloop algebra $L=\g\ot F[t_1^{\pm 1},
\ldots, t_n^{\pm 1}]$ is an invariant Lie torus by
Ex.~\ref{n:lieuntwist}. If $n\ge 2$, its universal central extension
has an infinite-dimensional centre, a result we already mentioned in
\S\ref{n:sec:toroidal}, see in particular Th.~\ref{n:ucemm}. Hence,
there are many possible central extensions. Should we only consider
the universal central extension?

(B) Which derivations should we add? Already in the affine case did
we not add all derivations, as follows for example from
Ex.~\ref{n:ex:derloop}! \sm

It turns out that the two problems are closely related, and we will
solve both at the same time. Rather than taking a $2$-step approach,
we will take one big step, by taking what one may call an
\textit{affine extension} (after all, the result will be an extended
affine Lie algebra). In fact, affine extensions are a special case
of so-called \emph{double extension}, see for example \cite{borde}.
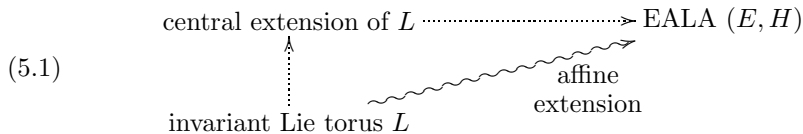
\begin{equation}   \label{n:outline} \vcenter{
  \xymatrix@C=80pt{
    \hbox{central extension of $L$} \ar@{.>}[r]
    & \hbox{EALA $(E,H)$}  \\
     \hbox{invariant Lie torus $L$} \ar@{.>}[u]
       \ar@{~>}[ur]_>>>>>>>>>{\txt{affine \\ extension}}
 } }  \end{equation}

The key idea is based on the construction of a $2$-cocycle in
Ex.~\ref{n:generic}: Any subspace $D$ of skew-symmetric derivations
will give rise to a $2$-cocycle and hence to a central extension.
But not only do we get examples of central extensions. By
Th.~\ref{n:thgencen} and Ex.~\ref{n:uebcent4}, up to isomorphism all
central coverings of $L$ are of the form $\rmE(L,D,\psi_D)$ for some
graded subspace $D$ of $\SDer(L)$ with $D\cap \IDer(L) = \{0\}$.
Observe that we can indeed apply this theorem: The invariant Lie
torus $L$
\begin{itemize}

\item[\rm (i)] is perfect by Ex.~\ref{n:ex:lietor} and is finitely
generated as Lie algebra by Th.~\ref{n:torfg} (recall that $\La$ is
free of finite rank, where $(S,\La)$ is the type of $L$),

\item[\rm (ii)] has finite homogeneous dimension, even bounded homogeneous
dimension also by Th.~\ref{n:torfg}, and

\item[\rm (iii)] has an invariant nondegenerate $\La$-graded
symmetric bilinear form, by definition of an invariant Lie torus.
\end{itemize}
But we need more than just a central covering. For example, the
axiom (EA2) requires that we construct an $\ad$-diagonalizable
subalgebra $H$ for which the subspaces $L_\xi^\la$, $\xi \ne 0$, are
root spaces, as can be seen from Th.~\ref{n:ealcor}. By
Lemma~\ref{n:weisp} we can realize the subspaces $L_\xi$ as root
spaces of some natural subalgebra $\frh \subset L$. But we do not
have a result, which describes the subspaces $L^\la$ in a similar
fashion, i.e., as root spaces of some toral subalgebra. There is in
fact no natural choice of a subalgebra to do so. Rather, we will
distinguish these subspaces ``externally'', i.e., via an action of
some non-inner derivation algebra. The required formalism to do
this, is described in the next section. This has nothing to do with
Lie algebras. Rather, it is a topic in the theory of graded vector
spaces, and we will therefore describe it in this setting.

\subsection{Degree maps} \label{n:sec:degree}

In this section, $V$ is vector space over a field $F$, which could
be of arbitrary characteristic until Prop.~\ref{n:degprop}(b). Also,
$\La$ denotes an arbitrary abelian group. We recall that a
$\La$-grading of $V$ is simply a direct vector space decomposition
of $V$ by a family $(V^\la: \la \in \La)$ of subspaces $V^\la
\subset V$. Our goal in this section is to present a method
describing the homogeneous subspaces $V^\la$ of a given
$\La$-grading of $V$ as the joint eigenspaces of a subspace of
diagonalizable endomorphisms. \sm

To motivate the construction, let us first look at the converse,
namely inducing a grading of $V$ via the action of endomorphisms. We
will say that a subspace $T\subset \End_F(V)$ is a \emph{subspace of
simultaneously diagonalizable endomorphisms\/}, if
\begin{align} \label{n:degree1}
  V &= \textstyle \bigoplus_{\la \in T^*} V^\la \quad \hbox{for} \\
  V^\la &= \{ v\in V: t(v) = \la(t) v \hbox{ for all } t \in T\}.
        \nonumber
\end{align}
In this case, $T$ obviously  consists of pairwise commuting
diagonalizable endomorphisms. Conversely, it is well-known that a
finite-dimensional subspace of pairwise commuting diagonalizable
endomorphisms is a subspace of simultaneously diagonalizable
endomorphisms (this is no longer true if $T$ is
infinite-dimensional). Observe that the decomposition
(\ref{n:degree1}) is a grading of $V$ by the group
$\Span_\ZZ(\supp_{T^*} V)$ where $\supp_{T^*} V =\{ \la \in T^* :
V^\la \ne 0 \}$. Our goal is to realize a given grading of $V$ in
this way. \sm

To do so, we will use the $F$-vector space
\[ \rmD(\La) = \Hom_\ZZ(\La, F) \]
consisting of all maps $\theta : \La \to F$ which are $\ZZ$-linear:
$\thet(\la_1 + \la_2) = \thet(\la_1) + \thet(\la_2)$ for all $\la_i
\in \La$. This is an $F$-vector space by defining for $\thet,
\thet_i \in \rmD(\La)$ and $s\in F$ the sum $\thet_1 + \thet_2$ and
the scalar multiplication $s\thet$ by $(\thet_1 + \thet_2)(\la) =
\thet_1(\la) + \thet_2(\la)$ and $(s\thet)(\la) = s (\thet(\la))$.

\begin{exercise} \label{n:uebdegree} (a) Show $\rmD(\La) \cong \Hom_F( \La \ot_\ZZ F, F) =
(\La \ot_F F)^*$. Thus $\rmD(\La)$ is naturally a dual vector space.

(b) If $\La$ is free of rank $n$, say with $\ZZ$-basis $\veps_1 ,
\ldots, \veps_n$, then $\rmD(\La) = F \pa_1 \oplus \cdots \oplus F
\pa_n$ where $\pa_i\in \rmD(\La)$ is defined by $\pa_i(\sum_j m_j
\veps_j ) = m_i$. In particular, $\dim_F \rmD(\La) = n$.
\end{exercise}

We now suppose that $V= \bigoplus_{\la \in \La} V^\la$ is a
$\La$-grading of the vector space $V$. Any $\thet \in \rmD(\La)$
defines an endomorphism $\pa_\thet \in \End_F(V)$ by
\[ \pa_\thet(v^\la) = \thet(\la) \, v^\la \quad\hbox{for }
v^\la \in V^\la. \] We put
\[ \euD(V) = \{ \pa_\thet : \thet \in \rmD(\La)\} \]
and call the elements of $\euD(V)$ \emph{degree maps\/}. If
$A=\bigoplus_{\la \in \La} A^\la$ is a $\La$-grading of an algebra
$A$, the maps $\pa_\thet$ are derivations and $\euD(A)$ is called
the space of \emph{degree derivations\/}.

The map $\pa : \rmD(\La) \to \euD(V)$ is clearly $F$-linear and
surjective by definition. Its kernel is $\{\thet \in \rmD(\La) :
\thet(\supp_\La V) = 0\}$. To make $\pa$ an isomorphism we will
\begin{equation} \label{n:degree2}
 \hbox{\it from now on assume $\Span_\ZZ (\supp_\La V) =
\La$.} \end{equation}  As we have pointed out at previous occasions,
this is not a serious assumptions since one can always replace $\La$
by $\Span_\ZZ (\supp_\La V)$ without changing the given grading.
Since now $\pa$ is an isomorphism, we can define a linear form
$\ev_\la\in \euD(V)^*$ for every $\la \in \La$:
\[  \ev_\la(\pa_\thet) = \thet(\la) \]
The $F$-linear map
\[ \ev : \La \to \euD(V)^*, \quad \la \mapsto \ev_\la \]
is called the \emph{evaluation map\/}. By construction,
\begin{equation} \label{n:degree3}
  V^\la \subset \{ v\in V : d(v) = \ev_\la(d)v \hbox{ for all }
      d\in \euD(V)\}  \end{equation}
since for $d=\pa_\thet$ and $v\in V^\la$ we have $\pa_\thet(v^\la) =
\thet(\la) v^\la = \ev_\la(\pa_\thet) v^\la$.

\begin{definition} In the setting of above, i.e., $V=\bigoplus_{\la \in
\La} V^\la$ is $\La$-graded and (\ref{n:degree2}) holds, we will say
that a subspace $T \subset \euD(V)$ \emph{induces the $\La$-grading
of $V$\/} if
\[
   V^\la = \{ v\in V : t(v) = \ev_\la(t)v \hbox{ for all }
      t\in T\}
\] holds for all $\la \in \La$.
\end{definition}

\begin{proposition}\label{n:degprop} Let $V= \bigoplus_{\la \in \La} V^\la$
be a $\La$-grading of the vector space $V$ such that {\rm
(\ref{n:degree2})} holds.  \sm

{\rm (a)} A subspace $T\subset \euD(V)$ induces the $\La$-grading
of\/ $V$ if the restricted evaluation map
\[ \ev_T : \La \to T^*, \quad \ev_T(\la ) = \ev_\la \mid_T
\] is injective. \sm

{\rm (b)} Suppose $F$ has characteristic $0$ and $\La$ is
torsion-free, i.e., $n\la = 0$ for some $n\in \ZZ$ implies $\la =
0$. Then $\La$ embeds into the $F$-vector space $U=\La \ot_\ZZ F$
and for every subspace $S\subset \rmD(\La)$ separating the points of
$\La$ in $U$ the corresponding subspace $T=\pa(S) \subset \euD(V)$
induces the $\La$-grading of $V$. In particular, this holds for
$\euD(V)$ itself. \end{proposition}

\subsection{The centroid of Lie algebras, in particular of Lie tori} \label{n:sec:centr}

After the intermezzo on how to induce gradings of vector spaces in
the previous section \ref{n:sec:degree} we now come back to Lie
algebras, but not immediately to Lie tori and EALAs. Of course, the
topic of this section is motivated by the over-all goal of this
chapter: The construction of EALAs from  Lie tori using certain
subspaces of derivations. The derivations, which in
\ref{n:sec:genconstr} will be used in the general construction, are
products of degree maps, studied in \ref{n:sec:degree}, and
so-called centroidal transformations, to which this section is
devoted. \sm

The basic idea of the centroid of a Lie algebra (or of any algebra
for that matter) is that it identifies the largest ring over which
the given algebra can be considered as an algebra. For example, if
one studies the real Lie algebra $L$ which is $\lsl_n(\CC)$
considered as a real Lie algebra by restricting the scalars to
$\RR$, the centroid will be $\cong \CC$ and will thus indicate that
$L$ can also be considered as a complex Lie algebra.

In general, the centroid will not be a field but only a
(commutative) ring. Hence, considering a Lie algebra as algebra over
its centroid, necessitates that in the following definition and the
Lemma~\ref{n:centlem} after it we will deviate from our standard
assumption and consider Lie algebras over rings. The definition of a
Lie algebra $L$ defined over a ring, say $k$, is not surprising: $L$
is a $k$-module with a $k$-bilinear map $[.,.]: L \times L \to L$
which is alternating, i.e., $[l,l]=0$ for all $l\in L$, and
satisfies the Jacobi identity.

\begin{definition}[{\cite[Ch.~X]{jake}}] \label{n:centdef} The
 \emph{centroid $\Cent_k(L)$\/}
of a Lie algebra $L$ defined over a ring $k$ is defined as
\[ \Cent_k(L) = \{ \chi \in \End_k(L) : \chi([l_1, l_2]) =
    [l_1, \chi(l_2)] \hbox{ for all } l_1, l_2\in L\}. \]
Of course, $\chi \in \Cent_k(L)\Iff \chi([l_1, l_2]) = [\chi(l_1),
l_2]$ for all $l_1, l_2\in L$. It is important to indicate $k$ in
the notation $\cent_k(L)$ since the centroid depends on the base
ring $k$.

We have $k\Id_L \subset \Cent_k(L)$ for every $L$. One calls $L$
\emph{central\/} if the map $k \to \Cent_F(L)$, $s \mapsto s\Id_L$,
is an isomorphism, and one says that $L$ is \emph{central-simple} if
$L$ is just that: central and simple. \sm

Let $L=\bigoplus_{\la \in \La} L^\la$ be a Lie algebra graded by an
abelian group $\La$. We can then also define the \emph{$\La$-graded
centroid} as \[
 \grCent_k(L) =    \grEnd_k(L) \cap  \Cent_k(L) =
\ts \bigoplus_{\la \in \La} \Cent_k(L)^\la
\]
where $\Cent_k(L)^\la$ consists of the centroidal transformations
which have degree $\la$: $\chi(L^\mu) \subset L^{\la + \mu}$ for all
$\mu \in \La$.\end{definition}

\begin{example} As an immediate example we calculate the centroid
of the Lie algebra $L=\lsl_2(\CC)$, considered as real Lie algebra
by restricting the scalars to $\RR$. We let $(e,h,f)$ be an
$\lsl_2$-triple in $\lsl_2(\CC)$. Then the relations $[h,ce]=2ce$
and $[h,cf]=-2cf$ for $c\in \CC$ show that $\chi(ce)$ and $\chi(cf)$
are uniquely determined by $\chi(h)$. For example, $2 \chi(ce) =
\chi([h,ce]) = [\chi(h), ce]$. Moreover, $\chi(h) \in \CC h$ because
$[\chi(h), ch]=\chi([h,ch])=0$. Hence $\dim_\RR\cent_\RR(L) \le 2$.
On the other side, $\CC \Id_L \subset \cent_\RR(L)$ is clear, whence
$\CC \Id_L = \cent_\RR(L)$.

We leave it to the reader to show $\cent_\RR(L) = \CC \Id_L$ for
$L=\lsl_n(\CC)$ considered as real Lie algebra, without using any of
the results mentioned below! \sm

The following lemma gives a mathematical meaning to the claims made
before the Def.~\ref{n:centdef}, and lists the most important
properties of the centroid of Lie algebras which are not necessarily
Lie tori.
\end{example}

\begin{lemma}[Folklore]\label{n:centlem} Let $L$ be a Lie algebra defined
over a ring $k$. \sm

{\rm  (a)} The centroid of $L$ is always a unital associative
subalgebra of the endomorphism algebra $\End_k(L)$ of $L$. Hence
$\cent_k(L)$ is a $k$-algebra and $L$ becomes a $\Cent_k(L)$-module
by defining the action of $\Cent_k(L)$ on $L$ by $\chi \cdot l =
\chi(l)$ for $\chi \in \cent_k(L)$ and $l\in L$. \sm

{\rm (b)} If the centroid of $L$ is commutative, then with respect
to the action of $\cent_k(L)$ on $L$ defined in {\rm (a)}, $L$ is a
Lie algebra over the ring $\cent_k(L)$. Moreover, $L$ is central as
a Lie algebra over its centroid. \sm

{\rm (c)} If $L$ is perfect, its centroid is commutative and does
not depend on the base ring $k$: $\cent_k(L) = \cent_\ZZ ({_\ZZ L})$
where $_\ZZ L$ is the Lie algebra $L$ with scalars restricted to
$\ZZ$. \sm

{\rm (d)} If $L$ is simple, its centroid is a field and $L$ as a Lie
algebra over the field $\cent_F(L)$ is central-simple. In
particular: \begin{itemize} \item[(i)] a finite-dimensional simple
Lie algebra over an algebraically closed field $F$ is
central-simple, and

\item[(ii)]  the centroid of a simple real Lie algebra $L$ is either $\cong \RR \Id$,
in which case $L$ is central-simple, or is $\cong \CC\Id$, in which
case $L$ is a simple complex Lie algebra, considered as a real Lie
algebra. \end{itemize} \sm

{\rm (e)} Suppose $L$ is $\La$-graded. Then $\grCent_k(L)$ is a
$\La$-graded subalgebra of the full centroid $\cent_k(L)$. Moreover,
$\grCent_k(L) = \cent_k(L)$ if $L$ is finitely generated as an
ideal, i.e., there exist $l_1, \ldots , l_n \in L$ such that the
ideal generated by $l_1, \ldots, l_n$ is all of $L$. \sm

{\rm (f)} If $\chi \in \cent_k(L)$ and $d\in \Der_k(L)$, then $\chi
\circ d \in \Der_k(L)$. With respect to this operation, $\Der_k(L)$
is a $\cent_k(L)$-module and $\IDer(L)$ is a submodule of the
$\cent_k(L)$-module $\Der_k(L)$.\end{lemma}

The proof of this lemma is a straightforward exercise, which the
reader will be asked to do now. The exercise also lists some
interesting additional facts on the centroid.

\begin{exercise} \label{n:ex:centroid} (a) For any $\chi \in \cent_k(L)$ the kernel $\Ker  \chi$ and the
image ${\rm Im}\, \chi$ are ideals of $L$ satisfying $[\Ker \chi,
{\rm Im}\, \chi] = 0$. \sm

(b) Prove Lemma~\ref{n:centlem}. For part (c) of the Lemma use (a)
above. \sm

(c) If $L$ is perfect, any $\chi \in \cent_k(L)$ is symmetric with
respect to any invariant bilinear form on $L$. \sm
\end{exercise}

Here is the result, which describes the centroid of the Lie algebras
of interest in this chapter. We will use the notion of an
associative torus, introduced in \ref{n:asstor} and further
discussed in  \ref{n:twigrrev}--\ref{n:ex:quntor}.

\begin{proposition}[{\cite[Prop.~3.13]{BN}}] \label{n:propcenli} Let $L= \bigoplus_{\xi \in S,
\, \la \in \La} L_\xi^\la$ be a centreless Lie torus of type
$(S,\La)$. \sm

{\rm (a)} With respect to the $(\scQ(S) \oplus \La)$-grading of $L$
we have \begin{equation} \label{n:cent1}  \cent_F(L) = \ts
\bigoplus_{\la \in \La} \cent_F(L)_0^\la
     = \grCent_F(L).\end{equation}
In particular, $\chi(L_\xi) \subset L_\xi$ for any $\chi \in
\cent_k(L)$ and $\xi \in S$. \sm

{\rm (b)} Moreover, with respect to the decomposition {\rm
(\ref{n:cent1})} the centroid $\cent_F(L)$ is an associative
commutative torus of type $\Ga$, where $\Ga = \supp_\La \cent_F(L)$
is a subgroup of $\La$, hence a twisted group algebra over $\Ga$.
\sm

{\rm (c)} In particular, if $\La$ is free abelian of finite rank
$n$, the centroid $\cent_F(L)$ is graded-isomorphic to $F[\Ga]$, the
group algebra of\/ $\Ga$ as defined in {\rm \ref{n:groudef}}, and is
thus isomorphic to a Laurent polynomial ring in $\nu$ variables,
$0\le \nu \le n$. Moreover, $L$ is a free module over its centroid.
\end{proposition}

\begin{proof} Parts (a) and (b) of this proposition are proven in
\cite[Prop.~3.13]{BN}. Part (c) is (\cite[Th.~7]{n:tori}). The first
part of (c) follows from (b): A twisted group algebra over a free
group is a group algebra. The second part is a special case of a
general fact: Any graded module over an associative torus is free.
\end{proof}

\begin{example} Let $L=\lsl_N(A)$ for $A$ an associative
$F$-algebra, see \ref{n:sec:lietypeA}, and let $Z(A)= \{ z\in A :
[z,A]=0\}$ be the centre of the associative algebra $A$. Any $z\in
Z(A)$ induces a centroidal transformation $\chi_z$ defined by
mapping $x=(x_{ij}) \in \lsl_N(A)$ to $\chi_z(x) = (zx_{ij})$. It is
easily seen (\cite[7.9]{n:persp}) that
\[ Z(A) \to \cent_F(\lsl_N(A)), \quad z \mapsto \chi_z \]
is an isomorphism of $F$-algebras (the only non-obvious part is
surjectivity).

Let us now specialize to the case of a Lie torus $\lsl_N(A)$ of type
$(\rma_{N-1}, \ZZ^n)$. Thus, by Cor.~\ref{n:bgkcor}, $A=\FF_\bq$ is
a quantum torus. A description of the centre $Z(\FF_\bq)$ is given
in Ex.~\ref{n:ex:quntor}(e) (see \cite[Prop.~2.44]{bgk} for a
proof): $Z(\FF_\bq) = \bigoplus_{\ga \in \Ga} F t^\ga$ where $\Ga$
is the subgroup
\[ \Ga = \{ \ga \in \ZZ^n : \ts \prod_{j=1}^n \, q_{ij}^{\ga_j}=1
   \hbox{ for } 1\le i \le n\}. \]
of $\ZZ^n$.  The centre of $\FF_\bq$ is therefore isomorphic to a
Laurent polynomial ring in, say, $\nu$ variables, as claimed in
Prop.~\ref{n:propcenli}(c). To see that the inequalities $0 \le \nu
\le n$ stated there are sharp, we consider the quantum torus
associated to the matrix
\[ \bq=\begin{bmatrix} 1 &q \\ q^{-1} & 1 \end{bmatrix}. \]
Specializing the description of $\Ga$ above we get
\[
  \Ga = \begin{cases} \{0\}, & \hbox{$q$ not a root of unity}, \\
     m\ZZ\oplus m \ZZ, &\hbox{$q$ an $m$th root of unity}.
   \end{cases}
\]
Hence $\nu = 0$ in the first case and $\nu=2=n$ in the second case.

However, the following result says that this is the only case in
which the centroidal grading group $\Ga$ has smaller rank than
$\La$.
\end{example}

\begin{theorem}[{\cite[Th.~7]{n:tori}}] \label{n:fgcto} Let $L$ be a centreless Lie
torus of type $(S,\ZZ^n)$ with $S$ not of type $\rma$. Then $[\ZZ^n
: \Ga]< \infty$ and $L$ is a free $\cent_F(L)$-module of finite
rank. \end{theorem}

This result, together with the Realization Theorem of \cite{abfp}
implies that an invariant Lie torus of type $(S,\ZZ^n)$, $S\ne
\rma_l$, is graded-isomorphic to a multiloop algebra as defined in
(\ref{n:multdeff}). A characterization of which multiloop algebras
are Lie tori is the main result of \cite{abfp2}. A more general
approach to realizing Lie tori as multiloop algebras is developed in
\cite{naoi}.

It is easy to verify Th.~\ref{n:fgcto} in case $L$ is a Lie torus of
type $(S,\ZZ^n)$ and $S$ of type $\rmd$ or $\rme$. As we have seen
in Th.~\ref{n:typede}, in this case $L=\g\otimes F[t_1^{\pm 1},
\ldots, t_n^{\pm 1}]$. The centroids of these types of Lie algebras
are described in the next example.

\begin{example} \label{n:centgtensora} Let $\g$ be a
finite-dimensional central simple Lie algebra. For example, by
\cite[Remark~3.6]{BN} any  finite-dimensional split simple Lie
algebra is central and thus central-simple. (Over algebraically
closed fields, this also follows from Lemma~\ref{n:centlem}(d).)
Also, let $A$ be an associative commutative $F$-algebra.

A straightforward verification shows that for $s\in F$ and $a\in A$
the map $\chi_{s,a}$, defined by $u \ot b \mapsto su \ot ab$, is a
centroidal transformation of the Lie algebra $\g \otimes A$. It
follows from \cite[Lemma~2.3(a)]{abp2.5} or
\cite[Lemma~1.2]{Azam:tensor} or \cite[Cor.~2.23]{BN} that these are
all the maps in the centroid of $\g \ot A$: \[ F \Id_\g \ot A \cong
\cent_F(\g \otimes A) , \quad \hbox{via $s\otimes a \mapsto
\chi_{s,a}$}.
\]\end{example}

Although this will not be needed in the following, we mention that
the centroid of an EALA is known too.

\begin{proposition}\label{n:centEALA} Let $E$ be an EALA, let $K=E_c$ be its core
and put $D=E/K$. Then $K$ is a central Lie algebra, and
\[ \Cent_F(E) = F \Id_E \,\oplus\, \scV(K), \quad \scV(K)=
 \{\chi \in \cent_F(E): \chi(K)=0\}.
\]
As a vector space, the ideal $\scV(K)$ of $\cent_F(E)$ is
canonically isomorphic to the $D$-module homomorphisms $D \to Z(K)$:
\[
 \scV(K) \cong \Hom_D(D, Z(K)).\]
\end{proposition}

This is proven in \cite[Cor.~4.13]{BN}. Observe that the reference
to \cite[Th.6]{n:eala} in the proof of \cite{BN} can now be replaced
by the combination of Th.~\ref{n:torfg}(c) and Ex.~\ref{n:uebcent4}.

\subsection{Centroidal derivations of Lie tori} \label{n:sec:centder}

In this section $L$ is a centreless Lie torus of type $(S,\La)$.
Regarding $L$ as a $\La$-graded Lie algebra, the results of section
\ref{n:sec:degree} apply and provide us with the subspace
\[ \euD = \euD(L)= \{ \pa_\thet : \thet \in \rmD(\La)\}
\]
of degree derivations of $L$. Moreover, we can apply
Lemma~\ref{n:centlem}(f) and get that $\chi \circ \pa_\thet \equiv
\chi \pa_\thet$ is a derivation for any $\chi \in \Cent_F(L)$. We
call the elements of
\[ \CDer_F(L) = \Cent_F (L) \, \euD \]
\emph{centroidal derivations}. (A notion of centroidal derivations
for arbitrary $\La$-graded Lie algebra is developed in
\cite[4.9]{n:persp}.) Recall from Prop.~\ref{n:propcenli} that
$\Cent_F(L) = \bigoplus_{\ga\in \Ga} \Cent_F (L)^\ga$ is a
commutative associative torus of type $\Ga$, where $\Ga$ is a
subgroup of $\La$. Since $\euD$ consist of degree $0$ endomorphisms,
$\CDer(L)$ is $\Ga$-graded,
\begin{align} \label{n:centder1} \CDer_F(L) &= \textstyle
\bigoplus_{\ga\in \Ga} \CDer_F(L)^\ga \quad
       \hbox{for} \\
   \CDer_F(L)^\ga &= \Cent_F(L)^\ga \, \euD = \Cent_F(L) \cap
             \End_F(L)^\ga. \nonumber
\end{align}
It is then easily seen that $\CDer_F(L)$ is a $\Ga$-graded
subalgebra of $\Der_F(L)$. For $\chi^\ga \in \Cent_F(L)^\ga$,
$\chi^\de \in \Cent_F(L)^\de$ and $\thet, \psi \in \rmD(\La)$ the
Lie algebra product of $\CDer_F(L)$ is given by the formula
\begin{equation} \label{n:centder2}
  [\chi^\ga \pa_\thet,\, \chi^\de \pa_\psi]
   = \chi^\ga \chi^\de\, \big( \thet(\de)\,\pa_\psi - \psi(\ga)\,
\pa_\thet \big).
\end{equation}
Thus, $\CDer_F(L)$ is a generalized Witt algebra, see for example
\cite[1.9]{NY}. \sm

Suppose now that $L$ is an invariant Lie torus, say with respect to
the invariant bilinear from $\inpr$. We can then consider the
\emph{skew centroidal derivations}
\[ \SCDer_F(L)= \SDer_F(L) \cap \CDer_F(L), \]
defined as the centroidal derivations which are skew-symmetric with
respect to $\inpr$. This is a $\Ga$-graded subalgebra of
$\CDer_F(L)$ whose homogenous components are given by
\begin{equation}\label{n:centder4}
   \SCDer_F(L)^\ga = \{ \chi^\ga
\pa_\thet :
  \chi^\ga \in \Cent_F(L)^\ga, \thet(\ga) = 0 \}.  \end{equation}
In particular,
\[ \SCDer_F(L)^0 = \euD \]
is a toral subalgebra of $\SCDer_F(L)$ since $[\pa_\thet, \,
\chi^\de \pa_\psi] = \thet(\de) \pa_\psi$ by (\ref{n:centder2}). It
is also of interest to point out that  $[\SCDer_F(L)^\ga, \,
\SCDer_F(L)^{-\ga}] = 0$, which implies that $\SCDer_F(L)$ is a
semidirect product,
\[  \SCDer_F(L) = \euD \ltimes \big(\textstyle \bigoplus_{\ga \ne 0}
      \SCDer_F(L)^\ga\big)
\]
of the toral subalgebra $\euD$ and the ideal spanned by the
homogeneous subspaces of non-zero degree. For the construction of
EALAs, the following theorem is fundamental. \sm

\begin{theorem}[{\cite[Th.~9]{n:tori}}] Let $L$ be an invariant Lie torus of type $(S,\La)$
with $\La$ free of finite rank. Then $\Der_F(L)$ is a semi-direct
product, \begin{align} \label{n:centder3}
   \Der_F(L) &= \IDer_F(L) \rtimes \CDer_F(L), \quad \hbox{hence}\\
   \SDer_F(L) &= \IDer_F(L) \rtimes \SCDer_F(L),   \nonumber
   \end{align}
where $\IDer_F(L)$ denotes the ideal of all inner derivations.
\end{theorem}

Some remarks on the proof of this theorem follow. By
Prop.~\ref{n:propcenli} the centroid of $L$ is a Laurent polynomial
ring. Let $K$ be its field of fractions, a field of rational
functions. As a $\Cent_F(L)$-module, $L$ is torsion-free and hence
$L$ embeds into the Lie $K$-algebra \[ \tilde L = L \ot_{\Cent_F(L)}
K, \] the so-called \emph{central closure of $L$}. If the
$\Cent_F(L)$-module $L$ is finitely generated, its central closure
is a finite-dimensional central-simple Lie algebra. Hence, in this
case $\Der_K(\tilde L) = \IDer(\tilde L)$, from which the theorem
easily follows. If however $L$ is not finitely generated as a
$\Cent_F(L)$-module, then we know from Th.~\ref{n:fgcto} that $L$ is
a Lie torus of type $\rma$. More precisely, as a consequence of the
results in \cite{y1} and \cite{NY} for type $\rma_1$, \cite{bgkn}
for  type $\rma_2$ and Cor.~\ref{n:bgkcor} for type $\rma_l$, $l\ge
3$, such a Lie torus is graded-isomorphic to $\lsl_n(\FF_\bq)$. But
in this case the result follows from \cite[2.17, 2.53]{bgk},
\cite[Th.~1.40]{bgkn} and \cite[Th.~4.11]{NY}. We will discuss the
special case $L=\g\ot F[t_1^{\pm 1}, \ldots, t_n^{\pm n}]$ in
Ex.~\ref{n:cderloo}. To avoid any confusion, we note that the
splitting (\ref{n:centder3}) is not the one proven in
\cite[Th.~3.12]{be:derinv} for arbitrary root-graded Lie algebras.
\sm

The importance of the theorem stems from Th.~\ref{n:thgencen}: It
identifies a natural complement of $\IDer(L)$ in $\SDer_F(L)$.
Hence, up to graded-isomorphism, any graded covering of $L$ has the
form $\rmE(L,D^{\gr *}, \psi_D)$ for a graded subspace $D \subset
\SCDer_F(L)$.  Moreover, since $\euD \subset \SCDer_F(L)$, we can
require that $D^0\subset \euD$ be not too small and use it to
distinguish the homogeneous spaces $L^\la$ by applying
Prop.~\ref{n:degprop}. This will be our approach in section
\ref{n:sec:genconstr}. But first some examples.

\begin{example}\label{n:cderloo}  Let $L = \g \ot A$ where $\g$ is a split simple
finite-dimensional Lie algebra with root system $S$ and where
$A=F[t_1^{\pm 1}, \ldots, t_n^{\pm 1}]$ is a Laurent polynomial ring
in $n$ variables. This is an invariant Lie torus of type
$(S,\ZZ^n)$, see the Examples~\ref{n:lieuntwistgen} and
\ref{n:lieuntwist}. We have seen in (\ref{n:eq:aff11}) that
\[\Der_F(\g \ot A) = \IDer(\g \ot A) \oplus (\Id_\g \ot \Der_F(A)).\] The reader
has (or should have) determined $\Der_F(A)$ in
Ex.~\ref{n:ex:derloop}: $\Der_F(A) = A\, \euD$ where $\euD =
\Span_\ZZ(\{ \pa_i : 1\le i \le n\})$ in the notation of the quoted
exercise. But by Ex.~\ref{n:uebdegree}, $\euD=\euD(A)$ is also the
space of degree derivations of $A$. Since the $\La$-grading of $L=\g
\ot A$ is concentrated in the factor $A$, it follows that $\Id \ot
\euD$ is the space of degree derivations of $L$, whence, by
Example~\ref{n:centgtensora}, we have
\[ \CDer_F(\g \ot A ) = \Id_\g \ot A \, \euD =
       \Id_\g \ot \Der_F(A).   \] Thus, for the invariant Lie torus $\g \ot
A$ the decomposition (\ref{n:eq:aff11}) is the same as the
decomposition (\ref{n:centder3})! \sm

We have seen in Ex.~\ref{n:lieuntwistgen} that  $L$ is an invariant
Lie torus with respect to the tensor product form $\inpr = \ka \ot
\be$ where $\ka$ is the Killing form of $\g$ and where $\be$ is the
bilinear form on $A$ defined by $\be(t^\la, t^\mu)= \de_{\la, -
\mu}$. It is then easy to identify $\SCDer_F(L)$ using
(\ref{n:centder4}). In particular, for $n=1$ we see that $\SCDer(\g
\ot F[t^{\pm 1}]) = F d$, where $d$ is the degree derivation of
(\ref{n:eq:aff3.5}). In particular, this together with
Th.~\ref{n:thgencen} gives a new proof of the theorem, mentioned in
\ref{n:sec:aff}, that the Lie algebra $\tilde \scL(\g,\si)$ of
(\ref{n:aff3.2}) is the universal central extension of the twisted
loop algebra $\scL(\g,\si)$.
\end{example}

\subsection{The general construction} \label{n:sec:genconstr}

Finally, we can describe the ingredients $(L,D,\ta)$ of the general
construction:
\begin{itemize}
 \item $L=\bigoplus_{\xi \in S, \la \in \La} L_\xi^\la$ is an
 invariant Lie torus of type $(S,\La)$ with $\La$ a free abelian
 group of finite rank; we put $\Ga = \supp_\La \cent(L)$, see Prop.~\ref{n:propcenli}.
\sm

 \item $D=\bigoplus_{\ga \in \Ga} D^\ga \subset \SCDer_F(L)$ is a
 graded subalgebra such that the evaluation map
  \begin{equation}\label{n:gencons0}
      \ev_{D^0} : \La \to D^{0\, *}, \quad \la \to \ev_\la \mid_{D^0}
     \hbox{ is injective.} \end{equation}

 \item $\ta : D\times D \to D^{\gr *}$ is an \textit{affine cocycle\/}, i.e.,
 $\ta$ is a bilinear map satisfying for all $d,d_i \in D$
    \begin{align} \label{n:gencons1}
     \ta(d,d) &= 0 \quad \hbox{and} \quad
          \textstyle \sum_\circlearrowleft d_1 \cdot \ta(d_2, d_3)
       =  \sum_\circlearrowleft\ta([d_1, d_2], d_3),\\
      \ta(D^0, D) &= 0, \quad \hbox{and} \quad
        \ta(d_1, d_2)(d_3) = \ta(d_2, d_3)(d_1)
        \label{n:gencons2}
       \end{align}
 \end{itemize}

Recall from Prop.~\ref{n:degprop} that the condition
(\ref{n:gencons0}) implies that $D^0$ induces the $\La$-grading of
$L$, i.e.,
\begin{equation} \label{n:gencons00}  L^\la = \{ l \in L : d^0 (l) =
\ev_\la(d^0) l , \hbox{ for all
  $d^0 \in D^0$}\}. \end{equation}
For example, (\ref{n:gencons0}) holds for $D = \euD = \SCDer_F(L)^0$
or $D$ any graded subalgebra with $D^0 = \euD$. In
(\ref{n:gencons1}), the symbol $\sum_\circlearrowleft$ denotes the
cyclic sum: $\sum_\circlearrowleft d_1 \cdot \ta(d_2, d_3) = d_1
\cdot \ta(d_2, d_3) + d_2 \cdot \ta(d_3, d_1) + d_3 \cdot \ta(d_1,
d_2)$ and analogously for $\sum_\circlearrowleft\ta([d_1, d_2],
d_3)$. Moreover, $d \cdot c$ for $c\in D^{\gr *}$ is the
contragradient action of $D$ on the graded dual space $D^{\gr *}$.
The condition (\ref{n:gencons1}) says that $\ta$ is an \emph{abelian
$2$-cocycle\/}, meaning that $D^{\gr *} \oplus D$ is a Lie algebra
with respect to the product formula \begin{equation}
\label{n:gencons4} [c_1 \oplus d_1, \, c_2 \oplus d_2] = \big( d_1
\cdot c_2 - d_2 \cdot c_1 +\ta(d_1, d_2) \big) \oplus [d_1, d_2]
\end{equation}  for $c_i \in D^{\gr *}$ and $d_i \in D$. Thus,
$$
  \xymatrix{
   0 \ar[r] & D^{\gr *} \ar[r]^(.4)\inc & D^{\gr *} \oplus D
          \ar[r]^(.65){\pr_D} & D \ar[r] & 0 }
$$
is an abelian extension: $D^{\gr *}$ is an abelian ideal, not
necessarily contained in the centre. The conditions in
(\ref{n:gencons2}) will allow us to define a toral subalgebra $H$
and an invariant bilinear form $\inpr$ below. We note that an affine
cocycle is necessarily graded of degree $0$:
  \[  \ta(D^\ga, D^\de) \subset (D^{\gr *})^{\ga + \de} \] for
  $\ga,\de \in \Ga$.
There do exist non-trivial affine cocycles, see
\cite[Rem.~3.71]{bgk} and  \cite{Rao-Moody}. \ms

To data  $(L,D,\ta)$ as above we associate a Lie algebra \[E= L
\oplus D^{\gr *} \oplus D\] with product
 ($l_i \in L$, $c_i \in D^{\gr *}$ and $d_i \in D$)
   \begin{equation} \label{n:gencons3} \begin{split}
   & [l_1 \oplus c_1 \oplus d_1, \, l_2 \oplus c_2 \oplus d_2 ]
     = \big( [l_1, l_2]_L + d_1(l_2) - d_2(l_1) \big) \\
   &\quad  \big(\psi_D(l_1, l_2) +d_1 \cdot c_2 - d_2 \cdot c_1
      + \ta(d_1, d_2) \big)  \oplus  [d_1, d_2].
 \end{split} \end{equation}
Here $[.,.]_L$ is the Lie algebra product of $L$, $d_i(l_j)$ is the
natural action of $D$ on $L$, and $\psi_D$ is the central
$2$-cocycle of (\ref{n:examcoc1}). It is immediate from the product
formula that
\begin{itemize}

\item[(i)] $L \oplus D^{\gr *}$ is an ideal of $E$, and the canonical projection
$L \oplus D^{\gr *} \to L$ is a central extension.

\item[(ii)] The Lie algebra $D^{\gr *} \oplus D$ of (\ref{n:gencons4})
is a subalgebra of $E$.
\end{itemize}

The Lie algebra $E$ has a a subalgebra \[H=\frh \oplus D^{0 \,
*}\oplus D^0\] where $\frh =
 \Span_F \{ h_\xi^\la : \xi \in S^\times, \la \in \La\}
    = \Span_F \{ h_\xi^0 : 0\ne \xi \in S\ind \}$.
We embed $S$ into the dual space $\frh^*$, using the evaluation map
of (\ref{n:gencons0}), and extend $\xi \in S\subset \frh^*$ to a
linear form of $H$ by $\xi (D^{0\, *} \oplus D^0) = 0$. We embed
$\La \subset D^{0\, *}$, using the evaluation map of
Prop.~\ref{n:degprop}, and then extend $\la\in \La \subset D^{0\,
*}$ to a linear form of $H$ by putting $\la (\frh \oplus D^{0\, *})
= 0$. Then $H$ is a toral subalgebra of $E$ with root spaces
\begin{equation*}
   E_{\xi \oplus \la} = \begin{cases}
                 L_\xi^\la,  & \xi \ne 0, \\
                L_0^\la \oplus (D^{-\la})^* \oplus D^\la,
                           & \xi=0. \end{cases}
\end{equation*}
Observe $H=E_0$ since $\frh=L_0^0$ by Ex.~\ref{n:ex:lietor}. The
symmetric bilinear form $\inpr$ on $E$, defined by
\[ \big( l_1 \oplus c_1 \oplus d_1 \mid l_2 \oplus c_2 \oplus d_2\big)
  = (l_1 \mid l_2)_L + c_1(d_2) + c_2(d_1) \]
where $\inpr_L$ is the given bilinear form of the invariant Lie
torus $L$, is nondegenerate and invariant. With respect to this
bilinear form the set of roots of $(E,H)$ is $R=R^0 \cup R\an$,
where
\begin{align*}
   R^0 &= \{ \la \in \La \subset H^* : L_0^\la \ne 0\} \quad\hbox{and}\\
   R\an &= \{ \xi \oplus \la : \xi \ne 0 \hbox{ and }
                    L_\xi^\la \ne 0\}.
\end{align*}
We have now indicated that the axioms (EA1) and (EA2) of an extended
affine Lie algebra holds for the pair $(E,H)$. The verification of
the remaining axioms can be easily be done by the reader, or can be
looked up in \cite[Prop.~5.2.4]{naoi}. This then shows part (a) of
the following theorem.

\begin{theorem}[{\cite[Th.~6]{n:eala}}] {\rm (a)} The pair
$(E,H)$ constructed above is an extended affine Lie algebra, denoted
$\rmE=\rmE(L,D,\ta)$. Its core is $L \oplus D^{\gr\, *}$ and its
centreless core is $L$. \sm

{\rm (b)} Conversely, let $(E,H)$ be an extended affine Lie algebra,
and let $L=E_c/Z(E_c)$ be its centreless core, which by {\rm
Cor.~\ref{n:ccore}} is an invariant Lie torus, say of type
$(S,\La)$, with $\La$ free of finite rank.

Then there exists a subalgebra $D\subset \SCDer_F(L)$ and an abelian
$2$-cocycle $\ta$ satisfying the conditions {\rm
(\ref{n:gencons0})--(\ref{n:gencons2})} on $(D,\ta)$ such that
$E\cong \rmE(L,D,\ta)$.
\end{theorem}

We have defined discrete EALAs in \ref{n:sec:eala-def} as a special
class of EALAs over the base field $F=\CC$. They can now be
characterized as follows.

\begin{corollary}[{\cite[Th.~8]{n:eala}}] Let $F=\CC$. {\rm (a)} Let
$L$ be an invariant Lie torus of type $(S,\La)$ with $\La$ free of
finite rank and let $D\subset \SCDer_\CC(L)$ be a graded subalgebra
such that the evaluation map $\ev : \La \to D^{0\, *}$ is injective
with discrete image. Then, for any affine $2$-cocycle $\ta$ the
extended affine Lie algebra $\rmE(L,D,\ta)$ is a discrete EALA.
Conversely, any discrete EALA arises in this way.
\end{corollary}



%
%


\newcommand\cprime{$'$}

\end{document}